\documentclass[11pt,reqno]{amsart}
\usepackage{amsfonts,amsmath,amssymb,amscd,amsthm}
\usepackage{hyperref, bookmark}
\usepackage[usenames,dvipsnames]{xcolor}
\usepackage{graphicx}
\usepackage{enumerate}
\usepackage[all]{xy}
\usepackage{mathtools}
\usepackage{mathrsfs}
\usepackage{xparse}
\usepackage{xpatch}
\usepackage{booktabs}

\usepackage{tikz}
\usetikzlibrary{arrows,backgrounds}

\hypersetup{
    pdftitle={ Fourier coefficients and small automorphic representations },         %
    pdfauthor={ Dmitry Gourevitch, Henrik P. A. Gustafsson, Axel Kleinschmidt, Daniel Persson, and Siddhartha Sahi },      %
    linktocpage=true,           
    colorlinks=true,            
    linkcolor=blue!75!black,    
    citecolor=green!50!black,   
    filecolor=magenta,          
    urlcolor=blue!75!black      
}

\let\oldtocsubsection=\tocsubsection
\let\oldtocsubsubsection=\tocsubsubsection
\renewcommand{\tocsubsection}[2]{\hspace{1em}\oldtocsubsection{#1}{#2}}
\renewcommand{\tocsubsubsection}[2]{\hspace{2em}\oldtocsubsubsection{#1}{#2}}

\usepackage{etoolbox}
\makeatletter
\patchcmd{\@startsection}
  {\@afterindenttrue}
  {\@afterindentfalse}
  {}{}
\makeatother

\newtheorem{maintheorem}{Theorem}
\renewcommand{\themaintheorem}{\Alph{maintheorem}}
\newtheorem{maincor}[maintheorem]{Corollary}

\xapptocmd{\maintheorem}{\pdfbookmark[subsubsection]{Theorem~\themaintheorem}{Theorem\themaintheorem}}{}{}
\xapptocmd{\maincor}{\pdfbookmark[subsubsection]{Corollary~\themaintheorem}{Corollary\themaintheorem}}{}{}

\theoremstyle{plain}

\newtheorem*{theorem*}{Theorem}
\newtheorem{lemma}{Lemma}[subsection]
\newtheorem{proposition}[lemma]{Proposition}
\newtheorem{theorem}[lemma]{Theorem}
\newtheorem{corollary}[lemma]{Corollary}

\newtheorem{prop}[lemma]{Proposition}
\newtheorem{lem}[lemma]{Lemma}
\newtheorem{cor}[lemma]{Corollary}

\theoremstyle{definition}

\newtheorem{definition}[lemma]{Definition}
\newtheorem{defn}[lemma]{Definition}

\newtheorem{notn}[lemma]{Notation}

\newtheorem{example}[lemma]{Example}

\newtheorem*{example*}{Example}

\theoremstyle{definition}

\newtheorem*{remark*}{Remark}

\newtheorem{remark}[lemma]{Remark}

\newcommand{\mf}[1]{\mathfrak{#1}}

 \theoremstyle{plain}

\newcommand{\WO}{\operatorname{WO}}

\newcommand{\fX}{\mathfrak{X}}
\newcommand{\bs}{\backslash}

\newcommand{\A}{\mathbb{A}}

\newcommand{\eps}{\varepsilon}

\newcommand{\Q}{{\mathbb Q}}
\newcommand{\R}{{\mathbb R}}

\newcommand{\C}{{\mathbb C}}
\newcommand{\K}{{\mathbb K}}

\newcommand{\proofend}{\hfill$\Box$\smallskip}

\newcommand{\Exp}{\operatorname{Exp}}

\newcommand{\Id}{\operatorname{Id}}

\newcommand{\alp}{{\alpha}}

\newcommand{\lam}{{\lambda}}

\newcommand{\bfG}{{\mathbf{G}}}

\newcommand{\cA}{{\mathcal{A}}}

\newcommand{\cF}{{\mathcal{F}}}

\newcommand{\g}{{\mathfrak{g}}}
\newcommand{\fa}{{\mathfrak{a}}}

\newcommand{\fg}{{\mathfrak{g}}}
\newcommand{\fh}{{\mathfrak{h}}}
\newcommand{\fn}{{\mathfrak{n}}}

\newcommand{\fu}{{\mathfrak{u}}}
\newcommand{\fv}{{\mathfrak{v}}}

\newcommand{\cO}{{\mathcal{O}}}
\newcommand{\GL}{\operatorname{GL}}
\newcommand{\SL}{\operatorname{SL}}
\newcommand{\SO}{\operatorname{SO}}
\newcommand{\Spin}{\operatorname{Spin}}

\newcommand{\ad}{\operatorname{ad}}
\newcommand{\Ad}{\operatorname{Ad}}

\newcommand{\cM}{\mathcal{M}}

\newcommand{\Chi}{\mathfrak{X}}

\newcommand{\fl}{\mathfrak{l}}
\newcommand{\fm}{\mathfrak{m}}

\newcommand{\fz}{\mathfrak{z}}

\newcommand{\cW}{\mathcal{W}}

\newcommand{\onto}{{\twoheadrightarrow}}

\newcommand{\into}{{\hookrightarrow}}

\newcommand{\ints}{\mathbb{Z}}
\newcommand{\reals}{\mathbb{R}}

\newcommand{\WS}{\operatorname{WS}}

\renewcommand{\sl}{\mathfrak{sl}}

\newcommand{\lie}[1]{\mathfrak{#1}}

\NewDocumentCommand{\intl}{g}{
    \IfNoValueTF{#1}{\int\limits}{\int\limits_{\mathclap{#1}}}
}
\NewDocumentCommand{\suml}{g}{
    \IfNoValueTF{#1}{\sum\limits}{\sum\limits_{\mathclap{#1}}}
}

\newcommand{\Oh}{\mathcal{O}}

\numberwithin{equation}{section}

\begin{document}

\author[D. Gourevitch]{Dmitry Gourevitch}
\address{Dmitry Gourevitch,
Faculty of Mathematics and Computer Science,
Weizmann Institute of Science,
POB 26, Rehovot 76100, Israel }
\email{dmitry.gourevitch@weizmann.ac.il}
\urladdr{\url{http://www.wisdom.weizmann.ac.il/~dimagur}}

\author[H. Gustafsson]{Henrik P. A. Gustafsson}
\address{Henrik Gustafsson, Department of Mathematics, Stanford University, Stanford, CA 94305}
\email{gustafsson@stanford.edu}
\urladdr{\url{http://hgustafsson.se}}

\author[A. Kleinschmidt]{Axel Kleinschmidt}
\address{Axel Kleinschmidt,  {\it Max-Planck-Institut f\"{u}r Gravitationsphysik (Albert-Einstein-Institut)},
Am M\"{u}hlenberg 1, DE-14476 Potsdam, Germany, and 
{\it International Solvay Institutes}, 
ULB-Campus Plaine CP231, BE-1050, Brussels, Belgium}
\email{axel.kleinschmidt@aei.mpg.de}

\author[D. Persson]{Daniel Persson} 
\address{Daniel Persson, Chalmers University of Technology, Department of Mathematical Sciences\\
SE-412\,96 Gothenburg, Sweden}
\email{daniel.persson@chalmers.se}

\author[S. Sahi]{Siddhartha Sahi}
\address{Siddhartha Sahi, Department of Mathematics, Rutgers University, Hill Center -
Busch Campus, 110 Frelinghuysen Road Piscataway, NJ 08854-8019, USA}
\email{sahi@math.rugers.edu}



\date{\today}
\title[Fourier coefficients of minimal and next-to-minimal representations]{Fourier coefficients of minimal and next-to-minimal automorphic representations of simply-laced groups}
\maketitle
\begin{abstract}
In this paper we analyze Fourier coefficients of automorphic forms on a finite cover $G$ of an adelic split simply-laced group. Let $\pi$ be  a minimal or next-to-minimal automorphic representation of $G$. We prove that any $\eta\in \pi$ is completely determined by its Whittaker coefficients with respect to (possibly degenerate) characters of the unipotent radical of a fixed Borel subgroup, analogously to the Piatetski-Shapiro--Shalika formula for cusp forms on $\GL_n$. We also derive explicit formulas expressing the form, as well as all its maximal parabolic Fourier coefficient in terms of these Whittaker coefficients. A consequence of our results is the non-existence of cusp forms in the minimal and next-to-minimal automorphic spectrum. We provide detailed examples for $G$ of type $D_5$ and $E_8$ with a view towards
applications to scattering amplitudes in string theory.
\end{abstract}

\pagebreak 
\setcounter{tocdepth}{2}
\tableofcontents

\section{Introduction and main results}
\subsection{Introduction}
\label{intro}

Let $\K$ be a number field and $\A=\A_{\K}=\Pi'\K_{\nu}$ its ring of adeles. 
Let ${\bf G}$ be a  reductive group defined over $\K$, ${\bf G}(\mathbb{A})$ the group of adelic points of ${\bf G}$ and $G$ be a finite central extension of ${\bf G}(\A)$.  
We assume that there exists a section ${\bf G}(\K)\to G$ of the covering $G\onto {\bf G}(\A)$,  fix such a section and denote its image by $\Gamma$.
This generality includes the covering groups defined in \cite{BryDel}. By \cite[Appendix I]{MWCov}, the covering $G\onto {\bf G}(\A)$ canonically splits over unipotent subgroups, and thus we will consider unipotent subgroups of ${\bf G}(\A)$ as subgroups of G.
 
Let $\eta$ be an automorphic form on $G$. Let $U$ be a unipotent subgroup of $G$ and $\chi_U$ be a unitary character of $U$ that is trivial on $U\cap \Gamma$. We define the \textit{Fourier coefficient} of $\eta$ associated with $U$ and $\chi_U$ as
\begin{equation}
\label{eq:FCU}
\mathcal{F}_{\chi_U}[\eta](g):=\int_{[U]} \eta(ug)\chi_U(u)^{-1} du\,,
\end{equation}
where $[U]:= (U\cap\Gamma)\backslash U$ denotes the compact quotient of $U$. Well-studied special cases of this definition arise when $U$ is the unipotent $N$ of a Borel subgroup and in that case the Fourier coefficients are called \textit{Whittaker coefficients}, see~\eqref{=cW} below. Another common case is when $U$ is the unipotent of a (non-minimal) parabolic subgroup $P=LU\subset G$ and we shall refer to~\eqref{eq:FCU} in that case as a \textit{parabolic Fourier coefficient}.

Generally, when $U$ is non-abelian, the coefficient $\mathcal{F}_{\chi_U}$ only captures a part of the Fourier expansion of $\eta$. To reconstruct $\eta$ from its coefficients one needs to consider a series of subgroups $U_{i_0}=\{1\}\subset U_{i_0-1}\subset \dots \subset U_1=U$ with { successive abelian  quotients}   $U_i/U_{i{+}1}$. Two examples are the derived series of $U$, and the lower central series of $U$. Denote by $\Chi_i$ the set all non-trivial unitary characters of $U_i$ that are trivial on $U_{i{+}1}$ and on $U_i\cap \Gamma$.
The complete Fourier expansion of $\eta$ with respect to $U$ takes the form 
\begin{equation}
\label{eq:NAFE}
\eta=\mathcal{F}_0[\eta]+\sum_{\chi\in{\Chi_{1}}}\mathcal{F}_{\chi_{U_{1}}}[\eta]+\sum_{\chi\in{\Chi_{1}}}\mathcal{F}_{\chi_{U_{1}}}[\eta]+\cdots + \sum_{\chi\in{\Chi_{i_0}}}\mathcal{F}_{\chi_{U_{i_0}}}[\eta] \, .
\end{equation}
The simplest case of a non-abelian $U$ is one that admits a Heisenberg structure, {\it i.e.} $[U,U]$ is a one-dimensional group, and this will be an important tool for us when we analyse groups of type $E_8$ that do not admit any abelian unipotents $U$ as radicals of parabolic subgroups.
In this case, the lower central series coincide with the derived series. Namely, we take $i_0=3$ and $U_2$ to be $[U,U]$ and call the Fourier coefficients $\mathcal{F}_{\chi_{U{1}}}[\eta]$ the \textit{abelian} Fourier coefficients and those for $U_{2}$  the \textit{non-abelian} Fourier coefficients.

For the most general Fourier coefficient~\eqref{eq:FCU} and automorphic form $\eta$ not much is known about its reduction theory and explicit formulas. In particular, $\mathcal{F}_{\chi_{U}}$ is non-Eulerian and no analogues of the Casselman--Shalika~\cite{CasselmanShalika} or Piatetski-Shapiro--Shalika formula~\cite{PiatetskiShapiro,Shalika} are known. The problem becomes more tractable when restricting to coefficients given by Whittaker pairs~\cite{GGS,GGS:support}, a technique that we have used in the companion paper~\cite{Part1} for studying the reduction theory.

In this paper we will analyze Fourier coefficients and expansions in the case of special classes of automorphic forms on split, 
simply-laced Lie groups. Specifically we consider automorphic forms $\eta$ attached to  so-called \emph{minimal} or \emph{next-to-minimal} automorphic representations $\pi_{\textnormal{min}}$ and $\pi_{\textnormal{ntm}}$ of the adelic group $G$. This means that all Fourier coefficients  attached to nilpotents outside of a union of Zariski closures of minimal or next-to-minimal nilpotent orbits vanish. We refer to \S \ref{subsec:small} below for the precise definitions.
We note that in type $D$ there are two next-to-minimal complex orbits, while in types A and E the next-to-minimal orbit is unique. Minimal orbits are unique in all simple Lie algebras.
 A sufficient condition for $\pi$ to be minimal or next-to-minimal is that one of its local components is minimal or next-to-minimal, see Lemma \ref{lem:GlobLoc} below. For minimal representations, this condition is also shown to be necessary under some additional assumptions on $G$, see 
 \cite{MR2123125,KobayashiSavin}.

Even though we shall not rely on explicit automorphic realizations of minimal and next-to-minimal representations, it might be instructive to indicate how they can be obtained. 
Minimal representations have been studied extensively in the literature, in particular due to their crucial role in establishing functoriality in the form of theta correspondences and~\cite{GRS2} discusses them as residues of degenerate principal series. Moreover, in a series of works \cite{GRS,GRS2,Ginz,MR3161096}, $\pi_{\textnormal{min}}$ was used to construct global Eulerian integrals. Next-to-minimal representations have not been analyzed as extensively though in recent years this has started to change, partly due to their importance  in understanding scattering amplitudes in string theory \cite{GMV,P,FKP2013,GKP,FGKP}; see \S\ref{sec:string} below for more details on this connection. Next-to-minimal representations
exist for all next-to-minimal orbits, see {\it e.g.}  \S\ref{sec:examples}, below and ~\cite{FGKP}. They can be obtained as different residues of degenerate principal series, see \cite{GMV,P} for type $E$. In types $A$, $E_6$, and for one of the orbits in type $D$ there are one-parameter families of next-to-minimal representations.

In \cite{GGS,GGS:support} it was shown that there exist $G$-equivariant epimorphisms between different spaces of Fourier coefficients, thus determining their vanishing properties in terms of nilpotent orbits. In \cite{Part1} we determined exact relations (instead of only showing the existence of such) between different types of Fourier coefficients. In this paper we apply the techniques of \cite{Part1}, and  reduce maximal parabolic Fourier coefficients that are difficult to compute into more manageable class of coefficients such as the known Whittaker coefficients with respect to the unipotent radical of a Borel subgroup.
Furthermore, we express minimal and next-to-minimal automorphic forms through their Whittaker coefficients.

In the next subsection we discuss the class of Fourier coefficients studied in \cite{GGS,GGS:support,Part1}. This class includes parabolic  coefficients,  coefficients of lower central series (but not the derived series) for unipotent radicals of parabolics, and the coefficients considered in \cite{GRS,Ginz,MR3161096,JLS}.

\subsection{Fourier coefficients associated to Whittaker pairs}
\label{subsec:FCWP}

Assume throughout this paper that ${\bf G}$ is a split simply-laced reductive group defined over $\K$. 
In order to explain our main results in more detail, we briefly introduce some terminology. Denote by $\fg$  the Lie algebra of {${\bf G}(\K)$}.
A \emph{Whittaker pair} is an ordered pair $(S,\varphi)\in \fg\times \fg^*$, where $S$ is a semi-simple element with eigenvalues of $\ad(S)$ in $\Q$, and $\ad^*(S)(\varphi)=-2\varphi$. 
This implies that $\varphi$ is necessarily nilpotent and corresponds to a unique nilpotent element $f=f_{\varphi}\in \fg$ by the Killing form pairing.
Each Whittaker pair $(S,\varphi)$ defines a unipotent subgroup $N_{S,\varphi}\subset G$ given by \eqref{=Nsphi} below and a unitary character $\chi_\varphi$ on $N_{S,\varphi}$ by $\chi_\varphi(n) = \chi(\varphi(\log n))$ for $n \in N_{S,\varphi}$. 

Our results are applicable to a wide space of functions on $G$, that we denote by $C^{\infty}(\Gamma\backslash G)$ and call the space of automorphic functions. This space consists of functions $f$ that are left $\Gamma$-invariant, finite under the right action of the preimage in $G$ of $\prod_{\text{finite }\nu}{\bf G}(\cO_\nu),$ and smooth when restricted to the preimage in $G$ of $\prod_{\text{infinite }\nu}{\bf G}(\K_\nu)$. In other words, we remove the usual requirements  of moderate growth and finiteness under the center $\fz$ of the universal enveloping algebra. Such cases arise in applications in string theory~\cite{Green:2005ba,DHoker:2015gmr,FGKP}.

Following \cite{MW,GRS2,GRS,GGS} we attach to each Whittaker pair $(S,\varphi)$ and automorphic function $\eta$ on $G$ the following Fourier coefficient 
\begin{equation} 
\mathcal{F}_{S,\varphi}[\eta](g)=\intl_{[N_{S,\varphi}]} \eta(ng)\, {\chi_\varphi(n)}^{-1}\, dn.
\end{equation}

We note that the integrals we consider in this paper are well-defined for automorphic functions as they are either compact integrals or represent Fourier expansions of periodic functions.

\begin{remark}
Note  that the unipotent group $N_{S,\varphi}$ is not necessarily the unipotent radical of a parabolic subgroup of $G$. Consider, for example, the case of $G=E_8$ and let $P=LU\subset E_8$ be the Heisenberg parabolic such that the Levi is $L=E_7\times \GL_1$ and the unipotent radical $U$ is the $57$-dimensional Heisenberg group with one-dimensional center $C=[U,U]$. Then the Fourier coefficient $\mathcal{F}_{S,\varphi}$  can include the ``non-abelian'' coefficient corresponding to $N_{S,\varphi}=C$ and $\chi_\varphi$ a non-trivial character on $C$. This case is relevant for applications to physics; see \S\ref{sec:string} below. 
\end{remark}

If a Whittaker pair $(h, \varphi)$ corresponds to a Jacobson--Morozov $\mathfrak{sl}_2$-triple $(e, h, f_\varphi)$ we say that it is a \emph{neutral} Whittaker pair, and call the corresponding coefficient a \emph{neutral Fourier coefficient}. This is the class studied in \cite{GRS,Ginz,MR3161096,JLS} and referred to simply as  a Fourier coefficient. 

We denote by $\operatorname{WO}(\eta)$  the set of nilpotent orbits $\mathcal{O}$ such that there exists a neutral pair $(h, \varphi)$ such that  $\mathcal{F}_{h,\varphi}[\eta]\not\equiv0$ and $\varphi \in \Oh$, see Definition \ref{def:WO} below. It was shown in \cite[Theorem C]{GGS} that if $\mathcal{F}_{h,\varphi}[\eta] = 0$ then $\mathcal{F}_{S,\varphi}[\eta] = 0$ for any Whittaker pair $(S, \varphi)$, not necessarily neutral. We denote the set of maximal elements in $\WO(\eta)$ by $\WS(\eta)$ and call it the \emph{Whittaker support} of $\eta$. We refer to automorphic functions $\eta_{\textnormal{min}}$ whose Whittaker support consists of the minimal nilpotent orbit as minimal automorphic functions and write likewise $\eta_{\textnormal{ntm}}$ for next-to-minimal automorphic functions.

\subsection{Statement of Theorem \ref{thm:min-rep}}\label{subsec:ThA}
Choose a $\K$-split maximal torus $T\subset G$ and a set of positive roots.
Let $\fh$ be the Lie algebra of $T\cap \Gamma$.
 For a simple root  $\alpha$ we denote by $P_{\alp}$ the corresponding maximal parabolic subgroup, by $L_{\alp}$ is standard Levi subgroup, and by $U_{\alp}$ its unipotent radical. In other words, 
 $\fu_{\alp}:=\operatorname{Lie} U_\alpha$ is spanned by the root spaces whose expression in terms of simple roots contains $\alp$ with positive coefficient. 
Define $S_{\alp}\in \fh$ by 
\begin{equation}
\alp(S_\alpha) = 2 \text{ and }\beta(S_\alpha) = 0 \text{ for all other simple roots } \beta.
\end{equation}
It will follow from the definition of $N_{S,\varphi}$ that for any $\varphi\in \fg^*$ such that $\ad^*(S_{\alp})\varphi=-2\varphi$, we have
that $N_{S_\alpha,\varphi} = U_\alpha$. 
This means that the Fourier coefficient $\cF_{S_\alpha, \varphi}$ is the {parabolic} Fourier coefficient with respect to the unipotent subgroup $U_\alpha$ and the character $\chi_\varphi$.
{Let $S_{\Pi}:=\sum_{\alp\in \Pi}S_{\alp},$ where $\Pi$ is the set of all simple roots. {Then the associated unipotent subgroup is the radical $N$ of the Borel subgroup defined by the choice of simple roots.} For any $\varphi\in \fg^*$ with $\ad^*(S_{\alp})\varphi=-2\varphi$, and any automorphic function $\eta$ {define the Whittaker coefficient by}
\begin{equation}\label{=cW}
\cW_{\varphi}[\eta]:=\cF_{S_{\Pi},\varphi}[\eta]\,.
\end{equation}
}

\begin{maintheorem}\label{thm:min-rep}
    Let $\eta_\textnormal{min}$ be a minimal automorphic function on 
    a simply-laced split group $G$ and $(S_\alpha, \varphi)$ a Whittaker pair with $S_\alpha$ determined by a simple root $\alpha$ as above. Depending on the orbit of $\varphi$, we have the following statements for the corresponding Fourier coefficient. 
    \begin{enumerate}[(i)]
    \item \label{it:min0phi}
   {The restriction of $\mathcal{F}_{S_\alpha, 0}[\eta_\textnormal{min}]$ to the Levi subgroup $L_{\alp}$ is a minimal {or a trivial} automorphic function.}
           \item \label{it:min-min} If $\varphi$ is minimal, then there exists $\gamma_0\in \Gamma\cap L_\alpha$ that conjugates 
            $\varphi$ to an element $\varphi'$ of weight ${-\alpha}$ by $\Ad^*({\gamma_0}) \varphi =\varphi'$ and for any such $\gamma_0$ we have
            \begin{equation}
            \label{eq:FWmin}
                \mathcal{F}_{S_\alpha, \varphi}[\eta_\textnormal{min}](g) = \mathcal{W}_{\varphi'}[\eta_\textnormal{min}](\gamma_0 g) \,.
            \end{equation}
        \item \label{it:min0} If $\varphi$ is not minimal and not zero then $\mathcal{F}_{S_\alpha, \varphi}[\eta_\textnormal{min}] = 0$.
    \end{enumerate}
\end{maintheorem} 

{For part \eqref{it:min0phi} we remark that $\mathcal{F}_{S_\alpha, \varphi}[\eta_\textnormal{min}]$ is the usual constant term in a maximal parabolic. For Eisenstein series it can be computed using the results of~\cite{MWCov}.
It can also be expressed through Whittaker coefficients using Theorem \ref{thm:G0min} below.}

\begin{remark}
We note that the formula (\ref{eq:FWmin}) is compatible with the expected equivariance of the Fourier coefficient $\mathcal{F}_{S_\alpha, \varphi}[\eta_\textnormal{min}](g)$, i.e. it satisfies 
\begin{equation}
\mathcal{F}_{S_\alpha, \varphi}[\eta_\textnormal{min}](ug)=\chi_\varphi(u)\mathcal{F}_{S_\alpha, \varphi}[\eta_\textnormal{min}](g),
\end{equation} 
for all $u\in U$. For this to hold one requires that $\gamma_0^{-1} u\gamma_0\in N$ for all $u\in U$ and 
\begin{equation}
\chi_{\varphi}(u)=\chi_{\varphi^{\prime}}(\gamma_0^{-1} u \gamma_0),
\end{equation}
which indeed holds due to the fact that $\gamma_0\in \Gamma \cap L_\alpha$.
\end{remark}

\begin{remark}
The notation $\cW_{S,\varphi}$ and $\cW_\varphi$ is used in \cite{GGS,GGS:support} to denote something quite different. It was chosen for the current paper since this is the notation in \cite{FGKP}.
\end{remark}

One can also obtain an expression for the minimal automorphic function itself.
This is the subject of the next subsection.

\subsection{Statement of Theorem~\ref{thm:G0min}}
\label{subsec:ThB}
For any root $\eps$ denote by $$\fg^*_{\eps}{:=\{\omega\in\fg^* \,|\, \ad^*(h)\omega = \eps(h) \omega\,\,\text{for all $h\in\mathfrak{h}$}\}}$$ the corresponding subspace of $\fg^*$ and by $\fg^{\times}_{\eps}$ the set of non-zero elements of this subspace. Note that $\fg^*_{\eps}$ is a one dimensional linear space over $\K$. 
We say that a simple root  $\alpha$ is an abelian or a Heisenberg simple root in $\lie g$ if $\lie u_\alpha$ is an abelian or Heisenberg Lie algebra, respectively, or, equivalently, if $[\lie u_\alpha, \lie u_\alpha]$ has dimension zero or one. 
If $\alpha$ is either abelian or Heisenberg we call it \emph{quasi-abelian}.
The classification of such roots reduces to simple components of $\lie g$, where we have the following explicit answer in terms of Bourbaki numbering.

\begin{table}[h]
\centering
\caption{\label{tab:QA} Quasi-abelian roots}
\begin{tabular}{|c|c|c|c|c|c|}
\hline
& $A_n$ & $D_n$ & $E_6$ & $E_7$ & $E_8$ \\ \hline
\text{abelian} & \text{all} & $\alpha_1, \alpha_{n-1}, \alpha_n$ & $\alpha_1, \alpha_6$ & $\alpha_7$ & - \\ \hline
\text{Heisenberg} & - & $\alpha_2$ & $\alpha_2$ & $\alpha_1$ & $\alpha_8$ \\ \hline
\end{tabular}
\end{table}

{To derive Table \ref{tab:QA} we note that the abelian roots are those that appear with coefficient one in the highest root. The Heisenberg roots are determined in \cite[Lemma 5.1.2]{Part1}. There are no such roots in type $A_n$, while in types $D_n$ or $E_n$ this is the unique root that connects to the affine node in the affine Dynkin diagram.}

{By \cite[\S VIII.3]{Bou} the abelian roots are precisely those that can be conjugated to the affine node by an automorphism of the affine Dynkin diagram.}

Let $I = (\beta_1, \ldots, \beta_n)$ be an enumeration of the simple roots of $\lie g$ in some order, and let $\lie l_i$ be the Levi subalgebra with simple roots $\{\beta_1, \ldots, \beta_i\}$. We will say that $I$ is \emph{abelian} if each $\beta_i$ is abelian in $\lie l_i$, and that $I$ is \emph{quasi-abelian} if each $\beta_i$ is {quasi-abelian}  in $\lie l_i$. 
From the table we see that the Bourbaki enumeration is quasi-abelian if $\lie g = E_8$  and abelian if $\fg$ is simple (simply-laced) and different from $E_8$. We also note that $\mf{l}_i\subset \mf{l}_j$ for $i<j$.

Let $I= (\beta_1, \ldots, \beta_n)$ be any quasi-abelian enumeration of the simple roots of $\fg$. Given an automorphic function $\eta$ on $\Gamma\backslash G$ we define functions $A_i[\eta]$, $B_i[\eta]$ and $C_i[\eta]$ on $G$ as follows.

Let $L_{i-1}$ be the Levi subgroup of $G$ with Lie algebra $\lie l_{i-1}$, and {let $Q_{i-1}$ be the parabolic subgroup of $L_{i-1}$ with Lie algebra $(\fl_{i-1})^{\beta_i^{\vee}}_{\leq 0}$. In Lemma  \ref{lem:LineStab} below we show that $Q_{i-1}$ is the stabilizer in $L_{i-1}$ of the root space $\lie g^*_{-\beta_i}$, as an element of the projective space of $\fl_i^*$.
 We} let $\Gamma_{i-1} = (L_{i-1}\cap \Gamma) / ({Q}_{i-1}\cap \Gamma)$ , and put {for $i\in \{1,\ldots, n\}$
\begin{equation}\label{=Ai}
    A_i[\eta](g) := \sum_{\gamma \in \Gamma_{i-1}} \sum_{\varphi \in \lie g^\times_{-\beta_i}} \mathcal{W}_{\varphi}[\eta](\gamma g) \, ,
\end{equation}
where $\Gamma_0=\{1\}$.
{\begin{remark}Note that although $\gamma$ is a coset, the inner sum $\sum_{\varphi \in \lie g^\times_{-\beta_i}} \mathcal{W}_{\varphi}[\eta](\gamma g)$ is independent of the choice of a representative for $\gamma$, since ${Q}_{i-1}\cap \Gamma$ stabilizes $\lie g^\times_{-\beta_i}$. Thus $ A_i[\eta]$ is well-defined. We will use similar summations over cosets in the future without further comment.
\end{remark}}

If $\beta_i$ is a Heisenberg root of $\lie l_i$, then we define
\begin{equation}\label{=Omega}
 \mathcal{B}_{\beta_i} := {\{\text{positive roots }\beta{\text{ of }\fl_i} : \langle \beta_i, \beta\rangle =1 \}\,},\quad \quad
        \Omega_i := \Exp(\bigoplus_{\beta \in \mathcal{B}_{\beta_i}} \lie g_{-\beta})\,.
 \end{equation}
 
{Note that $\Omega_i$ is a commutative subgroup of $\Gamma$. 
Denote by  $\alpha_\text{max}^i$  the highest root for the simple component of $\lie l_i$ containing $\beta_i$, and let $s_{\beta_i}$ and $s_{\alp_{\max}^i}$ denote the reflections with respect to the roots ${\beta_i}$ and ${\alp_{\max}^i}$. Then $s_{\beta_i}s_{\alp_{\max}^i}s_{\beta_i}$ is an involutive Weyl group element that switches $\beta_i$ and $\alp_{\max}^i$.
We fix  a representative $\gamma_{i}\in \Gamma$ for $s_{\beta_i}s_{\alp_{\max}^i}s_{\beta_i}$ and define
}

\begin{equation}\label{=Bi}
    B_i[\eta](g) := \sum_{\omega \in \Omega_i} \sum_{\varphi \in \lie g^\times_{-\beta_i}} \mathcal{W}_{\varphi}[\eta](\omega{\gamma_{i}} g) \, .
\end{equation}

Finally, we define
\begin{equation}
    C_i[\eta] :=
    \begin{cases}
        A_i[\eta] & \text{if } \beta_i \text{ is abelian} \\
        A_i[\eta] + B_i[\eta] & \text{if } \beta_i \text{ is Heisenberg.} \\
    \end{cases}
\end{equation}

\begin{maintheorem}\label{thm:G0min}
Let $\eta_\textnormal{min}$ be a minimal automorphic function on 
     $G$. Then, for any choice of a quasi-abelian enumeration we have
\begin{equation}
    \eta_\mathrm{min} = \mathcal{W}_0[\eta_\mathrm{min}] + \sum_{i = 1}^n C_i[\eta_\mathrm{min}]\,.
\end{equation}
\end{maintheorem}

\begin{example}
    \label{ex:ThmB} 
    Let $G = \SO_{4,4}(\A)$ with $\Gamma = \SO_{4,4}(\K)$ and $\eta_\text{min}$ a minimal automorphic function on $G$. We take the quasi-abelian enumeration $I = (\beta_1, \beta_2, \beta_3, \beta_4) = (\alpha_1, \alpha_3, \alpha_4, \alpha_2)$ where $\alpha_i$ are given by the Bourbaki labeling. Note that $\beta_4 = \alpha_2$ is a Heisenberg root in $G$, while $\beta_i$ for $1 \leq i \leq 3$ is an abelian root for the Levi subgroup $L_i$ with simple roots $\beta_1, \ldots, \beta_i$. Using Theorem~\ref{thm:G0min} we get that
    \begin{equation}
        \label{eq:D4-min-example}
        \begin{split}
            \eta_\text{min}(g) &= \cW[\eta_\text{min}](g) + B_4[\eta_\text{min}](g) + \sum_{i=1}^4 A_i[\eta_\text{min}](g) \\
                               &= \cW[\eta_\text{min}](g) + \sum_{\omega \in \Omega_4} \sum_{\varphi \in \lie g^\times_{-\beta_4}} \!\!\! \cW_\varphi[\eta_\text{min}](\omega \gamma_4 g) + \sum_{i=1}^4 \sum_{\gamma \in \Gamma_{i-1}} \sum_{\varphi \in \lie g^\times_{-\beta_i}} \!\!\! \cW[\eta_\text{min}](\gamma g)\, ,
        \end{split}
    \end{equation}
    where $\Omega_4$ is defined in \eqref{=Omega}, $\gamma_4$ is defined above \eqref{=Bi}, and $\Gamma_{i-1}$ above \eqref{=Ai}. For this example we get that the Lie algebra of $\Omega_4$ is $\lie g_{-\alp_2-\alp_1} \oplus \lie g_{-\alp_2-\alp_3} \oplus \lie g_{-\alp_2-\alp_4} \oplus \lie g_{-2\alp_2-\alp_1-\alp_3-\alp_4}$, $\gamma_4$ is a representative of the Weyl word $s_1 s_3 s_2 s_4 s_2 s_1 s_3$ in $\Gamma$ with simple reflections $s_i$, and $\Gamma_0 = \Gamma_1 = \Gamma_2 = \{1\}$ while $\Gamma_3 \cong (\mathbb{P}^1(\K))^3$.

    We picked this example to demonstrate the Heisenberg term $B_4$ and because the right-hand side of \eqref{eq:D4-min-example} is manifestly triality invariant. 
\end{example}

Let us now formulate analogs of Theorems \ref{thm:min-rep} and \ref{thm:G0min} for next-to-minimal automorphic functions. 

\subsection{{Statement of Theorem \ref{thm:ntm-rep}}}

{As before, l}et $\alpha$ be a simple root of $\lie g$, and let $(S_\alpha, {\psi})$ be a Whittaker pair such that $\psi \in \lie g^\times_{-\alpha}$ and $S_\alpha$ defines the maximal parabolic subgroup corresponding to $\alpha$.  
Let $I^{(\perp\alpha)} = (\beta_1, \ldots, \beta_{m} )$ be a quasi-abelian enumeration of the simple roots orthogonal to $\alpha$ which is always possible to find{, see Table \ref{tab:QA}}. For any $1\leq i \leq {m}$, we also define $\Gamma_{i-1}$ {and $\gamma_i$} as above, but with the enumeration $I^{(\perp\alpha)}$, and given an automorphic function $\eta$ on $\Gamma \backslash G$ we set
\begin{equation}
    A_i^\psi[\eta](g) = \sum_{\gamma \in \Gamma_{i-1}} \sum_{\varphi \in \lie g^\times_{-\beta_i}} \mathcal{W}_{\psi + \varphi}[\eta](\gamma g) \, .
\end{equation}
For any $1\leq i \leq  {m}$ with $\beta_i$ a Heisenberg root in the Levi subalgebra given by $\beta_1,\dots , \beta_i$, we furthermore  set
\begin{equation}
    B_i^\psi[\eta](g) = \sum_{\omega \in \Omega_i}\sum_{\varphi\in \fg^{\times}_{\beta_i}}\cW_{\psi+\varphi}[\eta](\omega{\gamma_i} g).
\end{equation}

Finally, we define
\begin{equation}
    C^{\psi}_i[\eta] =
    \begin{cases}
        A_i^{\psi}[\eta] & \text{if } \beta_i \text{ is abelian} \\
        A_i^{\psi}[\eta] + B_i^{\psi}[\eta] & \text{if } \beta_i \text{ is Heisenberg.} \\
    \end{cases}
\end{equation}

Furthermore, let $\overline{\lie b}$ be the Lie algebra of the negative Borel spanned by $\lie h$ and the root spaces of negative roots. For an element $\gamma \in \Gamma$, we define 
\begin{equation}
    \label{eq:V}
\lie v_{\gamma} := \lie g^{\gamma S_{\alp} \gamma^{-1}}_{>1}  \cap \overline{\lie b} \quad\text{and}\quad    V_{\gamma} := \Exp(\lie v_{\gamma}(\A)) \, .
\end{equation}

 \begin{remark}
 {Since $\Gamma_i$ is a partial flag variety for $L_{i}$, it coincides with the group of $\K$-points of the corresponding projective algebraic  variety. By the valuation criterion for properness (\cite[Ch. II, Theorem 4.7]{Hartshorne}), it then coincides with the (integral) $O_\K$-points of the same variety. }
 \end{remark}
 
 \begin{maintheorem}\label{thm:ntm-rep}
    Let $\eta_\textnormal{ntm}$ be a next-to-minimal automorphic function on
 $G$, let $(S_\alpha, \varphi)$ be a Whittaker pair with $S_\alpha$ as above and $I^{(\perp\alpha)} = (\beta_1, \ldots, \beta_m )$ a quasi-abelian enumeration as above. Depending on the orbit of $\varphi$, we have the following statements for the corresponding Fourier coefficient.     \begin{enumerate}[(i)]
    \item \label{it:ntm0phi} {For trivial $\varphi=0$ the restriction of $\mathcal{F}_{S_\alpha, 0}[\eta_\textnormal{ntm}]$ to the Levi subgroup $L_{\alp}$ is a {trivial, or }minimal, or next-to-minimal automorphic function.}
    \item 
    \label{it:Pt2}
            For $\varphi$ in the minimal orbit there exists $\gamma_0 \in L_\alpha \cap \Gamma$ such that ${\psi:=}\Ad^*(\gamma_0)\varphi \in \lie g^\times_{-\alpha}$. {For any such $\gamma_0 \in L_\alpha \cap \Gamma$} , we have
            \begin{equation}
                \cF_{S_\alpha, \varphi}[\eta_\textnormal{ntm}](g) = \mathcal{W}_{{\psi}}[\eta_\textnormal{ntm}](\gamma_0 g) + \sum_{i=1}^m C_i^{{\psi}}[\eta_\textnormal{ntm}](\gamma_0 g) \, .
            \end{equation}

        \item \label{itm:ntm} If $\varphi$ is next-to-minimal, then there exist  orthogonal simple roots $\alpha'$ and $\alpha''$, and an element $\gamma_0 \in \Gamma$ that is a product of an element of $L_{\alp}\cap \Gamma$ and a Weyl group representative, such that ${\psi:=}\Ad^*(\gamma_0) \varphi \in \lie g^{\times}_{-\alpha'} + \lie g^{\times}_{-\alpha''}$.
        {For any such  $\gamma_0$, $\alpha'$ and $\alpha''$, we have}
            \begin{equation}
                \mathcal{F}_{S_{\alp}, \varphi}[\eta_\textnormal{ntm}](g) = \intl_{V_{\gamma_0}} \mathcal{W}_{\psi}[\eta_\textnormal{ntm}](v \gamma_0 g) \, dv \, . 
                \label{eq:ntmint}
            \end{equation}
        \item \label{itm:larger} If $\varphi$ is not in the closure of any complex next-to-minimal orbit, then $\mathcal{F}_{S_{\alp}, \varphi}[\eta_\textnormal{ntm}] = 0$.
    \end{enumerate}
\end{maintheorem}

Colloquially, we will refer to the condition in (\ref{itm:larger}) as $\varphi$ being in an orbit larger than next-to-minimal.

\begin{remark}\hfill
\label{rmk:CC}

\begin{enumerate}[(i)]
\item For Theorem~\ref{thm:ntm-rep}\eqref{it:ntm0phi} we remark that the coefficient $\mathcal{F}_{S_\alpha, {0}}[\eta_\textnormal{ntm}]$ is the usual constant term that can be determined {for Eisenstein series} using the results of~\cite{MWCov}.
We note also that the restriction of $\mathcal{F}_{S_\alpha, 0}[\eta_\textnormal{ntm}]$ to the Levi subgroup $L_{\alp}$ can be expressed through Whittaker coefficients using Theorem \ref{thm:G0min} above and Theorem \ref{thm:ntm-rep2} below.

\item 
 {As will become clear in the proof presented in \S\ref{sec:pfC}, the expression in~\eqref{eq:ntmint} does not depend on the choice of $\gamma_0$. However, different choices of  $\gamma_0$ can lead to different expressions, some of which may be simpler, see for example \eqref{eq:Vij}. }
\label{rmk:indp}

\item 
We stress that, similarly to (\ref{eq:FWmin}), the right-hand side of the formula (\ref{eq:ntmint}) is compatible with the equivariance of the Fourier coefficient, i.e. satisfies 
\begin{equation}
 \mathcal{F}_{S_{\alp}, \varphi}[\eta_\textnormal{ntm}](ug)=\chi_\varphi(u) \mathcal{F}_{S_{\alp}, \varphi}[\eta_\textnormal{ntm}](g) 
 \end{equation}
 for all $u\in U$. However, in contrast to the minimal case, here we need no additional constraint on $\gamma_0$ since the equivariance is automatically ensured by the integration over $V_{\gamma_0}$.
\end{enumerate}
\end{remark}

\begin{example}
    Let $G = \SO_{4,4}(\A)$ with $\Gamma = \SO_{4,4}(\K)$ and $\eta_\text{ntm}$ a next-to-minimal automorphic function on $G$. Let $\alpha = \alpha_1$ and take the abelian enumeration $I^{(\perp \alpha_1)} = (\beta_1, \beta_2) = (\alpha_3, \alpha_4)$. Fix a minimal element $\varphi_\text{min} \in \lie g^\times_{-\alpha_1-\alpha_2}$ and let $\gamma_0^\text{min}$ be a representative of the simple reflection $s_2$ in $\Gamma$ which means that $\psi_\text{min} := \Ad^*(\gamma_0^\text{min}) \varphi_\text{min} \in \lie g^\times_{-\alpha_1}$.

    From Theorem~\ref{thm:ntm-rep}(\ref{it:Pt2}) we get that
    \begin{align}
            \cF_{S_{\alpha_1}, \varphi_\text{min}}[\eta_\text{ntm}](g) 
            &= \cW_{\psi_\text{min}}[\eta_\text{ntm}](\gamma_0^\text{min} g) + 
            \sum_{i=1}^2 A_i^{\psi_\text{min}}[\eta_\text{ntm}](\gamma_0^\text{min} g)\nonumber\\
           & = \cW_{\psi_\text{min}}[\eta_\text{ntm}](\gamma_0^\text{min} g) +
            \sum_{i=1}^2 \sum_{\gamma \in \Gamma_{i-1}} \sum_{\varphi \in \lie g^\times_{-\beta_i}} \mathcal{W}_{\psi_\text{min} + \varphi}[\eta_\text{ntm}](\gamma \gamma_0^\text{min} g)\nonumber\\
            &=
\cW_{\psi_\text{min}}[\eta_\text{ntm}](\gamma_0^\text{min} g) +
            \sum_{\varphi \in \lie g^\times_{-\alp_3}} \mathcal{W}_{\psi_\text{min} + \varphi}[\eta_\text{ntm}]( \gamma_0^\text{min} g)\\
            &\hspace{10mm}+ \sum_{\varphi \in \lie g^\times_{-\alp_4}} \mathcal{W}_{\psi_\text{min} + \varphi}[\eta_\text{ntm}](\gamma_0^\text{min} g)
 \, .\nonumber
    \end{align}
In order to obtain the last line we note that $\Gamma_{i-1}$ is defined above \eqref{=Ai} replacing $I$ with $I^{(\perp\alpha_1)}$, and evaluates to $\Gamma_0 = \Gamma_1 = \{1\}$ in this case.

    Now, fix a next-to-minimal element $\varphi_\text{ntm} \in \lie g^\times_{-\alpha_1 - \alpha_2 - \alpha_3} + \lie g^\times_{-\alpha_1 - \alpha_2 - \alpha_4}$ and let $\gamma_0^\text{ntm}$ be a representative of the Weyl word $s_2 s_1$ such that $\psi_\text{ntm} := \Ad^*(\gamma_0^\text{ntm}) \varphi_\text{ntm} \in \lie g^\times_{-\alpha_3} + \lie g^\times_{-\alpha_4}$.
     
    Using Theorem~\ref{thm:ntm-rep}(\ref{itm:ntm}) we get that
    \begin{equation}
        \label{eq:D4-ntm-ex}
        \mathcal{F}_{S_{\alpha_1}, \varphi_\text{ntm}}[\eta_\text{ntm}](g) = \intl{V_{\gamma_0^\text{ntm}}} \mathcal{W}_{\psi_\text{ntm}}[\eta_\text{ntm}](v \gamma_0^\text{ntm} g) \, dv \, ,
    \end{equation}
    where $V_{\gamma_0^\text{ntm}}$ is defined in \eqref{eq:V} and its Lie algebra here evaluates to $\lie g_{-\alpha_2}(\A) \oplus \lie g_{-\alpha_1 -\alpha_2}(\A)$.

    There are in fact three next-to-minimal (complex) orbits which are all related by triality. If the Whittaker support of $\eta_\text{ntm}$ does not include the orbit of $\varphi_\text{ntm}$ the corresponding Fourier coefficient $\mathcal{F}_{S_{\alpha_1}, \varphi_\text{ntm}}[\eta_\text{ntm}]$ is trivial and so is also the Whittaker coefficient $\mathcal{W}_{\psi_\text{ntm}}[\eta_\text{ntm}]$. The result \eqref{eq:D4-ntm-ex} is therefore only non-trivial for when the Whittaker support includes this orbit.
\end{example}

\begin{example}
    Let us also consider $G$ and $\eta_\text{ntm}$ as above, but now with $\alpha = \alpha_2$. We have that $I^{(\perp \alpha_2)}$ is empty. Thus, for any minimal $\varphi_\text{min}$, with an associated element $\gamma_0^\text{min} \in \Gamma$ and canonical form $\psi_\text{min} := \Ad^*(\gamma_0^\text{min}) \varphi_\text{min} \in \lie g^\times_{-\alpha_2}$, we get from Theorem~\ref{thm:ntm-rep}(\ref{it:Pt2}) that
    \begin{equation}
        \cF_{S_{\alpha_2}, \varphi_\text{min}}[\eta_\text{ntm}](g) = \cW_{\psi_\text{min}}[\eta_\text{ntm}](\gamma_0^\text{min} g) \, .
    \end{equation}
\end{example}

\begin{remark} 
It is interesting to ask which Fourier coefficients are Eulerian~\cite{Ginz,MR3161096}. The expectation, based on the reduction formula of~\cite{FKP2013} for Eisenstein series and explicit examples checked there, is that Whittaker coefficients $\mathcal{W}_{\varphi}[\eta]$ of an Eisenstein series $\eta$ on a group $G$ are Eulerian if the orbit of $\varphi$ is lies in $\mathrm{WS}(\eta)$. In general, the reduction formula expresses $\mathcal{W}_{\varphi}[\eta]$ through a \textit{sum of generic} Whittaker coefficients on a semi-simple group determined by $\varphi$. If $\Gamma\varphi\in\mathrm{WS}(\eta)$, this sum collapses to a single term in all known examples and since generic Whittaker coefficients on the subgroup are Eulerian this implies the same for $\mathcal{W}_{\varphi}[\eta]$.

For example, in the case of Eisenstein series attached  to the minimal representation of $E_6, E_7, E_8$ it was shown in~\cite{FKP2013} that $\mathcal{W}_{\varphi}[\eta]$ is given by just a single Whittaker coefficient on $\SL_2$, which is well known to be Eulerian. See also \cite[Ch.~10]{FGKP} for more details on these and other examples. {By Theorem \ref{thm:min-rep} this implies that   the parabolic Fourier coefficient $\mathcal{F}_{S_\alpha,\varphi}[\eta_{\textnormal{min}}]$ of an Eisenstein series in the minimal representation calculated in the unipotent of a maximal parabolic determined by $\alpha$ should be Eulerian for simply-laced split groups.

Conversely, \cite{KobayashiSavin} show  that if $G$ is linear, simply-connected and absolutely simple, and the form $\eta_{\min}$ generates an irreducible representation $\pi=\bigotimes \pi_{\nu}$ with all local components $\pi_{\nu}$ minimal
then $\mathcal{F}_{S_\alpha,\varphi}[\eta_{\textnormal{min}}]$ is Eulerian for any abelian root $\alp$ and non-zero $\varphi$ with $\ad^*(S_{\alp})\varphi=-2\varphi$. By Theorem \ref{thm:min-rep} this implies that the corresponding Whittaker coefficient is Eulerian.

We expect that Theorem \ref{thm:ntm-rep} will be useful to prove similar Eulerianity results for next-to-minimal representations.
By contrast, if $\Gamma \varphi\notin\WS(\eta)$ the Whittaker coefficients and Fourier coefficients corresponding to $\varphi$ are not expected to be Eulerian.}

\end{remark}

We can also express any next-to-minimal automorphic function in terms of its Whittaker coefficients, similar to Theorem \ref{thm:G0min} that treats the case of minimal automorphic functions. 
{This is the subject of the next subsection.}

\subsection{{Statement of Theorem \ref{thm:ntm-rep2}}}

{
\begin{notn}\label{not:D} Let $\alp$ be a simple root.
\begin{enumerate}[(i)]
\item Let $Q_{\alp}$ denote the parabolic subgroup of $L_{\alp}$ with Lie algebra $(\fl_{\alp})^{\alp^{\vee}}_{\leq 0}$. By Lemma \ref{lem:LineStab} below, $Q_{\alp}$ is the stabilizer in $L_{\alp}$ of the line $\fg^*_{-\alpha}$ as an element of the projective space of $\fg^*$.
Let $\Gamma_\alpha$ denote the {quotient of $L_{\alp}\cap \Gamma$ by $Q_{\alp}\cap \Gamma$.}
    
\item Let $\bfG_{\alp}$ denote the subgroup of $\bfG$ corresponding to the simple component of $\fg$ corresponding to $\alp$. Let $\alpha_\text{max}$ denote the highest root of $\bfG_{\alp}$.

\item We say that $\alp$ is \emph{nice} if one of the following holds:
\begin{enumerate}[(a)]
\item $\alp$ is an abelian root. 
\item $\bfG_{\alp}$ is of type $E$ and $\alp$ is a Heisenberg root.
\end{enumerate}
We exclude the Heisenberg root in type $D_n$ for several reasons. One is that it does not correspond to an extreme node in the Dynkin diagram.
We shall explain others in \S\ref{subsec:PfPropFab} below, see in particular Remark~\ref{rmk:alpha2} and Lemma~\ref{lem:An}.

\item\label{it:delta} {If $\alp$ is an abelian root, define $\delta_{\alp}:=\alp_{\max}$.}
If $\alp$ is a nice Heisenberg root, define $\delta_{\alp}:=\alp_{\max}-\alp-\beta_{\alp}$, where $\beta_{\alp}$ is the only simple root non-orthogonal to $\alp$. One can see that $\beta_{\alp}$ is unique by Table \ref{tab:QA}. For more details on $\delta_{\alp}${, and the proof that it is a root,}  see \S \ref{subsec:PfPropFab} below. 

\item \label{it:Lam} Let  $R_{\alp}$ denote the parabolic subgroup of $L_{\alp}$ with Lie algebra $(\fl_{\alp})^{\delta_{\alp}^{\vee}}_{\leq 0}$. Denote $\Lambda_{\alp}:=(L_{\alp}\cap \Gamma)/(Q_{\alp}\cap R_{\alp}\cap \Gamma)$. In \S\ref{subsec:PfPropStab} below we show that $Q_{\alp}\cap R_{\alp}\cap \Gamma$ is a subgroup of index two in the stabilizer in $L_{\alp}\cap \Gamma$ of the plane $\fg^*_{-\alp}\oplus \fg^*_{-\delta_{\alp}}$ as a point in the Grassmanian of planes in $\fg^*$.

\item Let $M_{\alp}$ denote the Levi subgroup of $G$ given by simple roots orthogonal to $\alp$. 
Denote $\cM_{\alp}:=(M_{\alp}\cap \Gamma)/(M_{\alp}\cap R_{\alp}\cap \Gamma)$.
In \S \ref{sec:D} below we show that $M_{\alp}\cap R_{\alp}\cap \Gamma$ is the stabilizer in $M_{\alp}\cap \Gamma$ of the plane $\fg^*_{-\alp}\oplus \fg^*_{-\delta_{\alp}}$.

\item If $\alp$ is a Heisenberg root we define 
\begin{equation}\label{=Omega2}
        \mathcal{B}_{\alp} := {\{\text{positive roots }\beta : \langle \alp, \beta \rangle = 1\}}, \quad \quad
        \Omega_\alp = \Exp(\bigoplus_{\beta \in \mathcal{B}_{\alp}} \lie g_{-\beta})\,.
\end{equation}
{We also fix  a representative $\gamma_{\alp}\in \Gamma$ for the Weyl group element $s_{\alp}s_{\alp_{\max}}s_{\alp}$, where $s_{\alp}$ and $s_{\alp_{\max}}$ denote the corresponding reflections. }
\end{enumerate}
\end{notn}
}
\begin{maintheorem}\label{thm:ntm-rep2}
Let $\eta_\textnormal{ntm}$ be a next-to-minimal automorphic function on $G$, and let $\alp$ be a nice simple root of $\fg$. 
\begin{enumerate}[(i)]
\item\label{it:D1}  If $\alp$ is an abelian root and $\langle \alp, \alp_{\max}\rangle >0$ then 
\begin{equation}\label{=ntm2AbEasy}
\eta_\textnormal{ntm}=\cF_{S_{\alp},0}[\eta_\textnormal{ntm}]+ \sum_{\gamma \in \Gamma_{\alp}}\sum_{ \varphi\in \fg^{\times}_{-\alp}}\mathcal{F}_{S_\alpha, \varphi}[\eta_\textnormal{ntm}](\gamma g)\,.
\end{equation}
{Denote the right-hand side of \eqref{=ntm2AbEasy} by $\cA'$, for any nice $\alp$.}
\item\label{it:2Ab} If $\alp$ is an abelian root and $\langle \alp, \alp_{\max}\rangle =0$ then

\begin{equation}\label{=ntm2Ab}
    \eta_\textnormal{ntm}=\cA'+{\frac{1}{2}}\sum_{\gamma\in \Lambda_{\alp}}\sum_{\substack{\varphi\in \fg^{\times}_{-\alp} \\ \psi\in \fg^{\times}_{-\delta_\alp}}}\cF_{S_\alpha,\varphi+\psi}[\eta_\textnormal{ntm}](\gamma g)\,.
\end{equation}
{Denote the right-hand side of \eqref{=ntm2Ab} by $\cA$.}

\item\label{it:2H} If $\alp$ is a nice Heisenberg root then
\begin{align}\label{=ntm2H}
    \eta_\textnormal{ntm}=\cA
  +  \sum_{\omega\in \Omega_{\alp}}\Biggl(\sum_{\varphi\in \fg^{\times}_{-\alp}}\cF_{S_{\alp},\varphi}[\eta_\textnormal{ntm}](\omega \gamma_{\alp}g)+\sum_{\gamma\in \cM_{\alp}} \sum_{\substack{\varphi\in \fg^{\times}_{-\alp} \\ \psi\in \fg^{\times}_{-\delta_\alp}}} \cF_{S_{\alp},\varphi+\psi}[\eta_\textnormal{ntm}](\gamma \omega \gamma_{\alp}g)\Biggr)\,.
\end{align}

\end{enumerate}
\end{maintheorem}

{Part~\eqref{it:D1} of the above theorem only arises in type $A$ when $\alpha$ is an extreme root of the diagram, part~\eqref{it:2Ab} applies to all other roots in type $A$ and to all abelian roots in types $D$ and $E$. Part~\eqref{it:2H} only applies to type $E$ and more specifically to root $\alpha_2$ for $E_6$, root $\alpha_1$ for $E_7$ and root $\alpha_8$ for $E_8$ using Bourbaki numbering. Note that $\delta_\alpha$ appearing in $\mathcal{A}$ in parts~\eqref{it:2Ab} and~\eqref{it:2H}  is as defined in Notation~\ref{not:D}\eqref{it:delta} and differs in the two parts.}

The right-hand sides of \eqref{=ntm2AbEasy}, \eqref{=ntm2Ab} and \eqref{=ntm2H} can be expressed in terms of Whittaker coefficients. Indeed, $\cF_{S_{\alp},\varphi+\psi}[\eta_\textnormal{ntm}]$ and $\cF_{S_{\alp},\varphi}[\eta_\textnormal{ntm}]$ can be expressed using Theorem \ref{thm:ntm-rep},
while $\cF_{S_{\alp},0}[\eta_\textnormal{ntm}]$ defines a next-to-minimal function on $L_{\alp}$, that can then be further decomposed using Theorem \ref{thm:ntm-rep2} by induction on the rank of $G$. To present this decomposition we will need some further notation.

\subsection{{Statements of Theorems \ref{thm:ntmFull}, \ref{thm:ntmFullSimple} and \ref{thm:ntmNoCusp}}}\label{subsec:EFG}

\begin{notn} {Let  $\beta_1,\dots,\beta_n$ be a quasi-abelian enumeration such that $\beta_1,\dots,\beta_{n-1}$ is an abelian enumeration for $L_{n-1}$}.  
In this notation we define the terms $A_{ij},B_{nj}$ to be used in the next two theorems.

\begin{enumerate}[(i)]

\item {For any $i\leq n$ we define $A_{ii}$ in the following way. Let $\alp_{\max}^i$ denote the highest root for the simple component of $\lie l_i$ containing $\beta_i$. If $\beta_i$ is abelian {in $\mf{l}_i$} and  $\alp_{\max}^i$ is not orthogonal to $\beta_i$ we set $A_{ii}=0$. Otherwise we define $\delta_i$ to be the root $\delta_{\beta_i}$ of $\fl_i$, and 
 fix 
a ${g_i}\in \Gamma$ that normalizes the torus and conjugates $\beta_i$ and $\delta_i$ to orthogonal simple roots. Such a ${g_i}$ exists by Corollary \ref{cor:2RootConjSimple}. Define $V_{{g_i}}$ as in \eqref{eq:V} and set 
 \begin{equation}\label{=Aii}
A_{ii}:={\frac{1}{2}} \sum_{\tilde \gamma \in\Lambda_{\beta_i}}\sum_{\varphi\in \fg^{\times}_{-\beta_i}}\sum_{ \psi\in \fg^{\times}_{-\delta_i}}\intl_{V_{{g_i}}} \mathcal{W}_{{\Ad^*(g_i)(\varphi+\psi)}}[\eta_\textnormal{ntm}](v {g_i} \tilde\gamma g) \, dv,
\end{equation}
where $\Lambda_{\beta_i}$ is the quotient of $L_{i-1}\cap \Gamma$ defined as in Notation \ref{not:D}\eqref{it:Lam} above. {As in Remark~\ref{rmk:CC}\eqref{rmk:indp}, the definition is independent of the choice of $g_i$.}}
\item Let $j<i$ such that $\langle \beta_i,\beta_j\rangle=0$. We define 
\begin{equation}A_{ij}[\eta]={\sum_{\gamma'\in\Gamma_{i-1}}}\sum_{\varphi\in \fg^{\times}_{-\beta_i}}\sum_{\gamma \in \Gamma_{j-1}} \sum_{\psi \in \lie g^\times_{-\beta_j}} \mathcal{W}_{\varphi + \psi}[\eta_\textnormal{ntm}](\gamma {\gamma'}g).
\end{equation}

\item {If $\beta_n$ is Heisenberg, 
fix  a representative $\gamma_{n}\in \Gamma$ for the Weyl group element $s_{\beta_n}s_{\alp_{\max}^n}s_{\beta_n}$, where $s_{\beta_n}$ and $s_{\alp_{\max}^n}$ denote the corresponding reflections.} 
\item {We will write $j\bot i$ if $\langle \beta_j,\beta_i\rangle=0$.}

\item {For any index $j$ with $j\bot n$ we define $B_{nj}$ in the following way.}
If $\beta_n$ is abelian we set $B_{nj}:=0$. For Heisenberg $\beta_n$ we define
$L'_{j}$ to be the Levi subgroup of $G$ given by the roots $\beta_k$ with $k< j$ and $k\bot n$, ${Q}'_{j{-1}}$ to be the subgroup of $L'_{j{-1}}$ that stabilizes the root space $\fg^*_{-\beta_j}$, and $\Gamma'_{j{-1}}:=(L'_{j{-1}}\cap \Gamma)/({Q}'_{j{-1}}\cap \Gamma)$. Set 

 \begin{equation}
 B_{nj}:=
\sum_{\omega\in \Omega_n} \sum_{\varphi\in \fg^{\times}_{-\beta_{{n}}}} 
     \sum_{\gamma \in \Gamma'_{j{-1}}}\sum_{\psi \in \lie g_{-\beta_j}^{\times} }
\mathcal{W}_{\varphi+\psi}[\eta_\textnormal{ntm}](\gamma  \omega{\gamma_{n}} g)\,.
 \end{equation}

\item If $\beta_n$ is abelian, we  define $B_{nn}$ to be zero. If $\beta_n$ is Heisenberg and nice, we {define 

 \begin{equation}
 B_{nn}:= \sum_{\omega\in \Omega_n}\sum_{\tilde \gamma \in \cM_{\beta_n}}\sum_{\varphi\in\fg^{\times}_{-\beta_n}}\sum_{\psi\in\fg^{\times}_{-\delta_n}} \, \intl_{V_{g_n}} \mathcal{W}_{{\Ad^*(g_n)(\varphi+\psi)}}[\eta_\textnormal{ntm}](v{g_n}\tilde \gamma \omega {\gamma_{n}}g) dv\,,
 \end{equation}}
 {which is again independent of the choice of $g_n$.}
\end{enumerate}
\end{notn}

Recall also the notation $A_i,B_i$ from \eqref{=Ai} and \eqref{=Bi}.
Applying Theorem \ref{thm:ntm-rep2} by induction and using {also Theorems \ref{thm:G0min} and \ref{thm:ntm-rep}, we obtain the following theorem.}
\begin{maintheorem}\label{thm:ntmFull}
{ Fix a quasi-abelian enumeration $\beta_1,\dots,\beta_n$ such that $\beta_1,\dots,\beta_{n-1}$ is an abelian enumeration for $L_{n-1}$, and $\beta_n$ is a nice quasi-abelian root.}
Let $\eta_\textnormal{ntm}$ be a next-to-minimal automorphic function on $G$. Then 
\begin{equation}\label{=ntmCor1}
\eta_\textnormal{ntm}=\cW_{0}[\eta_{\textnormal{ntm}}]+ \sum_{i} (A_i+A_{ii}+\sum_{j<i, j\bot i}A_{ij})+B_n+B_{nn}+\sum_{j\bot n}B_{nj}\,.
\end{equation}
\end{maintheorem}

{We note that if $\fg$ has at most one component of type $E_8$ then an enumeration as in Theorem \ref{thm:ntmFull} is always possible. For example one can take the Bourbaki enumeration on each component.
Note that the right-hand side of \eqref{=ntmCor1} is entirely expressed in terms of Whittaker coefficients.

One can simplify the expression in \eqref{=ntmCor1} by allowing oneself} to use in the final expression not only Whittaker coefficients, but also constant terms with respect to parabolic nilradicals, that in turn can be determined for Eisenstein series using \cite{MWCov}.
In this way one obtains the following statement.
\pagebreak

\begin{maintheorem}\label{thm:ntmFullSimple}
Assume that $[\fg,\fg]$ is simple {of rank $n$}, and fix the Bourbaki enumeration of its simple roots.  Let $\eta_\textnormal{ntm}$ be a next-to-minimal automorphic function on $G$. Then 
\begin{enumerate}[(i)]
\item In type $A$ we have $$\eta_\textnormal{ntm}=\cF_{S_{\alp_n},0}[\eta_\textnormal{ntm}]+ A_n+\sum_{{j\bot n}}A_{nj}\,.$$
\item In types $D,E_6$ and $E_7$ we have 
$$\eta_\textnormal{ntm}=\cF_{S_{\alp_n},0}[\eta_\textnormal{ntm}]+ A_n+\sum_{ j\bot n}A_{nj}+A_{nn}\,.$$
\item In type $E_8$ we have 
$$\eta_\textnormal{ntm}=\cF_{S_{\alp_n},0}[\eta_\textnormal{ntm}]+A_n+\sum_{ j\bot n}A_{nj}+A_{nn}+B_n+\sum_{ j\bot n}B_{nj}+B_{nn}\,.$$
\end{enumerate}
\end{maintheorem}

 Using Lemma \ref{lem:GlobLoc} below on the connection of Fourier coefficients {to wave-front sets of} local components we derive from Theorem \ref{thm:ntmFull} the following one.
\begin{maintheorem}
\label{thm:ntmNoCusp}
Let $\pi$ be an irreducible representation of $G$ and let $\pi=\bigotimes \pi_{\nu}$ be a decomposition of $\pi$ to local components. Suppose that there exists $\nu$ such that $\pi_{\nu}$ is minimal or next-to-minimal. Then $\pi$ cannot be realized in cuspidal automorphic forms on $G$.
\end{maintheorem}
For statements on the possibility of decomposing $\pi=\bigotimes \pi_{\nu}$ into local factors for covering groups, see \cite[\S 8]{Weis}.

{
\subsection{Illustrative examples}

Theorems \ref{thm:min-rep} and \ref{thm:G0min} build upon and extend the results of \cite{GRS2,MS,Ahlen:2017agd} for automorphic forms in the minimal representation. For the next-to-minimal representation, Theorems \ref{thm:ntm-rep} and \ref{thm:ntm-rep2} were established in \cite{Ahlen:2017agd} for $\SL_n$ and {are} here generalized to arbitrary simply-laced split Lie groups $G$. Together with Theorem~\ref{thm:ntm-rep2} they provide explicit expressions for the complete Fourier expansions of next-to-minimal automorphic forms on all split simply-laced groups.

In order to illustrate the types of explicit expansions one obtains we here give the expansion of minimal and next-to-minimal automorphic forms on $E_8$ using the $\alpha_8$ parabolic. For a minimal automorphic form one obtains
\begin{equation}
\eta_{\textnormal{min}}(g)=\mathcal{F}_{S_{\alpha_8} , 0}[\eta_{\textnormal{min}}] (g)+\sum_{\gamma\in\Gamma_7}\sum_{\varphi\in\mathfrak{g}_{-\alpha_8}^\times}\mathcal{W}_{\varphi}[\eta_{\textnormal{min}}](\gamma g)+\sum_{\omega\in \Omega_8} \sum_{\varphi\in \mathfrak{g}_{-\alpha_8}^{\times}}\mathcal{W}_\varphi[\eta_{\textnormal{min}}](\omega \gamma_8g)\,,
\end{equation}
while for a next-to-minimal automorphic form we have a slightly more complicated expression
\begin{align}
\eta_{\textnormal{ntm}}(g) &=\mathcal{F}_{S_{\alpha_8}, 0}(g)+\underbrace{\sum_{\gamma\in\Gamma_7}\sum_{\varphi\in\mathfrak{g}_{-\alpha_8}^\times}\mathcal{W}_{\varphi}(\gamma g)}_{A_8}
+\sum_{j=1}^{6} \underbrace{\sum_{\gamma'\in \Gamma_7}\sum_{\varphi\in \mathfrak{g}_{-\alpha_8}^{\times}} \sum_{\gamma \in \Gamma_{j-1}}\sum_{\psi\in \mathfrak{g}_{-\beta_j}^{\times}}\mathcal{W}_{\varphi+\psi}(\gamma \gamma' g)}_{A_{8j}} \nonumber
\\  \nonumber
&\quad+ \underbrace{\frac{1}{2}\sum_{\tilde{\gamma}\in \Lambda_{\alpha_8}}\sum_{\varphi\in \mathfrak{g}_{-\alpha_8}^{\times}}\sum_{\psi\in \mathfrak{g}_{-\delta_8}^{\times}}\int_{V_{g_8}} \mathcal{W}_{{\Ad^*(g_8)(\varphi+\psi)}}(vg_8\tilde{\gamma}g)dv}_{A_{88}}
+\underbrace{\sum_{\omega\in \Omega_8} \sum_{\varphi\in \mathfrak{g}_{-\alpha_8}^{\times}}\mathcal{W}_\varphi(\omega \gamma_8g)}_{B_8}
\\ \nonumber
&\quad+ \underbrace{\sum_{\omega\in \Omega_8} \sum_{\tilde{\gamma}\in \mathcal{M}_{\alpha_8}}\sum_{\varphi\in \mathfrak{g}_{-\alpha_8}^{\times}}\sum_{\psi\in \mathfrak{g}_{-\delta_8}^{\times}}\int_{V_{g_8}} \mathcal{W}_{{\Ad^*(g_8)(\varphi+\psi)}}(v{g_8}\tilde{\gamma}\omega \gamma_{8} g)dv}_{B_{88}}
\\ 
& \quad +\sum_{j=1}^6\underbrace{\sum_{\omega\in \Omega_8}\sum_{\varphi\in \mathfrak{g}_{-\alpha_8}^{\times}}\sum_{\gamma \in \Gamma'_{j-1}}\sum_{\psi\in \mathfrak{g}_{-\beta_j}^\times}\mathcal{W}_{\varphi+\psi}(\gamma \omega \gamma_{8} g)}_{B_{8j}}\,.
\end{align}
All coefficients are evaluated for the automorphic form $\eta=\eta_{\textnormal{ntm}}$. The elements $g_8$ and $\gamma_8$ are defined in \S\ref{subsec:EFG} and \S\ref{subsec:ThB}, respectively,

We shall compare these to other results available in the literature in \S\ref{subsec:compar}.

}

\subsection{Motivation from string theory}
\label{sec:string}

The results of this paper have applications in string theory. In short, string theory predicts certain quantum corrections to Einstein's theory theory of general relativity. These quantum corrections come in the form of an expansion in curvature tensors and their derivatives. The first non-trivial correction is of fourth order in the Riemann tensor, denoted schematically $\mathcal{R}^4$, and has a coefficient which is a function $\eta_n: E_n/K_n \to \mathbb{R}$, where $E_n/K_n$  is a particular symmetric space, the classical moduli space of the theory. The parameter $n=d+1$ contains the number of spacetime dimensions $d$ that have been compactified on a torus $T^d$. The groups $E_n$ are all split real forms of rank $n$ complex Lie groups (see Table \ref{tab:CJ}).

\begin{table}[h]
\centering
\caption{\label{tab:CJ} Table of Cremmer--Julia symmetry groups $E_n(\reals), \, n=d+1,$ with compact subgroup $K_n(\reals)$ and U-duality groups $E_n(\ints)$ for compactifications of  IIB string theory on a $d$-dimensional torus $T^d$ to $D = 10 - d$ dimensions.}
\begin{tabular}{|c|c|c|c|}
\hline
$d$ & $E_{d+1}(\reals)$ & $K_{d+1}(\reals)$ & $E_{d+1}(\ints)$ \\
\hline
$0$ & $\SL_2(\reals)$ & $\SO_2(\reals)$ & $\SL_2(\ints)$ \\
$1$ & $\SL_2(\reals)\times \reals_+$ & $\SO_2(\reals)$ & $\SL_2(\ints)$ \\
$2$ & $\SL_2(\reals)\times \SL_3(\reals)$ & $\SO_2(\reals)\times \SO_3(\reals)$ & $\SL_2(\ints)\times \SL_3(\ints)$\\
$3$ & $\SL_5(\reals)$ & $\SO_5(\reals)$ & $\SL_5(\ints)$ \\
$4$ & $\Spin_{5,5}(\reals)$ & $\Spin_5(\reals)\times \Spin_5(\reals)$ & $\Spin_{5,5}(\ints)$\\
$5$ & $E_6(\reals)$ & $\operatorname{USp}_8(\reals)/\ints_2$ & $E_6(\ints)$ \\
$6$ & $E_7(\reals)$ & $\operatorname{SU}_8(\reals)/\ints_2$ & $E_7(\ints)$\\
$7$ & $E_8(\reals)$ & $\operatorname{Spin}_{16}(\reals)/\ints_2$ & $E_8(\ints)$\\
\hline
\end{tabular}
\end{table}

In the full quantum theory the classical symmetry $E_n(\mathbb{R})$ is broken to an arithmetic subgroup $E_n(\mathbb{Z})$, called the \emph{U-duality group}, which is the Chevalley group of integer points of $E_n$~\cite{Hull:1994ys}. Thus, the coefficient functions $\eta_n$ are really functions on the double coset $E_n(\mathbb{Z})\backslash E_n(\mathbb{R})/K_n$ and in certain cases they can be uniquely determined. For the two leading order quantum corrections, corresponding to $\mathcal{R}^4$ and $\partial^4\mathcal{R}^4$, the coefficient functions $\eta_n$ are respectively attached to the minimal and next-to-minimal automorphic representations of $E_n$ \cite{P,GMV}. Fourier expanding $\eta_n$ with respect to various unipotent subgroups $U\subset E_n$  reveals interesting information about perturbative and non-perturbative quantum effects. Of particular interest are the cases when $U$ is the unipotent radical of a maximal parabolic $P_\alpha\subset G$ corresponding to a simple root $\alpha$ at an ``extreme'' node (or end node) in the Dynkin diagram. Consider the sequence of groups $E_n$ displayed in Table \ref{tab:CJ}, and the associated Dynkin diagram in ``Bourbaki labelling''. The extreme simple roots are then $\alpha_1, \alpha_2$ and $\alpha_n$ (this is slightly modified for the low rank cases where the Dynkin diagram becomes disconnected). Fourier expanding the automorphic form $\eta$ with respect to the corresponding maximal parabolics then have the following interpretations (see Figure \ref{fig:limits} for the associated labelled Dynkin diagrams):
\begin{itemize}
\item $P=P_{\alpha_1}$: {\bf String perturbation limit.}  In this case the constant term of the Fourier expansion corresponds to perturbative terms (tree level, one-loop etc.) with respect to an expansion around small string coupling, $g_s\to 0$. The non-constant Fourier coefficients encode non-perturbative effects of the order $e^{-1/g_s}$ and $e^{-1/g_s^2}$ arising from so-called D-instantons and NS5-instantons.
\item $P=P_{\alpha_2}$: {\bf M-theory limit}. This is an expansion in the limit of large volume of the M-theory torus $T^{d+1}$. The non-perturbative effects arise from M2- and M5-brane instantons.
\item $P=P_{\alpha_n}$: {\bf Decompactification limit.}   This is an expansion in the limit of large volume of a single circle $S^1$ in the  torus $T^d$ (or $T^{d+1}$ in the M-theory picture). The non-perturbative effects encoded in the non-constant Fourier coefficients correspond to so called BPS-instantons and Kaluza--Klein instantons. 
\end{itemize}

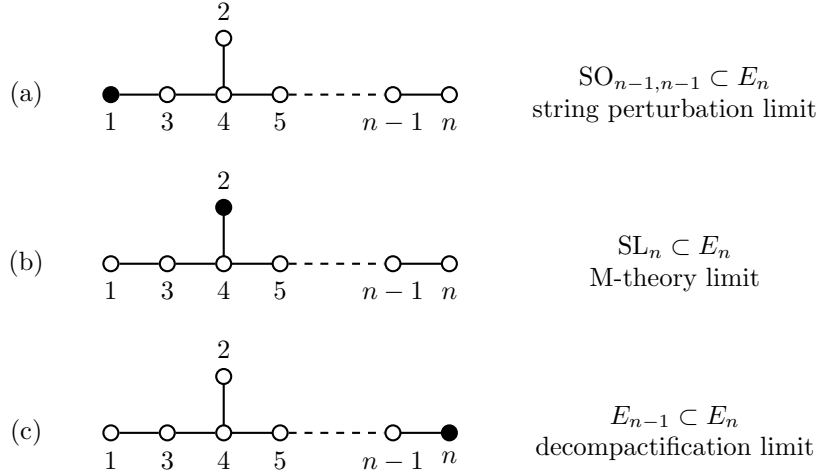
\begin{figure}[t!]
\centering
\begin{tikzpicture}[font=\small, scale=0.75]
\def\nodedist{1}
    \tikzstyle{dot}=[circle, draw, thick, fill=white, inner sep=2pt]
    \tikzstyle{fdot}=[circle, draw, thick, fill=black, inner sep=2pt]
    \tikzstyle{line}=[thick, shorten >=-2pt, shorten <=-2pt]
    \node (n1) at (0, 0) [dot, label=below:1] {};
    \node (n2) at (1*\nodedist, 0) [dot, label=below:3] {};
    \node (n3) at (2*\nodedist, 0) [dot, label=below:4] {};
    \node (n4) at (2*\nodedist, 1*\nodedist) [dot, label=above:2] {};
    \node (n5) at (3*\nodedist, 0) [dot, label=below:5] {};
     \node (n7) at (5*\nodedist, 0) [dot, label=below:$n-1$] {};
    \node (n8) at (6*\nodedist, 0) [fdot, label=below:$n$] {};

    \begin{scope}[on background layer] 
        \draw [line] (n1) -- (n2) -- (n3) -- (n5);
        \draw [line] (n3) -- (n4);
        \draw [dashed, thick] (n5) -- (n7);
        \draw [line] (n7) -- (n8);
    \end{scope}

    \node at (-1.5*\nodedist,0){(c)};
    \node at (10*\nodedist,0) {\begin{tabular}{cc}$E_{n-1}\subset E_{n}$\\decompactification limit\end{tabular}};

\def\yshift{6}

    \node (n1) at (0, \yshift) [fdot, label=below:1] {};
    \node (n2) at (1*\nodedist, \yshift) [dot, label=below:3] {};
    \node (n3) at (2*\nodedist, \yshift) [dot, label=below:4] {};
    \node (n4) at (2*\nodedist, 1*\nodedist+\yshift) [dot, label=above:2] {};
    \node (n5) at (3*\nodedist, \yshift) [dot, label=below:5] {};
     \node (n7) at (5*\nodedist, \yshift) [dot, label=below:$n-1$] {};
     \node (n8) at (6*\nodedist, \yshift) [dot, label=below:$n\vphantom{1}$] {};

    \begin{scope}[on background layer] 
        \draw [line] (n1) -- (n2) -- (n3) -- (n5);
        \draw [line] (n3) -- (n4);
        \draw [dashed, thick] (n5) -- (n7);
        \draw [line] (n7) -- (n8);
    \end{scope}
    
    \node at (-1.5*\nodedist,\yshift){(a)};
    \node at (10*\nodedist,\yshift) {\begin{tabular}{cc}$\SO_{n-1,n-1}\subset E_{n}$\\string perturbation limit\end{tabular}};

\def\yshift{3}

    \node (n1) at (0, \yshift) [dot, label=below:1] {};
    \node (n2) at (1*\nodedist, \yshift) [dot, label=below:3] {};
    \node (n3) at (2*\nodedist, \yshift) [dot, label=below:4] {};
    \node (n4) at (2*\nodedist, 1*\nodedist+\yshift) [fdot, label=above:2] {};
    \node (n5) at (3*\nodedist, \yshift) [dot, label=below:5] {};
     \node (n7) at (5*\nodedist, \yshift) [dot, label=below:$n-1$] {};
    \node (n8) at (6*\nodedist, \yshift) [dot, label=below:$n\vphantom{1}$] {};

    \begin{scope}[on background layer] 
        \draw [line] (n1) -- (n2) -- (n3) -- (n5);
        \draw [line] (n3) -- (n4);
        \draw [dashed, thick] (n5) -- (n7);
        \draw [line] (n7) -- (n8);
    \end{scope}

    \node at (-1.5*\nodedist,\yshift){(b)};
    \node at (10*\nodedist,\yshift) {\begin{tabular}{cc}$\SL_n\subset E_{n}$\\M-theory limit\end{tabular}};

\end{tikzpicture}
\caption{\label{fig:limits}The various string theory limits associated with different maximal parabolic subgroups $P_\alpha$. Roots are labeled in the Bourbaki ordering.}
\end{figure}

For the reasons presented above, it is of interest in string theory to have general techniques for explicitly  calculating Fourier coefficients of automorphic forms with respect to arbitrary unipotent subgroups. 

In string theory the abelian and non-abelian Fourier coefficients of the type defined in~\eqref{eq:FCU} typically reveal different types of non-perturbative effects (see for instance \cite{PP,BKNPP,Pe}). The archimedean and non-archimedean parts of the adelic integrals have different interpretations in terms of combinatorial properties of instantons and the instanton action, respectively. For example, in the simplest case of an Eisenstein series on $\SL_2$ the non-archimedean part is a divisor sum $\sigma_k(n)= \sum_{d|n} d^k$ and corresponds to properties of D-instantons~\cite{Green:1997tv,Green:1997tn,Kostov:1998pg,Moore:1998et} (see also \cite{FGKP} for a detailed discussion in the present context).
Theorem~\ref{thm:ntmFullSimple} provides explicit expressions for the Fourier coefficients of the automorphic coupling of the next-to-minimal $\partial^4 \mathcal{R}^4$ higher derivative correction in various limits; see section~\ref{sec:E8ex} for a more detailed discussion in the case of $E_8$.

\begin{remark}
Theorem \ref{thm:ntmNoCusp} resolves a long-standing question in string theory which concerns the possibility of having contributions from cusp forms in the $\mathcal{R}^4$ and $\partial^4\mathcal{R}^4$ amplitudes. The theorem ensures that this can never happen as there are no cusp forms in the minimal or next-to-minimal spectrum. 
\end{remark}

\subsection{Structure of the paper}

In \S \ref{sec:prel} we give the definitions of the notions mentioned above. 

In \S \ref{subsec:Relations} we introduce the results of \cite{Part1} that relate Fourier coefficients corresponding to different Whittaker pairs, in particular Theorem  \ref{thm:IntTrans}, which is  the main tool of the current paper. Two more results from  \cite{Part1} that we recall in \S \ref{subsec:Relations} and heavily use in the rest of the paper are Proposition \ref{prop:Heis} that expresses any automorphic function through  Heisenberg parabolic Fourier coefficients, and a geometric Lemma \ref{lem:SameOrbit3}.

In \S \ref{sec:ABC} we deduce Theorems \ref{thm:min-rep}-\ref{thm:ntm-rep} from \S\ref{subsec:Relations}. 
For  Theorem \ref{thm:min-rep}\eqref{it:min0phi} we show that any Fourier coefficient $\cF_{S,\psi}$ of the constant term equals a Fourier coefficient $\cF_{H,\psi}$ of $\eta$ for some $H$, and thus vanishes unless $\psi$ is minimal or zero. We first deduce from Lemma \ref{lem:SameOrbit3} that any minimal $\varphi\in (\fg^*)^{S_{\alp}}_{-2}$ can be conjugated into $\fg^{\times}_{-\alp}$ using $L_{\alp}\cap \Gamma$ (Corollary \ref{cor:EasyMin}). This, together with Theorem \ref{thm:IntTrans}, implies Theorem \ref{thm:min-rep}\eqref{it:min-min}. Part \eqref{it:min0} of Theorem \ref{thm:min-rep} follows from the definition of minimality and Corollary \ref{cor:domin}, which says that any Fourier coefficient is linearly determined by a neutral Fourier coefficient corresponding to the same orbit.  

To prove Theorem \ref{thm:G0min}, assume first that $\alp:=\beta_n$ is an abelian root. In this case we decompose the form $\eta_\textnormal{min}$ into Fourier series with respect to $U_{\alp}$. Each Fourier coefficient is of the form $\cF_{S_{\alp},\varphi}$. For $\varphi=0$,  the restriction of this coefficient to $L_{\alp}$ is minimal and we use the theorem for $L_{\alp}$ (by induction on rank). For non-zero and non-minimal $\varphi$, $\cF_{S_{\alp},\varphi}$ vanishes by Theorem \ref{thm:min-rep}\eqref{it:min0}. For minimal $\varphi$ the expressions for $\cF_{S_{\alp},\varphi}$ are given by Theorem \ref{thm:min-rep}\eqref{it:min-min}. We group them together using Corollary \ref{cor:EasyMin}. If $\alp$ is a Heisenberg root, we express $\eta_\textnormal{min}$ through parabolic Fourier coefficients $\cF_{S_{\alp},\varphi}$ using Proposition \ref{prop:Heis}. For $\varphi\neq 0$, $\cF_{S_{\alp},\varphi}$ is given by Theorem \ref{thm:min-rep}, and for $\varphi=0$ by induction.

Theorem \ref{thm:ntm-rep}\eqref{it:ntm0phi} is proven similarly to Theorem \ref{thm:min-rep}\eqref{it:min0phi}.
{To prove Theorem \ref{thm:ntm-rep}\eqref{it:Pt2} we restrict $\cF_{S_{\alp},\varphi}[\eta_\textnormal{ntm}]$ to $L_{\alp}$, show that it is a minimal automorphic function and apply Theorem \ref{thm:G0min}. 
Theorem \ref{thm:ntm-rep}\eqref{itm:ntm} and \ref{thm:ntm-rep}\eqref{itm:larger} follow from Theorem \ref{thm:IntTrans} and Corollary  \ref{cor:domin} respectively.
For Theorem \ref{thm:ntm-rep}\eqref{itm:ntm} we also use a geometric lemma saying that any next-to-minimal $\varphi\in (\fg^*)^{S_{\alp}}_{-2}$ can be conjugated into $\fg^{\times}_{-\alp}+\fg^{\times}_{-\beta}$ for some positive root $\beta$ orthogonal to $\alp$ using $L_{\alp}\cap \Gamma$ (Lemma \ref{lem:AppGeo}).}

 In \S \ref{sec:D} we {first} prove Theorem \ref{thm:ntm-rep2} using the same strategy as in the proof of Theorem \ref{thm:G0min}. However, we need two additional geometric propositions {(Propositions \ref{prop:nice} and \ref{prop:Fab})} that describe the action of $L_{\alp}$ on next-to-minimal elements of $(\fg^*)^{S_{\alp}}_{-2}$.
We prove these in \S \ref{subsec:PfPropFab}.  {In \S \ref{subsec:PfEFG} we derive Theorems \ref{thm:ntmFull}, \ref{thm:ntmFullSimple}, and \ref{thm:ntmNoCusp} from Theorems \ref{thm:G0min}, \ref{thm:ntm-rep}, and \ref{thm:ntm-rep2}.}

In \S \ref{sec:examples} we provide examples of Theorems \ref{thm:min-rep}-\ref{thm:ntm-rep2} for groups of type $D_5$ and $E_8$ computing the expansions of automorphic function and Fourier coefficients with respect to different parabolic subgroups of interest in string theory and compare our $E_8$ results to the available literature~\cite{Bossard:2016hgy,GKP, KazhdanPolishchuk}.

\subsection{Acknowledgements} 
We are  {grateful} for helpful discussions with Guillaume Bossard, Ben Brubaker, Daniel Bump, Solomon Friedberg, David Ginzburg, Joseph Hundley, Stephen D. Miller, Manish Patnaik, Boris Pioline and Gordan Savin. 
{We also thank the Banff International Research Station for Mathematical Innovation and Discovery and the Simons Center for Geometry and Physics for their hospitality during different stages of this project.} D.G. was partially supported by ERC StG grant 637912. H.G was supported by the Knut and Alice Wallenberg Foundation. D.P. was supported by the Swedish Research Council (Vetenskapsr\aa det), grant nr. 2018-04760.
S.S. was partially supported by Simons Foundation grant 509766.

\section{{Definitions and preliminaries}}\label{sec:prel}
Let $\K$ be a number field and let $\A=\A_{\K}$ be its ring of adeles. {Fix a non-trivial unitary character $\chi$ of} $\A$, which is trivial on $\K$. Then $\chi$ defines an isomorphism between $\A$ and $\hat{\A}$ via the map $a\mapsto \chi_{a}$, where $\chi_{a}(b)=\chi(ab)$ for all $b\in \A$. This isomorphism restricts to an  isomorphism
\begin{equation}\label{eq:chi_isomorphism}
 \widehat{\A/\K}\cong \{r\in \hat{\A}\, : \,|r|_{\K}\equiv 1\}=\{\chi_{a} : a \in \K\}\cong \K.
\end{equation}

Let ${\bf G}$ be a  reductive group defined over $\K$, ${\bf G}(\mathbb{A})$ the group of adelic points of ${\bf G}$ and $G$ be a finite central extension of ${\bf G}(\A)$. We assume that there exists a section ${\bf G}(\K)\to G$ of the covering $p:G\to {\bf G}(\A)$,  fix such a section and denote its image by $\Gamma$. 
By \cite[Appendix I]{MWCov}, for any unipotent subgroup $\bf U \subset G$, $p$ has a canonical section on $\bf U(\A)$. We will always use this  to identify $\bf U(\A)$ with a subgroup of $G$.
Let $\fg$ denote the Lie algebra of $\bfG(\K) \cong\Gamma$.

For a nilpotent subalgebra $\fv\subset \fg$, we denote by $\Exp(\fv)$ the unipotent subgroup of $\Gamma$ obtained by exponentiation of $\fv$. Similarly, we denote by $V:=\Exp(\fv(\A))$ the unipotent subgroup of $G$ obtained by exponentiation of the adelization $\fv(\A):=\fv\otimes_{\K}\A$.

\begin{definition}
\label{def:pair}
A \emph{Whittaker pair} is an ordered pair $(S,\varphi)\in \fg\times \fg^*$ such that $S$ is a rational semi-simple element (that is, with eigenvalues of the adjoint action $\ad(S)$ in $\Q$), and $\ad^*(S)(\varphi)=-2\varphi$. \end{definition}

{We will say that an element of $\fg^*$ is \emph{nilpotent} if it is given by the Killing form pairing with a nilpotent element of $\fg$. Equivalently, $\varphi\in \fg^*$ is nilpotent if and only if the Zariski closure of its coadjoint orbit includes zero. For example, if $(S,\varphi)$ is a Whittaker pair then $\varphi$ is nilpotent. 

For any rational semi-simple $S\in \fg$ and $i\in \Q$ we set 
\begin{equation}
\fg_{i}^{S}:=\{X\in \fg \, : \, [S,X]=iX\},\quad \fg_{> i}^{S}:=\bigoplus_{j>i\in \Q}\fg_{j}^S, \text{ and } \fg_{\geq i}^S :=\fg_{i}^{S}\oplus\fg_{>i}^S\,.
\end{equation}
We will also use similar notation for $(\fg^*)^S_{i}$.

For any $\varphi\in \fg^*$ we define an anti-symmetric form $\omega_{\varphi}$ of $\fg$ by 
\begin{equation}
\omega_{\varphi}(X,Y)=\varphi([X,Y]).
\end{equation}

Given a Whittaker pair $(S,\varphi)$ on $\g$, we set $\fu:=\fg_{>1}^S \oplus \fg_1^S$ and define 
\begin{equation}\label{=Nsphi}
\fn_{S,\varphi}:=\{ X \in \fu \,:\, \omega_\varphi(X,Y)=0 \,\,\textrm{for all $Y\in \fu$}\}\quad\text{and }N_{S,\varphi}:=\Exp \fn_{S,\varphi}(\A) 
\end{equation} 

By \cite[Lemma 3.2.6]{GGS}, 
\begin{equation}
    \label{eq:N_Sphi}
    \lie n_{S,\varphi} = \lie g^S_{>1} \oplus (\lie g^S_1 \cap \lie g_\varphi),
\end{equation}
where $\lie g_\varphi$ is} the centralizer of $\varphi$ in $\lie g$ under the coadjoint action.
 Note that $\fn_{S,\varphi}$ is an ideal in $\fu$ with abelian quotient, and that $\varphi$ defines a character of $\fn_{S,\varphi}$. Define an automorphic character on $N_{S,\varphi}$ by  $\chi_{\varphi}(\exp X):=\chi(\varphi(X))$.
For a unipotent subgroup $U\subset G$ we denote by $[U]$ the quotient $(U\cap \Gamma)\backslash U$.

{We call a function on $G$ an \emph{automorphic function} if  it is left $\Gamma$-invariant, finite under the right action of the preimage in $G$ of $\prod_{\text{finite }\nu}{\bf G}(\cO_\nu),$ and smooth when restricted to the preimage in $G$ of $\prod_{\text{infinite }\nu}{\bf G}(\K_\nu)$. We denote the space of all automorphic functions by $C^{\infty}(\Gamma\backslash G)$.}
\begin{definition}\label{def:Whit}
  For an automorphic function $\eta$, we define the \emph{Fourier coefficient} of $\eta$ with respect to a Whittaker pair $(S,\varphi)$ to be 
    \begin{equation}\label{eq:Whittaker-Fourier-coefficient}
        \cF_{S,\varphi}[\eta](g):=\intl_{[N_{S,\varphi}]}\eta(ng)\, \chi_{\varphi}(n)^{-1}\, dn.
    \end{equation}
\end{definition}

\begin{definition}
A Whittaker pair $(H,\varphi)$ is called a neutral Whittaker pair if either $(H,\varphi)=(0,0)$, or $H$ can be completed to an $\sl_2$-triple $(e,H,f)$ such that $\varphi$ is the Killing form pairing with $f$. Equivalently, the coadjoint action on $\varphi$ defines an epimorphism $\fg^H_{0}\onto(\fg^*)^H_{-2} $, and also $H$ can be completed  to an $\sl_2$-triple.
For more details on $\sl_2$-triples over arbitrary fields of characteristic zero see \cite[\S 11]{Bou}. 
\end{definition}

{
\begin{definition} 
\label{def:principal}

We call a Whittaker pair $(S,\varphi)$ \emph{standard} if $N_{S,\varphi}$ is the unipotent radical of a Borel subgroup of $G$.  By \cite[Corollary 2.1.5]{Part1}, a nilpotent $\varphi\in \fg^*$ can be completed to a standard Whittaker pair if and only if it is a principal nilpotent element of some $\K$-Levi subgroup of $G$. Here, principal means that the dimension of its centralizer equals the rank of the group. We call such $\varphi$ \emph{PL-nilpotent}, and their orbits \emph{PL-orbits}. For a standard pair $(S,\varphi)$, we call the Fourier coefficient $\cF_{S,\varphi}$ a \emph{Whittaker} coefficient {and denote it $\cW_{S,\varphi}$, or $\cW_{\varphi}$ if $S$ is defines the fixed Borel subgroup, see \eqref{=cW}}.
\end{definition}
}
\begin{remark}
\begin{enumerate}[(i)]
\item In \cite[\S 6]{GGS} the integral~\eqref{eq:Whittaker-Fourier-coefficient} above is called a Whittaker--Fourier coefficient, but in this paper we call it Fourier coefficient for short.
The  Whittaker coefficients are called in \cite[\S 6]{GGS} principal degenerate Whittaker--Fourier coefficients. {The notation $\cW_{S,\varphi}$ and $\cW_\varphi$ is used in \cite{GGS,GGS:support} to denote something quite different.}
\item Note that for $G=\GL_n$ all orbits $\mathcal{O}$ are PL-orbits. In general this is, however, not the case, see \cite[Appendix~A]{Part1}.
{
\item We refer the readers interested in the definitions of principal nilpotents, PL-nilpotents, and standard pairs for non-quasi-split groups to \cite[\S 2.1]{Part1}.}
\end{enumerate}
\end{remark}
\begin{defn}\label{def:WO}
For an automorphic function $\eta$, we define $\mathrm{WO}(\eta)$ to be the set of nilpotent orbits $\cO$ under $\Gamma$ in $\fg^*$ such that $\cF_{h,\varphi}[\eta]\neq 0$ for some neutral Whittaker pair $(h,\varphi)$ with $\varphi\in \cO$. We define the \emph{Whittaker support} $\WS(\eta)$ to be the set of maximal elements in $\mathrm{WO}(\eta)$.
\end{defn}

The following well known lemma relates these notions to the local notion of wave-front set. For a survey on this notion, and its relation to degenerate Whittaker models we refer the reader to \cite[\S 4]{GS_survey}.
\begin{lem}\label{lem:GlobLoc}
Suppose that $\eta$ is an automorphic form in the classical sense, and that it generates an irreducible representation $\pi$ of $G$. Let $\pi=\bigotimes_{\nu}\pi_{\nu}$ be the decomposition of $\pi$ to local factors. Let $\cO\in \mathrm{WO}(\eta)$. Then, for any $\nu$, there exists an orbit $\cO'_{\nu}$ in the wave-front set of $\pi_{\nu}$ such that $\cO$ lies in the Zariski closure of $\cO'_{\nu}$. Moreover, if $\nu$ is non-archimedean, then $\cO$ lies in the closure of $\cO'_{\nu}$ in the topology of $\fg^*(\K_{\nu})$.
\end{lem}
\begin{proof}
Acting by $G$ on the argument of $\eta$ we can assume that there exists a neutral pair $(h,\varphi)$ with $\varphi\in \cO$ such that $\cF_{h,\varphi}[\eta](1)\neq 0$. Moreover, decomposing $\eta$ to a sum of pure tensors, and replacing $\eta$ by one of the summands, we can assume that $\eta$ is a pure tensor and  $\cF_{h,\varphi}[\eta](1)\neq 0$ still holds. 
Let $\eta= \bigotimes'_{\mu}v_{\mu}$ be the decomposition of $\eta$ to local factors. Consider the functional $\xi$ on $\pi_{\nu}$ given by $\xi(v):=\cF_{h,\varphi}(v_{{\nu}}\otimes (\bigotimes'_{\mu\neq \nu}v_{\mu}))(1)$. Substituting the vector $v_{\nu}$ we see that this functional is non-zero. It is easy to see that this $\xi$ is $(\exp({\fn}_{h,\varphi}(\K_{\nu})),\chi_{\varphi})$-equivariant. The theorem follows now from \cite[Proposition I.11]{MW} and \cite{Var} for non-Archimedean $\nu$, and from \cite[Theorem D]{Ros} and \cite{Mat} for archimedean $\nu$.
\end{proof}

{For convenience, we fix a complex embedding  $\sigma:\K\into\C$. This embedding will allow us to speak about the complex nilpotent orbit corresponding to an orbit $\cO$ of $\Gamma$ in $\fg$.} {One can show using \cite{Dok} that the complex orbit corresponding to $\cO$ does not depend on $\sigma$, although we shall not use this fact. None of our statements depends on the choice of complex embedding $\sigma$.}

\subsection{Minimal and next-to-minimal representations}\label{subsec:small}

{
We call a non-zero complex orbit in $\fg^*(\C)$ \emph{minimal} if its Zariski closure  $\overline{\cO}$ is a disjoint union of $\cO$ and the zero orbit.  We call a complex orbit $\cO$ \emph{next-to-minimal} if $\cO$ does not intersect any component of $\fg$ of type $A_2$, and  
 $\overline{\cO}$ is a disjoint union of $\cO$,  minimal orbits, and the zero orbit.

\begin{lemma}\label{lem:BCsmall}
Let $\fg$ be simple and let $\cO\subset \fg^*(\C)$ be a complex nilpotent orbit. 
Then  $\cO$ is minimal if and only if it has Bala-Carter label $A_1$ and next-to-minimal if and only if it has Bala-Carter label $A_1\times A_1$.
\end{lemma}
\begin{proof}
Follows  from the Hasse diagrams for the closure order on nilpotent orbits.   
\end{proof}
\begin{remark}
\begin{enumerate}[(i)]
\item Lemma \ref{lem:BCsmall} only holds for simply-laced Lie algebras. Indeed, already for $C_2$ the minimal orbit is represented by the long root, and the next-to-minimal by the short root. Both roots of course lie in Levi subalgebras of type $A_1$. 
\item We exclude the regular orbit of $A_2$ because it does not behave like a next-to-minimal orbit. This behaviour is manifested by Lemma \ref{lem:BCsmall}.
\end{enumerate}
\end{remark}

\begin{lem}\label{lem:MinProd}
Let $\fg=\bigoplus_{i=1}^k \fg_i$, with  $\lie g_i$ simple. {Then} the minimal orbits of $\fg^*(\C)$ are of the form $\bigtimes_{i=1}^j\{0\}\times \cO\times \bigtimes_{i=j+1}^k\{0\}$, with $\cO$ a minimal orbit. The next-to-minimal orbits of $\fg^*(\C)$ are either of the same form  with $\cO$ next-to-minimal, or of the form $\bigtimes_{i=1}^{j-1}\{0\}\times \cO\times \bigtimes_{i=j+1}^{l-1}\{0\}\times \cO'\times \bigtimes_{i=l+1}^k\{0\},$
where $\cO$ and $\cO'$ are minimal orbits in $\fg_j$ and $\fg_l$ respectively.
\end{lem}
\begin{proof}
If $\fg = \fg_1\times \fg_2$ and $\cO=\cO_1\times \cO_2$ then $\overline{\cO} =\overline{\cO}_1\times  \overline{\cO}_2$.
\end{proof}

}

We call a (rational) element of $\fg^*$ or a rational orbit in $\fg^*$ minimal/next-to-minimal if its complex orbit is minimal/next-to-minimal.

We say that an automorphic function $\eta$ is \emph{minimal} if $\WS(\eta)$ consists of minimal orbits. By \cite[Theorem C]{GGS} (or by  
Proposition \ref{prop:domin} below), this implies that $\cF_{H,\varphi}[\eta]=0$ for any Whittaker pair $(H,\varphi)$ with $\varphi$ non-zero and non-minimal. We call an automorphic function $\eta$ \emph{trivial} if $\WS(\eta)=\{0\}$.
By \cite[Corollary 8.2.2]{GGS}, the semi-simple part of $G$ acts on any trivial automorphic function by $\pm\Id$.
We call a  representation of $G$ in automorphic functions \emph{minimal} if all the functions in this representation are minimal or trivial.

We say that an automorphic function $\eta$ is \emph{next-to-minimal} if $\WS(\eta)$
consists of next-to-minimal orbits.
 Again, by \cite[Theorem C]{GGS} (or by  
Proposition \ref{prop:domin} below), this implies that $\cF_{H,\varphi}[\eta]=0$ for any Whittaker pair $(H,\varphi)$ with $\varphi$ higher than next-to-minimal. 
We call a representation $\pi$ of $G$ in automorphic functions \emph{next-to-minimal} if it includes a next-to-minimal function, and all the functions in this representation are next-to-minimal, minimal or trivial. By Lemma \ref{lem:GlobLoc}, if $\pi$ consists of automorphic forms in the classical sense, is non-trivial, irreducible and has a minimal local factor then it is minimal. Similarly, if it has a next-to-minimal local factor then it is minimal or next-to-minimal.

\subsection{Relating different Whittaker pairs}\label{subsec:Relations}

\begin{lemma}[{\cite[Lemma 3.3.1]{Part1}}]
    \label{lem:conjugation-translation}
    Let $(S, \varphi)$ be a Whittaker pair, $\eta$ an automorphic function and $\gamma \in \Gamma$. Then,
    \begin{equation}
        \label{eq:conjugation}
        \cF_{S, \varphi}[\eta](g) = \cF_{\Ad(\gamma) S, \Ad^*(\gamma)\varphi}[\eta](\gamma g) \, .
    \end{equation}
\end{lemma}

\begin{defn}\label{def:dominate}
    Let $(H,\varphi)$ and $(S,\varphi)$ be Whittaker pairs with the same $\varphi$. We will say that $(H,\varphi)$ \emph{dominates} $(S,\varphi)$ if $H$ and $S$ commute and 
\begin{equation}\label{=domin}
        \fg_{\varphi}\cap \fg^H_{\geq 1}\subseteq \fg^{S-H}_{\geq0}\,.
    \end{equation}
\end{defn}

The following lemma provides two fundamental special cases of domination.

\begin{lemma}{\cite[Corollary 3.2.2 and Proposition 3.2.3]{Part1}}
\label{lem:domin} Let $(S,\varphi)$ be a 
 Whittaker pair. Then 
 \begin{enumerate}[(i)]
 \item $(S,\varphi)$ is dominated by a neutral Whittaker pair.
 \item If $\varphi$ is a PL-nilpotent then $(S,\varphi)$ dominates a standard Whittaker pair.
\end{enumerate}
\end{lemma}

The importance of the domination relation is due to the next three statements.
{
\begin{prop}[{\cite[Proposition 4.0.1]{Part1}}]\label{prop:domin}
Let $(H,\varphi)$ and $(S,\varphi)$ be Whittaker pairs such that $(H,\varphi)$ dominates $(S,\varphi)$, and let $\eta$ be an automorphic function with $\cF_{H,\varphi}[\eta]=0$. Then $\cF_{S,\varphi}[\eta]=0$.
\end{prop}

\begin{corollary}
    \label{cor:domin}
Let $\eta$ be an automorphic function and let $(S,\varphi)$ be Whittaker pair, with $\Gamma\varphi \notin \mathrm{WO}(\eta)$. Then $\cF_{H,\varphi}[\eta]=0$.
\end{corollary}}

\begin{theorem}[{\cite[Theorem {C}(i)]{Part1}}]\label{thm:IntTrans}
Let $\eta$ be an automorphic function on $G$, and let $\varphi\in \mathrm{WS}(\eta)$.
Let $(H,\varphi)$ and $(S,\varphi)$ be Whittaker pairs such that $(H,\varphi)$ dominates $(S,\varphi)$.
    Denote 
    \begin{equation}
        \fv:= \fg^H_{> 1}\cap \fg^{S}_{<1}, \text{ and }\, {V}:=\Exp(\fv(\A)).
    \end{equation}

            If $\fg^H_{1}= \fg^{S}_{1}=0$ then 
            \begin{equation}\label{=HfromS_easy}
                \cF_{H,\varphi}[\eta](g) = 
                \intl{V} \cF_{S,\varphi}[\eta](vg) \, dv \, .
            \end{equation}
\end{theorem}
We emphasize that the integral over $V$ is an adelic integral.

{For the next proposition, recall from \S\ref{subsec:ThB} that we say that a simple root $\alp$ is a Heisenberg root if the nilradical of the maximal parabolic subalgebra defined by $\alp$ is a Heisenberg Lie algebra. All such roots for simple (simply-laced) Lie algebras are listed in the second row of Table \ref{tab:QA} in \S \ref{subsec:ThB} above.}

\begin{prop}[{\cite[Proposition 5.1.5]{Part1}}]\label{prop:Heis}
Let
 $\alp$ be a Heisenberg  root, and let $\alp_{\max}$ denote the highest root of the component of $\fg$ corresponding to $\alp$. Let  $\Omega_{\alp}$ denote the {abelian} group obtained by exponentiation of the {abelian} Lie algebra given by the direct sum of the root spaces of negative roots $\beta$ satisfying $\langle \alp, \beta\rangle=1$.
Let $\gamma_{\alp}$ be a representative of a Weyl group element that conjugates $\alp$ to $\alp_{\max}$. Let $$\Psi_\alpha :=\{ \textnormal{ root } \eps \mid \langle \eps, \alp \rangle \leq 0, \,\eps(S_\alp) =2 \}.$$ Then
\begin{equation}
\eta(g)=\sum_{\varphi\in (\fg^*)^{S_{\alp}}_{-2}}\cF_{S_{\alp},\varphi}[\eta](g)+\sum_{\varphi\in \fg^{\times}_{-\alp}}\sum_{\omega\in \Omega_{\alp}}\sum_{ \psi\in \bigoplus_{\eps\in \Psi_\alpha}\fg^*_{-\eps}}\cF_{S_{\alp},\varphi+\psi}[\eta](\omega \gamma_{\alp}g)\,.
\end{equation}
\end{prop}

\begin{lem}[{\cite[Lemma B.0.3]{Part1}}]\label{lem:SameOrbit3}
    Let $S,Z\in \fg$ be rational semi-simple commuting elements, let $\varphi\in \fg^Z_0\cap \fg^S_{-2} $ and $\varphi'\in \fg^Z_{>0}\cap\fg^S_{-2}$. Assume that $\varphi$ is conjugate to $\varphi+\varphi'$ by ${\bf G}(\C)$. Then there exist $X\in \fg^{Z}_{>0}\cap\fg^S_{0}$ and  $v\in \Exp(\fg^{Z}_{>0}\cap\fg^S_{0})$ such that $\ad^*(X)(\varphi)=\varphi'$ and 
$v(\varphi)=\varphi+\varphi'$.
\end{lem}

\section{Proof of Theorems \ref{thm:min-rep}, \ref{thm:G0min} and \ref{thm:ntm-rep}}
\label{sec:ABC}
\setcounter{lemma}{0}

For the whole section we assume that ${\bf G}$ is split and the Dynkin diagram of $\fg$ is simply-laced, i.e. all the connected components have types $A,D,$ or $E$. As in \S\ref{subsec:ThB}, let, for any root $\delta$, $\fg^*_{\delta}$ denote the corresponding root-subspace of $\fg^*$ and $\fg^{\times}_{\delta}$ the set of non-zero elements of this subspace.

\begin{lem}\label{lem:RootConj}
If $[\fg,\fg]$ is simple then any two roots are Weyl-conjugate.
\end{lem}
\begin{proof}
Any root is Weyl-conjugate to a simple root, and any two simple roots in a connected simply-laced diagram are Weyl-conjugate.
\end{proof}

\begin{cor}\label{cor:RootPairClasses}
For any root $\delta$, any $\varphi\in \fg^{\times}_{\delta}$ lies in a minimal orbit.
\end{cor}
{
\begin{cor}\label{lem:2RootConj}
Assume that $\fg$ is simple. 
\begin{enumerate}[(i)]
\item If $\fg$ is of type $A$ or $E$ then any two pairs of orthogonal roots are Weyl-conjugate.
\item If $\fg$ is of type $D_n$ with $n\geq 5$ then any  pair of orthogonal roots is Weyl-conjugate to exactly one of the pairs $(\alp_1,\alp_3)$ and $(\alp_{n-1},\alp_n)$.
\item If $\fg$ is of type $D_4$ then any  pair of orthogonal roots is Weyl-conjugate to exactly one of the pairs $(\alp_1,\alp_3)$, $(\alp_1,\alp_{4})$ and $(\alp_{3},\alp_4)$.
\end{enumerate}
\end{cor}
\begin{proof}
In types $A$ and $E$, we apply Lemma \ref{lem:RootConj}, and assume that both pairs include the highest root.
Since the diagram consisting of roots orthogonal to the highest one is still connected, the stabilizer of the highest root acts transitively on it.

In type $D_n$ we use the standard realization of roots as 
\begin{equation}
\{\pm \eps_i\pm \eps_j,\}
\end{equation}
where $\eps_i$ denotes the unit vector in $\R^n$. The Weyl group acts by permutation of the indices, and even number of sign changes. The usual choice of simple roots is 
\begin{equation}
\alp_1:=\eps_1-\eps_2\,,\quad \dots\,,\quad \alp_{n-1}:=\eps_{n-1}-\eps_n\,,\quad \alp_n:=\eps_{n-1}+\eps_n
\end{equation}
Using reflections, we can conjugate any pair of orthogonal roots to a pair of orthogonal {positive} roots. The  pairs of orthogonal {positive} roots have one of the two forms
\begin{enumerate}
\item $(\eps_i+\eps_j,\eps_i-\eps_j) \text{ or }(\eps_i-\eps_j,\eps_i+\eps_j)$, with $i<j$.
\item  $(\eps_i\pm\eps_j,\eps_k\pm\eps_l)$ with $i<j$ and $k<l$ all distinct.
\end{enumerate}
We can conjugate any pair of type (1) to $(\alp_{n-1},\alp_n)=(\eps_{n-1}-\eps_n,\eps_{n-1}+\eps_n)$.
For $n\geq 5$, any pair of type (2) is conjugate to $(\alp_{1},\alp_3)=(\eps_{1}-\eps_2,\eps_{3}-\eps_4)$.
For $D_4$ we have two non-conjugate pairs of type (2): $(\alp_1,\alp_3)=(\eps_1-\eps_2,\eps_3-\eps_4)$ or $(\alp_1,\alp_4)=(\eps_1-\eps_2,\eps_3+\eps_4)$. 
It is easy to see that one cannot conjugate a pair of type (1) into a pair of type (2).
\end{proof}

We remark that in type $D_n$, the pairs  $(\alp_1,\alp_3)$ and $(\alp_{n-1},\alp_n)$ correspond to two distinct next-to-minimal orbits, given by the partitions  $2^41^{2n-8}$ and $31^{2n-3}$ respectively.

\begin{cor}\label{cor:2RootConjSimple}
 Any pair of orthogonal roots in $\fg$ is Weyl-conjugate to a pair of orthogonal simple roots.
\end{cor} 
\begin{proof}
If $[\fg,\fg]$ is not simple and the roots lie in different simple components this follows from Lemma \ref{lem:RootConj} by conjugating each of them to a simple root.  If the roots lie in the same component, this follows from Corollary \ref{lem:2RootConj}.
\end{proof}
}

\subsection{Proof of Theorem \ref{thm:min-rep}}

Throughout the subsection fix a simple root $\alp$.
Define $S_{\alp}\in \fh$ by $\alp(S_{\alp})=2$ and $\gamma (S_{\alp})=0$ for any other simple root $\gamma$. 
 
As mentioned in the introduction, if a Fourier coefficient $\cF_{S,\varphi}$ is a Whittaker coefficient, i.e. $N_{S,\varphi}$ is the unipotent radical of a Borel subgroup, we will denote it by $\cW_{S,\varphi}$, where we may drop the $S$ if it corresponds to a fixed choice of Borel subgroup and simple roots. In other words, we define $S_\Pi \in \fh$ by $S_\Pi(\gamma)=2$ for any simple root $\gamma$ and write $\mathcal{W}_{S_\Pi,\varphi} = \mathcal{W}_\varphi$. 
 
\begin{lem}\label{lem:1_1}
    If $\eta$ is a minimal automorphic function and $\varphi\in \fg^{\times}_{-\alp}$ then
$\cF_{S_{\alp},\varphi}[\eta] = \mathcal{W}_{\varphi}[\eta]$.
\end{lem}
\begin{proof} 
We have $\fg^S_{1}=\{0\}=\fg^{S_{\alp}}_{\geq 1}\cap \fg^S_{<1}$, which implies the lemma by Theorem \ref{thm:IntTrans} 
\end{proof}

Let $L_{\alp}$ denote the Levi subgroup of the parabolic subgroup $P_{\alp}$ of $G$.    
\begin{lem}\label{lem:WeylMin}
Any root $\delta$ with $\delta(S_{\alp})=-2$ can be conjugated to $-\alp$ using the Weyl group of $L_{\alp}$. 
\end{lem}
\begin{proof}
We can assume that $\fg$ is simple.
This statement can be proved using the language of minuscule representations, i.e., representations such that the Weyl group has a single orbit on the weights of the representation.  By {\cite[\S VIII.3]{Bou}
these are the fundamental representations corresponding to the abelian roots (see Table \ref{tab:QA}). 

It suffices to show that the representation of the Levi $L_\alpha$ on the first internal Chevalley module $V_{\alp}:=\fu_\alpha/[\fu_\alpha,\fu_\alpha]$
is minuscule. These modules are explicitly computed in \cite[\S 5]{MS}; and this can be checked case-by-case. For completeness we give a  conceptual argument.

We claim first that $V_{\alp}$ is irreducible with lowest weight $\alp$. Evidently $\alp$ is a weight of $V_{\alp}$ with multiplicity one. Also any positive root $\beta$ of $L_{\alp}$ involves only simple roots different from $\alp$, and thus $\alp-\beta$ is not a root. Hence $\alp$ is a lowest weight  of $V_{\alp}$. On the other hand, any weight of $V_{\alp}$ is of the form $\alp+\gamma$, where $\gamma$ is a sum of positive roots from $L_{\alp}$. Thus $\alp$ is the unique lowest weight of $V_{\alp}$.

The Dynkin diagram of $L_{\alp}$ is obtained from that of $G$ by removing $\alp$, and each component has exactly one simple root adjacent to $\alp$, which is easily checked to be an abelian root for the component. Thus the corresponding fundamental representations are minuscule, and thus so is their tensor product $W_{\alp}$. However, $W_{\alp}$ has highest weight $-\alp$, since $\langle -\alp,\beta\rangle$ is 1 if $\beta$ is adjacent to $\alp$ and zero otherwise. It follows that $V_{\alp}\simeq W_{\alp}^*$, and hence $V_{\alp}$ is minuscule.}
\end{proof}

\begin{cor}\label{cor:EasyMin}Let $R$ denote the  set of minimal elements in $(\fg^*)^{S_{\alp}}_{-2}$. 
\begin{enumerate}[(i)]
\item $R= (L_{\alp}\cap \Gamma)(\fg^{\times}_{-\alp})$. \label{it:Min1}
\item \label{it:Min2} $R\cap \left (\fg_{-\alp}^{\times}+\bigoplus_{\eps\in \Psi_\alpha}\fg^*_{-\eps}\right )=\fg_{-\alp}^{\times}$, where 
\begin{equation}\label{=Psia}
\Psi_\alpha :=\{ \textnormal{ root } \eps \mid \langle \eps, \alp \rangle \leq 0, \eps(S_\alp) = 2 \}\,.
\end{equation}
\end{enumerate}
\end{cor}
\begin{proof}
\eqref{it:Min1}
 Let $z$ be a generic element of $\fh$ that is 0 on $\alp$ and negative on other positive roots. Decompose  $(\fg^*)^{S_{\alp}}_{-2}=\oplus_{i=0}^k V_{k}$ by eigenvectors of $z$, with eigenvalues $0=t_0<t_1<\dots<t_k$. Note that $V_0=\fg^*_{-\alp}$. Let $X\in  (\fg^*)^{S_{\alp}}_{-2}$ be a minimal element and $X=\sum_i X_i$ its decomposition by eigenvalues of $z$.  By Lemma \ref{lem:WeylMin} we can assume, by replacing $X$ by its $L_{\alp}\cap \Gamma$-conjugate, that $X_0\neq 0$. By Lemma \ref{lem:SameOrbit3},  $X$ is conjugate to $X_0$ using $\Exp((\fl_{\alp})^{z}_{>0})\subset L_{\alp}\cap \Gamma$ .

\eqref{it:Min2} Let $Y=Y'+Y''\in R$, where $Y'\in \fg_{-\alp}^{\times}$ and $Y''\in   \bigoplus_{\eps\in \Psi_\alpha}\fg^*_{-\eps}$.
Identify $Y'$ with some $f\in \fg_{-\alp}$ using the Killing form, and complete $f$ to an $\sl_2$-triple $e,h,f$ with $e\in \fg_{\alp}$. Then $Y''\in (\fg^*)^e$, since for every root $\eps\in \Psi_{\alp},$ $\alp-\eps$ is not a root. Thus $Y$ belongs to the Slodowy slice $Y'+(\fg^*)^e$, that is transversal to the orbit of $Y'$. Since the orbit of $Y$ is minimal,  $Y'$ must lie in the same orbit and thus $Y''=0$. 
\end{proof}

{
\begin{lem}\label{lem:min2minGeo}
Let $\fl \subset \fg$ be a $\K$-Levi subalgebra, and let $\cO$ be the minimal nilpotent orbit in $\fg$. Then $\cO\cap \fl$ is either empty or the minimal orbit of $\fl$.
\end{lem}
\begin{proof}
Suppose the contrary. Let $\cO_{\fl}$ denote the minimal orbit of $\fl$. Then $\cO_{\fl}$ lies in the Zariski closure of $\cO\cap \fl$. Thus there exists an $\sl_2$ triple $(e,h,f)$ in $\fl$ such that $f\in \cO_{\fl}$, and the Slodowy slice $f+\fl^e$ to $\cO_{\fl}$ at $f$ intersects $\cO$. Namely, there exists a non-zero $X\in  \fl^e$ with $f+X \in \cO$. This contradicts the minimality of $\cO$, since $f+\fl^e$ is transversal to the orbit of $f$.
\end{proof}

}

\begin{proof}[Proof of Theorem~\ref{thm:min-rep}]
Part \eqref{it:min0} follows from Proposition \ref{prop:domin} and the minimality of $\eta$.

Part \eqref{it:min-min} follows from Corollary \ref{cor:EasyMin}\eqref{it:Min1}, and Lemmas \ref{lem:1_1} and \ref{lem:conjugation-translation}.

{
For part \eqref{it:min0phi}, suppose that there exists a Whittaker pair $(H,\psi)$ for $L_{\alp}$ with $\psi\neq 0$ such that $\cF_{H,\psi}[\cF_{S_{\alp},0}[\eta]]\neq 0$. 
Then, for $T$ big enough, we have $\cF_{H,\psi}[\cF_{S_{\alp},0}[\eta]]=\cF_{H+TS_{\alp},\psi }[\eta]$. Thus, the orbit of $\psi$ is minimal in $\fg^*$ and thus, by Lemma \ref{lem:min2minGeo} also in $\fl_{\alp}^*$. }
\end{proof}

\subsection{Proof of Theorem \ref{thm:G0min}}
Let $\eta$ be a minimal automorphic function. 
 
As above, for any simple root $\alp$ let $L_\alpha$ be the Levi subgroup of $P_\alpha$. {Let $Q_{\alp}\subset L_{\alp}$ be the parabolic subgroup  with Lie algebra $(\fl_{\alp})^{\alp^{\vee}}_{\leq 0}$.
 \begin{lemma}\label{lem:LineStab}
The stabilizer in $L_{\alp}$ of the line $\fg^*_{-\alp}$ as an element of the projective space of $\fg^*$ is $Q_{\alp}$.
\end{lemma}
\begin{proof}
For any root $\varepsilon$, $\varepsilon(\alp^{\vee})\leq 0$ if and only if $\varepsilon-\alp$ is not a root. Thus the Lie algebra of the stabilizer of $\fg^*_{\alp}$ is the parabolic subalgebra $(\fl_{\alp})_{\leq 0}^{\alp^\vee}$ of $\fl_{\alp}$. Thus the stabilizer is $Q_{\alp}$.
\end{proof}
 }

Let $\Gamma_{\alpha}:=(L_{\alp}\cap \Gamma)/({Q}_{\alp}\cap \Gamma)$.

\begin{prop}\label{prop:0_1}
Let $\alp$ be a (simple) abelian root. Then
\begin{equation}\label{=EasyMin}
\eta(g) =\cF_{S_{\alp},0}[\eta](g) +\sum_{\gamma\in \Gamma_\alpha}  \sum_{\varphi\in\fg^{\times}_{-\alp}}\cW_{\varphi}[\eta](\gamma g).
\end{equation}
\end{prop}
\begin{proof}
By definition of an abelian root, the group $U_{\alp}$ is abelian. Decompose $\eta$ into Fourier series on $U_{\alp}$. The coefficients in the Fourier series will be given by $\cF_{S_{\alp},\varphi'}[\eta]$ with $\varphi'\in(\fg^*)^{S_{\alp}}_{-2}$. Note that this coefficient vanishes unless $\varphi'$ is minimal or zero, and that by Corollary \ref{cor:EasyMin}, all minimal $\varphi'\in (\fg^*)^{S_{\alp}}_{-2}$ can be conjugated into $\fg^{\times}_{-\alp}$ using $L_{\alp}\cap \Gamma$. Thus we have
 \begin{equation}
        \begin{split}
            \eta(g) &= \sum_{\varphi'\in(\fg^*)^{S_{\alp}}_{-2}}\cF_{S_{\alp},\varphi'}[\eta](g) =\cF_{S_{\alp},0}[\eta](g) + \sum_{\gamma\in\Gamma_\alpha}  \sum_{\varphi\in \fg^{\times}_{-\alp}}\cF_{S_{\alp},\varphi}[\eta](\gamma g)\,.
        \end{split}
    \end{equation}
Lemma \ref{lem:1_1}
and the minimality of $\eta$ imply that $\cF_{S_{\alp},\varphi}[\eta](\gamma g)=\cW_{\varphi}[\eta](\gamma g)$.
\end{proof}

\begin{proof}[Proof of Theorem~\ref{thm:G0min}]
The proof is  by induction on the rank of $G$, that we denote by $n$. The base case of rank 1 group is the classical Fourier series decomposition.
For the induction step {let us show that
\begin{equation}\label{=minStep}
\eta=\cF_{S_{\beta_n},0}[\eta]+C_n[\eta]
\end{equation}

For that purpose, assume first that the root $\alp:=\beta_n$ is abelian.  By Proposition \ref{prop:0_1} we have 
\begin{equation}\label{=A1}
\eta(g) =\cF_{S_{\alp},0}[\eta](g) +\sum_{\gamma\in\Gamma_\alpha}  \sum_{\varphi\in \fg^{\times}_{\alp}}\cW_{\varphi}[\eta](\gamma g)=\cF_{S_{\alp},0}[\eta](g)+A_n[\eta](g)=\cF_{S_{\alp},0}[\eta](g)+C_n[\eta](g)\,.
\end{equation}

If $\alp:=\beta_n$ is a Heisenberg root then by Proposition \ref{prop:Heis} we have
\begin{align}\label{=DiffMinInt1}
\eta(g)&=\sum_{\varphi\in (\fg^*)^{S_{\alp}}_{-2}}\cF_{S_{\alp},\varphi}[\eta](g)+\sum_{\varphi\in \fg^{\times}_{-\alp}}\sum_{\omega\in \Omega_{\alp}}\sum_{ \psi\in \bigoplus_{\eps\in \Psi_\alpha}\fg^*_{-\eps}}\cF_{S_{\alp},\varphi+\psi}[\eta]({\omega \gamma_{n}} g)\nonumber\\
&=\cF_{S_{\alp},0}[\eta](g)+A_n[\eta](g)+B_{n}[\eta](g)=\cF_{S_{\alp},0}[\eta](g)+C_n[\eta](g)\,.
\end{align}

Formula \eqref{=minStep} in now established. 
By Theorem \ref{thm:min-rep}\eqref{it:min0phi},  $\cF_{S_{\alp},0}[\eta]$ is  a minimal automorphic function on $L_{\alp}$. 
As before, let $S_\Pi\in \fh$ denote the element that is 2 on all positive roots. Note that for any $\varphi\in (\fl_{\alp}^*)^{S_\Pi}_{-2}$, 
we have $\cW'_{\varphi}[\cF_{S_{\alp},0}[\eta]]=\cW_{\varphi}[\eta]$ where the prime denotes a Whittaker coefficient with respect to $L_\alpha$.
This implies that $C_i'[\cF_{S_{\alp},0}[\eta]]=C_i$ for any $i<n$. 
From the induction hypothesis and \eqref{=minStep} we obtain
\begin{equation}\label{=G01easy}
\eta(g) =\cF_{S_{\beta_n},0}[\eta]+C_n=\cW_{0}[\eta](g) +\sum_{i=1}^{n-1}{C}_i+C_n=\cW_{0}[\eta](g) +\sum_{i=1}^{n}{C}_i\,.
\end{equation}}

 \end{proof}

\subsection{Proof of Theorem \ref{thm:ntm-rep}}
\label{sec:pfC}

Suppose that $\operatorname{rk}(\fg)>2$.
Let $\eta$ be a next-to-minimal automorphic function. 
Let $\alp$ be a 
simple root and let $\psi\in \fg^{\times}_{-\alp}$.

{
\begin{lemma}\label{lem:RootOrbs}
Let $\gamma\neq \alp$ be a positive root, and let $\varphi'\in\fg^{\times}_{-\gamma}$. Let $\cO$ denote the orbit of $\psi+\varphi'$. Then 
$\cO$  is minimal if 
 $\langle \alp, \gamma \rangle >0$,   
$\cO$ is next-to-minimal if $\langle \alp, \gamma \rangle =0$ and $\cO$ is neither minimal nor next-to-minimal if $\langle \alp, \gamma \rangle <0$.
\end{lemma}

\begin{proof}
By Lemma \ref{lem:MinProd} we can assume that $[\fg,\fg]$ is simple. 
Let $\fh' \subset \fh$ be the simultaneous kernel of $\alpha$ and $\gamma$, and let $\fl$ be its centralizer in $\fg$. Then $\fh'$ has codimension at most 2 in $\fh$, hence $\fl$ is a Levi subalgebra of semisimple rank $\leq 2$ whose roots include $\alpha$ and $\gamma$. Note that $\cO\cap \fl$ is a principal nilpotent orbit in $\fl$.
By a straightforward rank 2 calculation we see that $\fl$ has type $A_1$ if  $\langle \alp, \gamma \rangle >0$, type $A_1\times A_1$ if  $\langle \alp, \gamma \rangle =0$ and type $A_2$  if $\langle \alp, \gamma \rangle <0$. The lemma follows now from Lemma \ref{lem:BCsmall}.
\end{proof}
}

\begin{notn}\label{notn:SNTM}
Denote by $\Delta_{\alp}$ the set of simple roots orthogonal to $\alp$.
Define $S\in \fh$ to be 0 on any simple root $\eps\in\Delta_{\alp}$, and 2 on other simple roots. 
\end{notn}

\begin{proposition}\label{prop:1_2}
We have $\cF_{S_{\alp},\psi}[\eta]=\cF_{S,\psi}[\eta]$ {for any $\psi\in \fg^*_{-\alp}$}.
\end{proposition}
\begin{proof}
Note that $S_{\alp}$ dominates $S$, and that $\fg^{S_{\alp}}_1=\fg^S_1=\fg^{S_{\alp}}_{>1}\cap\fg^S_{<1}=\{0\}.$ Thus the statement follows from Theorem \ref{thm:IntTrans}.
\end{proof}

Let $G'\subset G$ be the Levi subgroup given by $\Delta_{\alp}$.
\begin{prop}\label{prop:2to1}
The restriction
$\cF_{S,\psi}[\eta]|_{G'}$ is a minimal or a trivial automorphic function on $G'$.
\end{prop}
For the proof we will need the following geometric lemma.
\begin{lem}\label{lem:minmin}
Let $\varphi'\in {\lie g'}^*$ be nilpotent such that $\varphi'+\psi$ belongs to a next-to-minimal orbit in $\fg^*$. Then $\varphi'$ belongs to the minimal orbit of ${\lie g'}^*$.
\end{lem}
\begin{proof}
Clearly $\varphi'\neq 0$. If the orbit of $\varphi'$ is not minimal then it belongs to the Slodowy slice of some element $\psi'$ of the minimal orbit of $(\fg')^*$. 
Then $\varphi'+\psi'$ belongs to a next-to-minimal orbit of $\fg^*$, and $\varphi'+\psi$ belongs to the Slodowy slice of $\varphi'+\psi'$ and thus lies in an orbit that is higher than next-to-minimal. 
\end{proof}

\begin{proof}[Proof of Proposition \ref{prop:2to1}]
Let $Z:=S-S_{\alp}$. Note that $Z$ vanishes on simple roots in $\Delta_{\alp}$ and on $\alp$ and is 2 on other simple roots. Suppose that there exists a Whittaker pair $(H,\psi')$ with $\psi'\neq 0$ such that $\cF_{H,\psi'}[\cF_{S,\psi}[\eta]]\neq 0$. 
Then, for $T$ big enough, we have $$\cF_{H,\psi'}[\cF_{S,\psi}[\eta]]=\cF_{S+TZ+H,\psi+\psi'}[\eta].$$ By Proposition~\ref{prop:domin} and Lemma~\ref{lem:minmin}, $\psi$ lies in the minimal orbit of $\lie g'^*$.
\end{proof}

\begin{lemma}[See \S \ref{subsec:Geo} below]\label{lem:AppGeo}
For any next-to-minimal element $\varphi\in (\fg^*)^{S_{\alp}}_{-2}$, there exist $\gamma_0\in L_{\alp}\cap \Gamma$ and a positive root $\beta$ orthogonal to $\alp$ {\it s.t.} $\Ad^{*}(\gamma_{0})\varphi \in  \fg^{\times}_{-\alp}+ \fg^{\times}_{-\beta}\subset \fg^*_{-\alp}\oplus \fg^*_{-\beta}$.
\end{lemma}

{
\begin{remark} 
The above lemma only establishes that any next-to-minimal $\varphi$ can be mapped to two orthogonal root spaces by $L_\alpha\cap \Gamma$. However, the action of $L_{\alpha}\cap\Gamma$ is often even transitive on $(\fg^*)_{-2}^{S_\alpha}$, giving a single orbit. One can show that this happens in all cases except for:
\begin{itemize}
\item $A_3$ and node $\alpha_2$.
\item $D_4$ and nodes $\alpha_1$, $\alpha_3$, $\alpha_4$ (all related by triality).
\item $D_n$ and when the two orthogonal roots $(\alpha,\beta)$ are Weyl conjugate under $D_n$ to $(\alpha_{n-1},\alpha_n)$, see Corollary~\ref{lem:2RootConj}, corresponding to the orbit $31^{2n-3}$. This happens for $n\geq 4$ always for node $\alpha_1$ as well as for nodes $\alpha_ i$ with $2\leq i \leq n-2$ if $\varphi$ belongs to that orbit.
\end{itemize}
For instance, for $A_3$ and node $\alpha_2$ one has that next-to-minimal are $\varphi\in \fg_{-\alpha_1}^\times + \fg_{-\alpha_3}^\times$. The torus element for node $i$ scales elements in $\fg_{-\alpha_i}^\times$ by rational squares ($i=1,3$) while keeping the other space unchanged. The torus element for node $2$ scales both spaces by rational elements in the same way and so one cannot use the torus action in $L_{\alpha_2}\cap \Gamma$ to arrive at a unique representative. The other cases can be seen to reduce to the same phenomenon.

For $A_n$ with $n\geq 4$ and all exceptional cases there is a unique rational representative for next-to-minimal nilpotents in $(\fg^*)_{-2}^{S_\alpha}$.
\end{remark}
}

\begin{proof}[Proof of Theorem~\ref{thm:ntm-rep}] 
Part \eqref{itm:larger} follows from Proposition \ref{prop:domin}, since $\eta$
is a next-to-minimal {function}.

For part \eqref{itm:ntm}, by Lemma \ref{lem:AppGeo} 
we may assume $\varphi\in \fg_{-\alp}^{\times}+ \fg_{-\beta}^{\times}$ for some positive roots $\beta$ orthogonal to $\alp$. By Corollary \ref{cor:2RootConjSimple}, one can conjugate the pair of roots $(\alp,\beta)$  to a pair of orthogonal simple roots $(\alp',\alp'')$, using the Weyl group. Let $\fa$ be the joint kernel of $\alp'$ and $\alp''$ in $\fh$, and let $z\in \fa$ be a generic rational semi-simple element. Let $S_T:=\alp'^{\vee}+\alp''^{\vee}+Tz$ for $T\gg 0\in \Q$, where $\alpha'^{\vee}$ and $\alpha''^{\vee}$ are the dual co-roots. 
Since no linear combination of $\alp'^{\vee}$ and $\alp''^{\vee}$ lies in $\fa$, $S_T$ is a generic element of $\fa$  and thus for $T$ big enough, $\fg^{S_T}_{\geq 2}$ is a Borel subalgebra of $\fg$ that contains $\fh$. Thus it is conjugate under the Weyl group to our fixed Borel subalgebra. The statement follows now from Theorem \ref{thm:IntTrans}.
We note that different choices for $z$ may give $V$ of different dimensions.

For part \eqref{it:Pt2}, 
Proposition \ref{prop:1_2} implies $\cF_{S_{\alp},\psi}[\eta]=\cF_{S,\psi}[\eta]$.
By Proposition \ref{prop:2to1}, $\eta':=\cF_{S,\psi}[\eta]|_{G'}$ is a minimal or a trivial automorphic function on $G'$. The statement follows now from Theorem \ref{thm:G0min} applied to $\eta'$ together with the fact that its Whittaker coefficients, obtained by integration over the maximal unipotent subgroup $N'=N\cap G'$, are, in fact, equal to the Whittaker coefficients $\cW_{\varphi+\psi}[\eta]$ due to the extra integral present in the definition of $\eta'$.

{Part \eqref{it:ntm0phi} is proven very similarly to Theorem \ref{thm:min-rep}\eqref{it:min0phi}.}
\end{proof}

{

\subsection{Proof of Lemma \ref{lem:AppGeo}}\label{subsec:Geo}

Let $\alp$ be a simple root. We assume that $\fg$ is simple.

Note that in simply-laced root systems orthogonal roots are strongly orthogonal, and thus the sum of two roots is a root if and only if they have scalar product $-1$. 
Also, for any two non-proportional roots, the scalar product is in $\{-1,0,1\}$.
 For any root $\eps$ we  denote by $\eps^\vee$ the coroot given by the scalar product with $\eps$.

\begin{notn}\label{not:Fa}
Denote 
$z:={\alp}^\vee-S_{\alp}$ and $\fu_z:=(\fl_{\alp})^{z}_{{>}0}$, {and $U_{z}:=\Exp(\fu_z)\subset L_{\alp}\cap \Gamma$}.
\end{notn}
Note that $\fu_z=(\fl_{\alp})^{{\alp}^\vee}_{1}$ and $\lie g^\times_{-\alp} \subset (\lie g^*)^z_0$.
\begin{lem}\label{lem:minV}
Let $\varphi\in \fg^{\times}_{-\alp}$ and $\psi\in (\fg^*)^{S_{\alp}}_{-2}\cap (\fg^*)^{{\alp}^\vee}_{-1} \subset (\lie g^*)^z_{1}$. Then there exists $v\in U_z$ such that $\Ad^*(v) \varphi =\varphi+\psi$.
\end{lem}
\begin{proof}
\begin{enumerate}[{Case} 1.]
\item  $\psi\in \fg^{*}_{-\eps}$ for some $\eps$:\\
By the assumption that $\psi \in (\fg^*)^{{\alp}^\vee}_{-1}$ and Lemma~\ref{lem:RootOrbs},
$\varphi+\psi$ is conjugate to $\varphi$ over $\C$. By Lemma~\ref{lem:SameOrbit3}, there exists $v\in U_z$ such that $\Ad^*(v)\varphi=\varphi+\psi$. 

\item General: \\
We can assume $\psi\neq 0$. 
Let $H\in \fh$ be a generic element that has {distinct} negative integer values on all positive roots. 
Note that $\fu_z \subset \lie g^H_{>0}${ and $\psi \in (\lie g^*)^H_{>0}$}.
Decompose $\psi=\sum_{i>0}{\psi_i}$, where $\psi_i\in (\fg^*)^{H}_{i}$. We prove the lemma by descending induction on the minimal $j$ for which $\psi_j\neq 0$. The base of the induction {is $j$ that equals the maximal eigenvalue of $\ad^*(H)$. In this case $\psi=\psi_j$ and we are in }Case 1. For the induction step, let $j$ be minimal with $\psi_j\neq 0$. By Case 1, there exists $v_{1}\in U_{z}$ with $\Ad^*(v_1)\varphi=\varphi-\psi_j$. Then $\Ad^*(v_1)(\varphi+\psi)=\varphi+\sum_{i>j}\psi'_i,$
for some $\psi'_i\in (\fg^*)^{H}_{i}$. By the induction hypothesis, there exists $v_2\in U_{z}$ such that $\Ad^*(v_2)\varphi=\Ad^*(v_1)(\varphi+\psi)$. Take $v:=v_1^{-1}v_2$.
\end{enumerate}
\end{proof}

\begin{proof}[Proof of Lemma \ref{lem:AppGeo}]
Let $\varphi\in (\fg^*)^{S_{\alp}}_{-2}$ be next-to-minimal.
Decompose $\varphi=\sum_{\eps}\varphi_{\eps},$ where $\varphi_{\eps}\in \fg^*_{-\eps}$. Let $F:=\{\eps\, \vert \, \varphi_{\eps}\neq 0\}.$  By Lemma \ref{lem:WeylMin}, we can assume $\alp\in F$. Using Lemma \ref{lem:minV}, we can assume that for any other $\eps \in F$ we have $\langle \alp,\eps \rangle \leq 0$, {\it i.e.} $F\subset \{\alp\}\cup \Psi_{\alp}$, where $\Psi_{\alp}$ is as in \eqref{=Psia}, namely
\begin{equation}
\Psi_\alpha =\{ \text{ root } \eps \mid \langle \eps, \alp \rangle \leq 0, \eps(S_\alp) = 2 \}\,.
\end{equation} 

Assume first that there exists $\beta\in F$ with $(\alp,\beta)=0$, and let $Z:=\alp^\vee+\beta^\vee-S_{\alp}$. Then 
\begin{equation}\label{=ZGeo}
\alp(Z)=\beta(Z)=0,\text{ and }\eps(Z)<0 \text{ for any }\eps\in F\setminus\{\alp,\beta\}.
\end{equation}
 Indeed, $$\eps(Z)=\langle \alp,\eps \rangle +\langle\beta,\eps\rangle-2\leq 0+1-2=-1\,.$$
 By \eqref{=ZGeo}, we see that $\varphi_{\alp}+\varphi_{\beta}$ lies in the closure of the complex orbit $\cO$ of $\varphi$. Now, by Lemma \ref{lem:RootOrbs}, $\varphi_{\alp}+\varphi_{\beta}$ is next-to-minimal, and thus lies in $\cO$.
Thus, Lemma \ref{lem:SameOrbit3} and \eqref{=ZGeo} imply that $\varphi$ is conjugate to  $\varphi_{\alp}+\varphi_{\beta}$ under $L_{\alp}\cap \Gamma$.

Let us now show that $\beta\in F$ with $\langle\alp,\beta\rangle=0$ indeed exists.
Assume the contrary, {\it i.e. } $(\alp,\eps)=-1$ for all $\eps \in F$. 
{Note that $F$ is not empty, since $\varphi$ is not minimal.}
Pick any $\omega\in F$ and let $Z':=\alp^\vee+\omega^\vee-S_{\alp}/2$. Then
\begin{equation}\label{=ZpGeo}
\alp(Z')=\omega(Z')=0,\text{ and }\eps(Z')<0 \text{ for any }\eps\in F\setminus\{\alp,\omega\}.
\end{equation}
 Indeed, $$\eps(Z')=\langle\alp,\eps\rangle+\langle\omega,\eps\rangle-1\leq -1+1-1=-1\,.$$ 
 By \eqref{=ZGeo}, we see that $\varphi_{\alp}+\varphi_{\omega}$ lies in the closure of the complex orbit $\cO$ of $\varphi$, and thus is minimal or next-to-minimal. This contradicts Lemma \ref{lem:RootOrbs} since $\langle\alp,\omega\rangle<0$.

Thus there exists $\beta\in F$ with $\langle\alp,\beta\rangle=0$, and as we showed above $\varphi$ is conjugate to  $\varphi_{\alp}+\varphi_{\beta}$ under $L_{\alp}\cap \Gamma$.
\end{proof}
}

\section{Proof of Theorems \ref{thm:ntm-rep2}, \ref{thm:ntmFull}, \ref{thm:ntmFullSimple} and \ref{thm:ntmNoCusp}}
\label{sec:D}
\setcounter{lemma}{0}

Let $\alp$ be a nice root.
Denote by ${R}$ the set of minimal elements in $ (\mathfrak{g^*})^{S_{\alp}}_{-2}$ and by $X$ the set of next-to-minimal elements in $ (\mathfrak{g^*})^{S_{\alp}}_{-2}$.
Let $\alp_{\max}$ be the highest root of the component of $\fg$ that includes $\alp$. {Recall that $\delta_{\alp}$ denotes $\alp_{\max}$ if $\alp$ is an abelian root, and denotes $\alp_{\max}-\alp-\beta$, where $\beta$ is the only simple root non-orthogonal to $\alp$, if $\alp$ is a nice Heisenberg root. Denote $\delta:=\delta_{\alp}$. See \S \ref{subsec:PfPropFab} below for more details on this $\delta$ in the Heisenberg case.}

We will use the following geometric proposition{s}, that we will prove in \S \ref{subsec:PfPropFab} below.

\begin{proposition}\label{prop:nice}
\begin{enumerate}[(i)]
\item \label{it:nice0}{If $\alp$ is abelian and }$\langle \alp, \alp_{\max}\rangle>0$ then $X$ is empty. 
\item \label{it:niceTran} {If $\langle \alp, \alp_{\max}\rangle=0$ or if $\alp$ is a nice Heisenberg root then} 
$X=(L_{\alp}\cap \Gamma)(\fg^{\times}_{-\alp}+ \fg^{\times}_{-\delta})$.
\end{enumerate}
\end{proposition}
Note that this implies  that at most one next-to-minimal orbit can intersect $X$.

{For the next proposition we assume that either  $\langle \alp, \alp_{\max}\rangle=0$ or $\alp$ is a nice Heisenberg root. Recall that in these cases $R_{\alp}$ denotes the parabolic subgroup of $L_{\alp}$  with Lie algebra $(\fl_{\alp})^{\delta}_{\leq 0}$, and let $RQ_{\alp}=R_{\alp}\cap Q_{\alp}.$ Denote further by  $St_{\alp}$ the stabilizer in $L_\alp\cap \Gamma$ of the plane $\fg^*_{-\alp}\oplus\fg^*_{-\delta}$, as an element of the  Grassmanian of planes in $\fg^*$. 
\begin{prop}\label{prop:RQ}
$RQ_{\alp}\cap \Gamma$ is a subgroup of $St_{\alp}$ of index two.
\end{prop}

}

\subsection{{Proof of Theorem \ref{thm:ntm-rep2}}}

Let $\eta$ be a next-to-minimal automorphic function on $G$.   

Suppose first that $\alp$ is an abelian root, {\it i.e.} the nilradical $U_{\alp}$ of the maximal parabolic $P_{\alp}$ is abelian. 
Using Fourier transform on $U_{\alp}$ we obtain 
\begin{align}\label{=ntm1easy}
\eta(g) = \mathcal{F}_{S_{\alp},0}[\eta](g)  + \sum_{\varphi\in  {R}}\mathcal{F}_{S_\alpha, \varphi}[\eta](g) + \sum_{\varphi\in X} \mathcal{F}_{S_{\alp}, \varphi}[\eta](g)\,.
\end{align}

By Corollary \ref{cor:EasyMin}, ${R}=(L_{\alp}\cap \Gamma)(\fg^{\times}_{-\alp})$. {By Lemma \ref{lem:LineStab}, $Q_{\alp}$ is the stabilizer in $L_{\alp}$ of the line $\fg_{-\alp}^*$ (as a point in the projective space of $\fg^*$).}
Thus 
\begin{equation}\label{=Q}
\sum_{\varphi \in {R}}\mathcal{F}_{S_\alpha, \varphi}[\eta](g)=\sum_{\gamma \in \Gamma_{\alp}}\sum_{ \varphi_0\in \fg^{\times}_{-\alp}}\mathcal{F}_{S_\alpha, \varphi_0}[\eta](\gamma g)\,,
\end{equation}
{where $\Gamma_{\alp}$ denotes the quotient of $L_{\alp}\cap \Gamma$ by} {$Q_{\alp}\cap \Gamma$.}
If $\langle \alp, \alp_{\max}\rangle>0$  then by Proposition \ref{prop:nice}, $X$ is empty. This implies part (i) of Theorem \ref{thm:ntm-rep2}.
Let us now assume $\langle \alp, \alp_{\max}\rangle= 0$  and prove part (ii) of Theorem \ref{thm:ntm-rep2}.  
By Proposition \ref{prop:nice}, $X=(L_{\alp}\cap \Gamma)(\fg^{\times}_{-\alp}+ \fg^{\times}_{-\delta})$. 
{Recall that we denote  by $\Lambda_{\alp}$ the quotient of $L_{\alp}{\cap \Gamma}$ by $RQ_{\alp}\cap \Gamma$.
By Proposition \ref{prop:RQ} we have }
 \begin{equation}\label{=DAbNTM}
 \sum_{\varphi\in X} \mathcal{F}_{S_{\alp}, \varphi}[\eta](g)={\frac{1}{2}}\sum_{\gamma\in \Lambda_{\alp}}\sum_{\varphi\in \fg^{\times}_{-\alp}}\sum_{\psi\in \fg^{\times}_{-\alp_{\max}}}\cF_{S_{\alp},\varphi+\psi}[\eta](\gamma g)\,.
 \end{equation}
From \eqref{=ntm1easy}, \eqref{=Q} and \eqref{=DAbNTM}, we obtain

 \begin{equation}\label{=DAb}
 \eta(g)= \mathcal{F}_{S_{\alp},0}[\eta]+
 \sum_{\gamma \in \Gamma_{\alp}}\sum_{\varphi\in \fg^{\times}_{-\alp}}\mathcal{F}_{S_\alpha, \varphi}[\eta](\gamma g)  +{\frac{1}{2}}
\sum_{\gamma\in \Lambda_{\alp}}\sum_{\varphi\in \fg^{\times}_{-\alp}}\sum_{\psi\in \fg^{\times}_{-\alp_{\max}}}\cF_{S_{\alp},\varphi+\psi}[\eta](\gamma g){=\cA\,,}
 \end{equation}
as required.

Suppose now that $\alp$ is a nice Heisenberg root.
{Let $\gamma_{\alp}$ be a representative of the Weyl group element  $s_\alp s_{\alp_{\max}}s_{\alp},$ where $s_{\alp}$ and $s_{\alp_{\max}}$ denote the corresponding reflections. Since $\langle \alp, \alp_{\max}\rangle =1$, }$\gamma_{\alp}$ conjugates $\alp$ to $\alp_{\max}$. Thus, by Proposition \ref{prop:Heis}, \begin{equation}\label{=DiffNTMInt1}
\eta(g)=\sum_{\varphi\in (\fg^*)^{S_{\alp}}_{-2}}\cF_{S_{\alp},\varphi}[\eta]({g})+\sum_{\varphi\in \fg^{\times}_{-\alp}}\sum_{\omega\in \Omega_{\alp}}\sum_{ \psi\in \bigoplus_{\eps\in \Psi_\alpha}\fg^*_{-\eps}}\cF_{S_{\alp},\varphi+\psi}[\eta](\omega \gamma _{\alp}g)\,.
\end{equation}
We call the first sum the \emph{abelian} term, and the second sum the \emph{non-abelian} term. In the same way as above we obtain 
{
 \begin{align}\label{=DAb2}
\sum_{\varphi\in (\fg^*)^{S_{\alp}}_{-2}}\cF_{S_{\alp},\varphi}[\eta](g) =\cA
  \end{align}    

}

To determine the non-abelian term we will need a further geometric statement.
{Recall that $M_{\alp}\subset L_{\alp}$ denotes} the Levi subgroup generated by the roots orthogonal to $\alp$. {Note that $M_{\alp}$ is the standard Levi subgroup of the parabolic $Q_{\alp}$ of $L_{\alp}$. 
\begin{lemma}\label{lem:Malp}
The group $M_{\alp}\cap R_{\alp}\cap \Gamma$ is the stabilizer in $M\cap \Gamma$ of the line $\fg^*_{-\delta}$, and of the plane $\fg^{*}_{-\alp}\oplus \fg^*_{-\delta}$.
\end{lemma}
\begin{proof}
The first assertion follows from Lemma \ref{lem:LineStab} applied to the root $\delta$. The second one follows from Proposition \ref{prop:RQ}, since $M_{\alp}\cap R_{\alp}$ is a parabolic subgroup of $M_{\alp}$.
\end{proof}}

Denote by $\Chi$ the set of next-to-minimal elements in $\fg^{\times}_{-\alp}+ \bigoplus_{\eps\in \Psi_{\alp}}\fg^*_{-\eps}.$

\begin{prop}[See \S \ref{subsec:PfPropFab} below]\label{prop:Fab}
$\Chi=(M_{\alp}\cap \Gamma)(\fg^{\times}_{-\alp}+ \fg^{\times}_{-\delta})$. 
\end{prop}

{Recall that $\cM_{\alp}$ denotes}
the quotient of $M_{\alp}\cap \Gamma$ {by $M_{\alp}\cap R_{\alp}\cap \Gamma$}.
 By Theorem \ref{thm:ntm-rep}\eqref{itm:larger}, Proposition \ref{prop:Fab}, Lemma \ref{lem:Malp}  and {Corollary \ref{cor:EasyMin}\eqref{it:Min2}} 
 we have, for any $\omega \in \Omega_\alp$,
\begin{multline}\label{=NA}
\sum_{\varphi\in \fg^{\times}_{-\alp}}\sum_{ \psi\in \bigoplus_{\eps\in \Psi_\alpha}\fg^*_{-\eps}}\cF_{S_{\alp},\varphi+\psi}[\eta](\omega sg)=\\\sum_{\varphi\in \fg^{\times}_{-\alp}}\cF_{S_{\alp},\varphi}[\eta](\omega sg)+\sum_{\gamma'\in \cM_{\alp}} \sum_{\varphi\in \fg^{\times}_{-\alp}}\sum_{\psi\in \fg^{\times}_{-\delta}}\cF_{S_{\alp},\varphi+\psi}[\eta](\gamma '\omega \gamma_{\alp} g)\,.
\end{multline}
From (\ref{=DiffNTMInt1}), (\ref{=DAb2}) and  (\ref{=NA}), we obtain
{
\begin{multline}\label{=D}
\eta{(g)}= \cA+
\sum_{\omega\in \Omega_\alp}\left(\sum_{\varphi\in \fg^{\times}_{-\alp}}\cF_{S_{\alp},\psi}[\eta](\omega \gamma_{\alp} g)+\sum_{\gamma'\in \cM_{\alp}}\sum_{\varphi\in \fg^{\times}_{-\alp}}\sum_{\psi\in \fg^{\times}_{-\delta}}\cF_{S_{\alp},\varphi+\psi}[\eta](\gamma '\omega \gamma_{\alp} g)\right ),
\end{multline}}
as required.
 \proofend

\subsection{Proof of Theorems \ref{thm:ntmFull}, \ref{thm:ntmFullSimple} and \ref{thm:ntmNoCusp}}\label{subsec:PfEFG} \begin{proof}[Proof of Theorem \ref{thm:ntmFull}]
{We proceed by induction on the rank of $\fg$. The base case is rank one, that has no next-to-minimal forms, and the statement vacuously holds.
For the induction step, }
let $\eta$ be a next-to-minimal automorphic function. Let $I=\{\beta_1,\dots,\beta_n\}$ be a convenient quasi-abelian enumeration of the roots of $\fg$. {Denote $\alp:=\beta_n$.}
Theorem \ref{thm:ntm-rep} provides the expressions for all the terms in the right-hand side of the expressions in Theorem \ref{thm:ntm-rep2}, except the constant term. 
{By Theorem \ref{thm:ntm-rep}\eqref{it:ntm0phi},} the restriction of the constant term $ \mathcal{F}_{S_{\alp},0}[\eta](g)$ to the Levi subgroup $L_{\alp}$ is next-to-minimal or minimal or trivial. Thus, we can obtain the expressions for the constant term by Theorem \ref{thm:G0min} and the induction hypothesis.
{Applying  Theorem \ref{thm:ntm-rep}\eqref{it:Pt2} to $\cF_{S_{\alp},\varphi}$ for any $\varphi\in \fg^{\times}_{-\alp}$, we get, in the notation of Theorem \ref{thm:ntm-rep}, $\gamma_0=1$, $\psi=\varphi$, $C^{\psi}_j[\eta]=A^{\varphi}_j[\eta]$ for all $1\leq j\leq n-1$, and}
\begin{equation}
\mathcal{F}_{S_\alpha, \varphi}[\eta](\gamma g)=\cW_{\varphi}[\eta](\gamma g)+\sum_{j=1}^{n-1}A^{\varphi}_j[\eta](\gamma g)\,.
\end{equation}
Thus
\begin{equation}\label{=DAbMin}
\sum_{\gamma \in \Gamma_{n}}\sum_{ \varphi\in \fg^{\times}_{-\alp}}\mathcal{F}_{S_\alpha, \varphi}[\eta](\gamma g)=A_n+\sum_{ j\bot n}A_{nj}\,.
\end{equation}

Further, for any $\varphi\in \fg^{\times}_{-\alp}$ and   
$\psi\in \fg^{\times}_{-\alp_{\max}}$, Theorem \ref{thm:ntm-rep}(\ref{itm:ntm})  provides an expression for {$\cF_{S_\alpha,\varphi+\psi}[\eta]$}.
This expression implies 
\begin{equation}\label{=DAbNTMCor}
\sum_{\gamma\in \Lambda_{\alp}}\sum_{\varphi\in \fg^{\times}_{-\alp}}\sum_{\psi\in \fg^{\times}_{-\alp_{\max}}}{\cF_{S_\alpha,\varphi+\psi}[\eta](\gamma g)}=A_{nn} \, .
\end{equation}
 Assume first that $\alp:=\beta_n$ is an abelian root. Then, using Theorem \ref{thm:ntm-rep2},  (\ref{=DAbMin}), (\ref{=DAbNTMCor}) and the induction hypothesis, we obtain
\begin{equation}\label{=DAbCor}
\eta= \mathcal{F}_{S_{\alp},0}[\eta]+A_n+\sum_{j\bot n}A_{nj}+A_{nn}=\cW_{0}[\eta]+\sum_{i=1}^n(A_i+A_{ii}+\sum_{j<i, j\bot i}A_{ij}),
\end{equation}
as required. 

Suppose now that $\alp$ is a nice Heisenberg root. Then we need to add the expressions for the non-abelian term in  \eqref{=D}. These are also provided by Theorem \ref{thm:ntm-rep}. Namely, 
\begin{align}
\label{=corNTMNonAb1}
\sum_{\omega\in \Omega_n} \sum_{\varphi\in \fg^{\times}_{-\alp}}\cF_{S_{\alp},\psi}[\eta](\omega \gamma_{{n}}g)&=B_n+\sum_{j \bot n}B_{nj}\,,\\ 
\label{=corNTMNonAb2}
\sum_{\omega\in \Omega_{{n}}}\sum_{\gamma'\in \cM_{\alp}}\sum_{\varphi\in \fg^{\times}_{-\alp}}\sum_{\psi\in \fg^{\times}_{-\delta_\alp}}\cF_{S_{\alp},\varphi+\psi}[\eta](\gamma '\omega \gamma_{{n}} g)&=B_{nn}\,.
\end{align}
The theorem follows now from Theorem \ref{thm:ntm-rep2}, and   (\ref{=DAbMin})--(\ref{=corNTMNonAb2}).
\end{proof}
Theorem \ref{thm:ntmFullSimple}  follows in a similar way, but without using the induction, and  omitting some terms that vanish.
  
{
\begin{proof}[Proof of Theorem \ref{thm:ntmNoCusp}]
{Suppose the contrary. Embed $\pi$ into the cuspidal spectrum and let } $\eta\neq0\in \pi$. By Lemma, \ref{lem:GlobLoc}, $\eta$ is either minimal or next-to-minimal. If $\fg$ has a component of type $E_8$ we let $G'\subset G$ be the subgroup corresponding to this component. Otherwise we let $G':=G$. Let $\eta'$ be the restriction of $\eta$ to $G'$. Note that $\eta'$ is still  minimal or next-to-minimal, and that it is cuspidal in the sense that the constant term of $\eta'$ with respect to the unipotent radical of any proper parabolic subgroup of $G'$ vanishes. Thus, for any two simple roots $\eps_1,\eps_2$ and any $\varphi \in \fg^*_{\eps_1}\oplus\fg^*_{\eps_2}$,
the Whittaker coefficient $\cW_{\varphi}[\eta']$ vanishes identically. 
Since all the terms in the right-hand sides of Theorems \ref{thm:G0min} and \ref{thm:ntmFull} is obtained from such Whittaker coefficients by summation, integration, and shift of the argument, we obtain from those theorems that $\eta'$ vanishes identically. This implies $\eta(1)=0$. Replacing $\eta$ in the argument above by its right shifts, we obtain $\pi=0${, reaching  a contradiction}.
\end{proof}
}
\subsection{Proof of geometric propositions }\label{subsec:PfPropFab}

{In this subsection we assume that $\fg$ is simple, since for  Propositions \ref{prop:nice}, \ref{prop:RQ}, and \ref{prop:Fab} it is enough to consider this case.}

{
\subsubsection{Proof of Proposition \ref{prop:RQ}}
\label{subsec:PfPropStab}

\begin{lemma}\label{lem:wLa}
There exists $w$ in the Weyl group of $L_{\alp}$ such that $w^2=1$ and $w(\alp)=\delta_{\alp}$.
\end{lemma}
\begin{proof}
We can assume that $\fg$ is simple.
If $\alp$ is abelian, we take $w$ to be $w_0$, where $w_0$ is the longest element in the Weyl group of $L_{\alp}$. Since $\alp$ is the lowest weight of the first internal Chevalley $L_{\alp}$-module $\fn_{\alp}$, and $\delta_{\alp}=\alp_{\max}$ is its highest weight, $w_0(\alp)=\delta_{\alp}$.

If $\alp$ is Heisenberg, we take $w$ to be $s_{\beta}w_0$, where $\beta$ is the only root attached to $\alp$. In this case, the highest weight of $\fn_{\alp}$ is $\alp_{\max}-\alp$, while the lowest weight is still $\alp$. Thus, $w_0(\alp)=\alp_{\max}-\alp$. Since $\alp$ is a Heisenberg root, $\beta$ is orthogonal to $\alp_{\max}$. Thus, $s_{\beta}(\alp_{\max}-\alp)=\alp_{\max}-\alp-\langle\alp_{\max}-\alp,\beta\rangle\beta=\alp_{\max}-\alp-\beta=\delta_{\alp}$.
To prove that $w$ is an involution, we will show that $w_0(\beta)=-1$. To see this, we apply the well known fact that $-w_0$ is a graph automorphism of the Dynkin diagram. For $\fg$ of type $E_8$, we have $\alp=\alp_8$, $L_{\alp}$ is of type $E_7$, and the Dynkin diagram has no automorphisms. For $\fg$ of type $E_7$, we have $\alp=\alp_1$, $L_{\alp}$ is of type $D_6$, and $w_0$ is known to be $-1$. In the remaining case of $\fg$ of type $E_6$, we have $\alp=\alp_2$, $\beta=\alp_4$, $L_{\alp}$ is of type $A_5$, and $-w_0$ induces the non-trivial graph automorphism, which however fixes $\beta$.
\end{proof}

\begin{proof}[Proof of Proposition \ref{prop:RQ}]
We first note that  $St_{\alp}$ preserves the union of coordinate axis in $\fg^*_{-\alp}\oplus\fg^*_{-\delta}$, since this union is also the union of $\{0\}$ with the set of minimal elements in $\fg^*_{-\alp}\oplus\fg^*_{-\delta}$.
Since the action of $St_{\alp}$ on $\fg^*_{-\alp}\oplus\fg^*_{-\delta}$ is linear, any $g\in St_{\alp}$ either preserves the line $\fg^*_{-\alp}$ or sends all its elements to elements of $\fg^*_{-\delta}$. Thus, by Lemma \ref{lem:wLa}, exactly one of the elements $\{g, w_0g\}$ preserves both lines $\fg^*_{-\alp}$ and $\fg^*_{-\delta}$. By Lemma \ref{lem:LineStab}, $Q_{\alp}\cap \Gamma$ is the stabilizer of the line $\fg^*_{-\alp}$. By the same lemma applied to $\delta$, $R_{\alp}\cap \Gamma$ is the stabilizer of the line $\fg^*_{-\delta}$.
Thus, $RQ_{\alp}\cap \Gamma$ is the joint stabilizer of both lines, and has index two in $St_{\alp}$.
\end{proof}

}
\subsubsection{Preparation lemma}\label{subsubsec:prep} Assume $\fg$ is not of type $A_n$, and let $\alp$ be a quasi-abelian root. If $\fg$ is of type $D_n$ we assume further that $\alp$ is an abelian root. By Table \ref{tab:QA}, these assumptions imply that $\alp$ corresponds to an extreme node in the Dynkin diagram, {\it i.e.} there exists a unique simple root $\beta$ not orthogonal to $\alp$.
Thus $M_{\alp}=L_{\alp}\cap L_{\beta}$. 
Denote 
\begin{equation}
    \label{eq:Phi_a}
\Phi_{\alp}:=\{ \text{ root } \eps \mid \langle \eps, \alp \rangle = 0, \,\eps(S_\alp) = 2 \}\,.
\end{equation}

\begin{lemma}\label{lem:Fa}
The Weyl group of $M_{\alp}$  acts transitively on $\Phi_{\alp}$. 
\end{lemma}
\begin{proof}
 
By the defining property of minuscule representations, it is enough to show that $\Phi_\alpha$ corresponds to the set of weights of a minuscule representation of $M_\alpha$.
Note that for any root $\eps,\, \langle \alp,\eps\rangle = \eps(S_{\alp})-\eps(S_{\beta})/2$. Thus $\eps\in\Phi_{\alp}$ if and only if $\eps(S_{\beta})=4$. In other words, $\Phi_{\alp}$ is the set of roots of the $L_{\beta}$-module $\fg^{S_{\beta}}_{4}$, described in \cite[\S 5]{MS} where it is called  the second internal Chevalley module, therefore we have to show that the second internal Chevalley modules that arise are minuscule.

The second Chevalley module for the node $\beta$ is given by all roots of $\fg$ with coefficient $2$ along $\beta$. This can never happen for $\fg$ of type $A$, so the second Chevalley module is trivial. For types $D$ and $E$, and $\alpha$ an extreme node of the Dynkin diagram, not necessarily nice, the second Chevalley module for the adjacent $L_\beta$ is irreducible~\cite{MS}. This irreducible representation can be found uniformly by finding the lowest root $\theta$ of $\fg$ with coefficient $2$ along $\beta$.  This root $\theta$ is equal to the highest root of the smallest $D$-type diagram that can be embedded in the diagram of $\fg$ such that $\beta$ is the second node (in Bourbaki enumeration) of that $D$-type diagram. With this characterization $\theta$ is zero on torus elements $\gamma^\vee$ for all simple roots different from $\beta$ and the set of nodes $I$ directly attaching to the embedded $D$-type diagram. The root $\theta$ is $-1$ on the generators $\alpha_i^\vee$ ($i\in I$), thus making the restriction of $\theta$ a lowest weight of $M_\beta$. In particular, $\theta$ is trivial on $\alpha^\vee$ and by inspection one finds the following list of modules $\pi$ of $M_\alpha$ when $\alpha$ is nice. The same information is also illustrated in Table~\ref{tab:minuscule}.
\\[0.5\baselineskip]
Case $D_n$, $\alp=\alp_1$, $\beta=\alp_2$, $I=\emptyset$,  $\pi$=1-dimensional representation of $M_{\alpha}\cong D_{n-2}$.\\
Case $D_n$, $\alp=\alp_{n-1}$ (or $\alp=\alpha_n$), {$\beta=\alp_{n-2}$},  $I=\{n-4\}$, $\pi$=exterior square of the standard representation of $M_{\alpha}\cong A_{n-3}$.\\
{Case $E_6$, $\alp=\alp_2$, $\beta=\alp_4$, $I=\{1,6\}$, $\pi$=tensor product of the vector representation with the contragredient vector representation of $M_{\alpha}\cong A_2\times A_2$.}\\
Case $E_6$, $\alp=\alp_1\, (\alp_6)$, $\beta=\alp_3\,(\alp_5)$, $I=\{6\}\, (\{1\})$, $\pi$=standard representation of $M_{\alpha}\cong A_4$. \\
Case $E_7$, $\alp=\alp_1$, $\beta=\alp_3$, $I=\{6\}$, $\pi$=exterior square of $M_{\alpha}\cong A_5$.\\
Case $E_7$, $\alp=\alp_7$, $\beta=\alp_6$, $I=\{1\}$, $\pi$=standard representation of $M_{\alpha}\cong D_5$.\\
Case $E_8$: $\alp=\alp_8$, {$\beta=\alp_7$}, $I=\{1\}$, $\pi$=$27$-dimensional representation of $M_\alpha\cong E_6$.
\\[0.5\baselineskip]
All the modules listed are minuscule by~\cite[\S VIII.3]{Bou}. On the weights $\Phi_\alpha$ of such modules, the action of the Weyl group of $M_\alpha$ is transitive.
\end{proof}

\begin{table}[tpb]
    \centering
    \caption{Diagrammatic list of all Levi subgroups $M_\alpha$ and second internal Chevalley modules $\pi$ as a fundamental representation of $M_\alpha$ determined by a set $I$ of filled nodes. The extreme node $\alpha$ and its neighbouring node $\beta$ appear with a dotted pattern while $M_\alpha$ is obtained from the remaining, solid part of the diagram.}
    \label{tab:minuscule}
    \tikzstyle{dot}=[circle, draw, thick, fill=white, inner sep=2pt]
    \tikzstyle{line}=[thick, every node/.append style={dot}]
    \tikzstyle{hidden}=[densely dotted]
    \tikzstyle{marked}=[fill=black]
    \tikzstyle{fade}=[thick, dash pattern=on 7pt off 1pt on 1pt off 1pt on 1pt off 1pt on 1pt off 1pt, every node/.append style={dot, solid}]
    \begin{tabular}{ll@{\hskip 1em}l}
        Case $D_n$ &
        \begin{tikzpicture}[scale=0.75, baseline={(0,-.125)}]
            \draw[line, hidden] (0,0)  node {} -- (1,0) node {} -- (2,0);
            \draw[line] (2,0) node {} -- (3,0);
            \draw[fade] (3,0) node {} -- (4,0);
            \draw[line] (5,-0.5) node {} -- (4,0) node {} -- (5,0.5) node {};
        \end{tikzpicture} 
        &
        \begin{tikzpicture}[scale=0.75, baseline={(0,-.125)}]
            \draw[fade] (0,0) node {} -- (1,0);
            \draw[line] (1,0) node {} -- (2,0) node[marked] {} -- (3,0);
            \draw[line, hidden] (3,0) node[solid] {} -- (4,0);
            \draw[line, hidden] (5,0.5) node {} -- (4,0) node {} -- (5,-0.5) node[solid] {};
        \end{tikzpicture}
        \\
        Case $E_6$ &
        \begin{tikzpicture}[scale=0.75, baseline={(0,-.125)}]
            \draw[line, hidden] (1,0) -- (2,0);
            \draw[line, hidden] (3,0) -- (2,0) node {} -- (2,1) node {};
            \draw[line] (0,0) node[marked] {} -- (1,0) node{};
            \draw[line] (3,0) node {} -- (4,0) node[marked]{};
        \end{tikzpicture}
        &
        \begin{tikzpicture}[scale=0.75, baseline={(0,-.125)}]
            \draw[line, hidden] (0,0) node {} -- (1,0) node{} -- (2,0);
            \draw[line] (4,0) node[marked] {} -- (3,0) node{} -- (2,0) node {} -- (2,1) node {};
        \end{tikzpicture}
        \\
        Case $E_7$ &
        \begin{tikzpicture}[scale=0.75, baseline={(0,-.125)}]
            \draw[line, hidden] (0,0) node {} -- (1,0) node{} -- (2,0);
            \draw[line] (5,0) node {} -- (4,0) node[marked]{}-- (3,0) node{} -- (2,0) node {} -- (2,1) node {};
        \end{tikzpicture}
        &
        \begin{tikzpicture}[scale=0.75, baseline={(0,-.125)}]
            \draw[line] (0,0) node[marked] {} -- (1,0) node {} -- (2,0);
            \draw[line,hidden] (5,0) node {} -- (4,0) node{}-- (3,0);
            \draw[line] (3,0) node{} -- (2,0) node {} -- (2,1) node {};
        \end{tikzpicture} 
        \\
        Case $E_8$ &
        \begin{tikzpicture}[scale=0.75, baseline={(0,-.125)}]
            \draw[line] (0,0) node[marked] {} -- (1,0) node {} -- (2,0);
            \draw[line,hidden] (6,0) node {} -- (5,0) node{} -- (4,0);
            \draw[line] (4,0) node{} -- (3,0) node{} -- (2,0) node {} -- (2,1) node {};
        \end{tikzpicture}
                   &
    \end{tabular}
\end{table}

\subsubsection{Proof of Proposition \ref{prop:nice}}
\label{subsec:PfPropNTM}

{
Let $\alp$ be a nice root, {\it i.e.} an abelian root for any $\fg$, or a Heisenberg root in types $E_6,E_7,E_8$.  Let $X$ denote the set of next-to-minimal elements in $(\fg^*)^{S_{\alp}}_{-2}$.

\begin{lemma}\label{lem:An}
Assume $\fg$ is of type $A_n$, and $\alp=\alp_k$ in the Bourbaki enumeration with $k\notin \{1,n\}$.  Then the stabilizer of $\alp$ in the Weyl group of $L_{\alp}$ acts transitively on $\Phi_{\alp}$.
\end{lemma}
\begin{proof}
In the $\eps$ notation we have $\alp=\eps_k-\eps_{k+1}$, and $\Phi_{\alp}$ consists of all the roots $\eps_i-\eps_j$ with $i<k<k+1<j$. The stabilizer of $\alp$ in the Weyl group of $L_{\alp}$ permutes all $i<k$ and all $j>k+1$ independently.
\end{proof}

\begin{proof}[Proof of Proposition \ref{prop:nice}]
{\eqref{it:nice0}
If $\alp$ is abelian and $\langle \alp, \alp_{\max}\rangle> 0$, then $\fg$ is of type $A_n$, and $\alp$ is either $\alp_1$ or $\alp_n$ in the Bourbaki enumeration. In both cases, $(\fg^*)^{S_{\alp}}_{-2}\setminus \{0\}$ is given by trace pairing with rank one matrices, and thus has only minimal orbit and $X=\emptyset$.

\eqref{it:niceTran}
We have $\delta \in \Phi_{\alp}$ and Lemma \ref{lem:RootOrbs} implies that $X$ is non-empty. Further, by Lemma \ref{lem:AppGeo}, any $\varphi \in X$ can be conjugated by $L_{\alp}\cap \Gamma$ into $\fg^{\times}_{-\alp}+\fg^{\times}_{-\omega}$ for some $\omega\in \Phi_{\alp}$. By Lemmas \ref{lem:Fa} and \ref{lem:An}, we can assume $\omega=\delta$.}
\end{proof}
}

\subsubsection{Proof of Proposition \ref{prop:Fab} }
By the assumption of the proposition, $\alp$ is a nice Heisenberg root. 
In other words, $\alp$ is a Heisenberg root, and $\fg$ is of type $E_n$ for $n\in \{6,7,8\}$. 
Recall that 
\begin{equation}
\Psi_{\alp}=\{ \text{ root } \eps \mid \langle \eps, \alp \rangle \leq 0, \eps(S_\alp) = 2 \}\,.
\end{equation}
 and that $\Chi$ denotes the set of next-to-minimal elements in $\fg^{\times}_{-\alp}+ \bigoplus_{\eps\in \Psi_{\alp}}\fg^*_{-\eps}.$ Let $\alp_{\max}$ denote the maximal root of $\fg$.
Since $\alp$ is a Heisenberg root, $\langle \alp, \alp_{\max}\rangle =1$ and thus $\gamma:=\alp_{\max}-\alp$ is a root. 

\begin{lemma}\label{lem:ea}
\begin{enumerate}[(i)]
\item $ \Psi_{\alp}=\Phi_{\alp}\cup \{\gamma\}$. \label{it:Fab}
\item For any $\eps \in \Psi_{\alp}$, $\eps - \alp$ is not a root. \label{it:eaUR}
\end{enumerate}
\end{lemma}
\begin{proof}
\eqref{it:Fab} For any $\eps \in \Psi_{\alp} \setminus \Phi_{\alp}$, $\alp + \eps$ is a root and $(\alp+\eps)(S_{\alp})=4$. Since $\alp$ is a Heisenberg root, this implies $\alp+\eps =\alp_{\max}$.\\
\eqref{it:eaUR} $\langle \eps,-\alp\rangle \geq 0$  by definition of $\Psi_{\alp}$.
\end{proof}

As in \S\ref{subsubsec:prep}, let $\beta$ be the unique simple root  not orthogonal to $\alp$.
Note that  $\langle \beta, \gamma\rangle=-\langle \beta, \alp \rangle=1$ and thus $\delta:=\gamma - \beta$ is a root. 
\begin{lemma}\label{lem:lam}
Let $\lambda$ be a root with $\lambda(S_{\alp})=0$. Then 
\begin{enumerate}[(i)]
\item \label{it:albet} $\langle \lambda,\alp\rangle \cdot\langle \lambda, \beta \rangle \leq 0$.
\item \label{it:lambet} If $\langle \lambda,\alp\rangle \neq 0$ and $\delta+\lambda \in  \Psi_{\alp}$ then $\lambda=\beta$.
\end{enumerate}

\end{lemma}
\begin{proof}
\eqref{it:albet} Suppose the contrary. Then, $\lambda \notin \{\pm \alp, \pm\beta\}$. Also, replacing $\lambda$ by $-\lambda$ if needed, we may assume that $\langle \lambda,\alp\rangle =\langle \lambda, \beta \rangle =-1$. Thus $\lambda + \beta$ is a root and $\langle \alp, \lambda + \beta \rangle =-2$. Thus $\lambda+\beta = -\alp$. This contradicts $(\lambda+\beta)(S_{\alp})=0$.\\
\eqref{it:lambet} Since $\delta+\lambda \in  \Psi_{\alp}$,  $\langle \alp, \delta + \lam\rangle\leq 0$. But $\langle \alp,\delta\rangle =0$ and   $\langle \alp,\lam\rangle\neq 0$, thus $\langle \alp, \delta + \lam\rangle< 0$ and thus $\delta + \lam\in \Psi_{\alp} \setminus \Phi_{\alp}$. By Lemma \ref{lem:ea}\eqref{it:Fab}, this implies $\delta +\lam = \gamma$ and thus $\lam=\beta$.
\end{proof}

Recall that $\mathfrak{X}$ denotes the set of next-to-minimal elements in $\fg^{\times}_{-\alp}+ \bigoplus_{\eps\in \Psi_{\alp}}\fg^*_{-\eps}$.
{As before, for any root $\eps$ let $\eps^\vee$ denote the coroot given by the scalar product with $\eps$.}
 Note that $M_{\alp}\cap \Gamma$ preserves $\fX$, since $\Psi_{\alp}$ is the set of roots on which $S_{\alp}-{\alp}^\vee$ is at least 2 and $S_{\alp}$ is 2, and $M_{\alp}$ is the joint centralizer of ${\alp}^\vee$ and $S_{\alp}$. {For the rest of this section let
\begin{equation}
Z:={\beta}^\vee+2^{-1}S_{\alp}\,. 
\end{equation}
Note that} $\alp(Z)=\delta(Z)=0$.

\begin{lem}
\begin{enumerate}[(i)]
\item \label{it:epsZ} Let $\eps\neq \delta \in \Psi_{\alp}$. Then $\eps(Z)\in \{1,2\}$.
\item \label{it:minZ} The maximal eigenvalue of $Z$ on $\fg$ is $2$.
\end{enumerate}
\end{lem}
\begin{proof}
\eqref{it:epsZ} Suppose the contrary. Since $\eps(S_{\alp})=2$ and $\eps\neq \pm \beta$, this implies $\langle \beta,\eps\rangle =-1$, and thus $\eps+\beta$ is a root. Then $\langle \alp, \eps+\beta\rangle <0$, and by Lemma \ref{lem:ea}\eqref{it:Fab} $\eps+\beta=\gamma$. Thus $\eps=\gamma-\beta=\delta$, contradicting the assumption.\\
\eqref{it:minZ} We have to show that for any root $\mu$, $\mu(Z)\leq 2$. If $\mu =\alp_{\max}$ then $\mu(2^{-1}S_{\alp})=2$ and $\mu({\beta}^\vee)=0$. If $\mu=\beta$ then $\mu(2^{-1}S_{\alp})=0$ and $\mu({\beta}^\vee)=2$. For any other $\mu$, $\max(\mu(2^{-1}S_{\alp}),\mu({\beta}^\vee))\leq 1$.
\end{proof}

We are now ready to prove Proposition \ref{prop:Fab}. Let $x\in \fX$ and decompose it to a sum of root covectors $x=x_{\alp}+\sum_{\eps\in \Psi_{\alp}}x_{\eps}$ with $x_{\eps}\in \fg^*_{-\eps}$. Let $F:=\{\eps \in \Psi_{\alp}\, \vert x_{\eps}\neq 0\}$. By Lemma \ref{lem:RootOrbs} $F$ intersects $\Phi_{\alp}$ and thus, by Lemma \ref{lem:WeylMin}, we can assume $\delta\in F$. 
 Decompose $x=x_0+x_1+x_2$ with $x_i\in (\fg^*)^Z_{i}$. We have $x_0=x_{\alp}+x_{\delta}$.
Applying Lemma \ref{lem:SameOrbit3} to $S:=S_{\alp}$ and $Z$, we obtain that there exists a nilpotent $X\in (\fl_{\alp})^Z_{>0}$ with 
\begin{equation}\label{=X}
    \ad^*(X)(x_{0})=x_1+x_{2}\,.
\end{equation} 

Decompose $X$ to a sum of root vectors $X=\sum_{\lam\in \Psi} X_{\lam}$, $X_{\lam}\neq 0\in \fg_{-\lam}$, where $\Psi$ is some set of roots. 
Choose some $X\in (\fl_{\alp})^Z_{>0}$ satisfying \eqref{=X} such that the cardinality of $\Psi$ is minimal possible.
\begin{lemma}
$X\in \fm:=\textnormal{Lie}(M_{\alp})$.
\end{lemma}
\begin{proof}
Since $X\in(\fl_{\alp})^Z_{>0}$, we have $\langle \beta, \lam \rangle > 0$ {for any $\lambda\in \Psi$}. 
{Suppose by way of contradiction $X\notin \fm$. Then $\langle \alp, \lam \rangle \neq 0$ for some $\lam \in \Psi$. Fix such $\lambda$.}
Then Lemma \ref{lem:lam}\eqref{it:albet} implies $\langle \alp, \lam \rangle < 0$ and thus $\langle \alp, \lam \rangle =-1$ and thus {$\lambda+\alp$ is a root and} $[X_{\lam},x_{\alp}]\neq 0$. By Lemma \ref{lem:ea}\eqref{it:eaUR}, $\alp+\lam \notin\Psi_{\alp}$, and thus this term has to be canceled by $[X_{\mu}, x_{\delta}]$ for some $\mu \in \Psi$. Thus $\mu=\alp+\lam-\delta$ is a root and thus $\langle\alp+\lam,\delta\rangle=1$. But  this contradicts$$\langle\alp+\lam,\delta\rangle=\langle \lam, \delta\rangle=\langle\lam,\alp_{\max}-\alp-\beta\rangle=0+1-\langle \lam,\beta\rangle \leq 0\,.$$
\end{proof}

Thus $X\in \fm^Z_{>0}$. But $\fm^Z_{>0}=\fm^Z_{1}$. Thus $\ad^*(X)(x_0)\in (\fg^*)^Z_{1}$ and thus $x_2=0$ and $\ad^*(X)x_0=x_1$. 
Let 
\begin{equation}\label{=y}
y:=\Exp(-X)x-  x_{0}=-\ad^*(X)(x_1)+1/2(\ad^*(X))^2({x_0}).
\end{equation}
The right-hand side of \eqref{=y} has only {these} two terms because $X\in \fg^Z_{1}$,
$x\in \fg^Z_{\geq 0}$, and  $\fg=\fg^Z_{\leq 2}$.
Since $\ad^*(X)$ {raises} the $Z$-eigenvalues by 1, we get that
$y\in (\fg^*)^Z_{2}$.
Note that all the roots of $y$ still lie in $\Psi_{\alp}\setminus \{\delta\}$, since $X\in \fm$. Thus $x_{0}+y \in \fX$. By the same argument as above, there exists $Y\in \fm^Z_{1}$ such that $\ad^*(Y)(x_{0})=y$.
However, $\ad^*(Y)(x_{0})\in (\fg^*)^Z_{1}$ and thus $y=0$. 
Thus $\Exp(-X)x=x_0=x_{\alp}+x_{\delta},$ i.e.
 we can conjugate $x$ using $\Exp(-X)\in M_{\alp}\cap \Gamma$ into $\fg^{\times}_{-\alp}+ \fg^{\times}_{-\delta}$.  This proves  Proposition \ref{prop:Fab}.
\proofend
\begin{remark}\label{rmk:alpha2}
The assumption that  $G_{\alp}$ is not of type $D_n$  is necessary, since in type $D_n$ the Heisenberg root is $\alp_2$ and the set $\Phi_{\alp_2}\subset \Psi_{\alp_2}$ intersects both complex next-to-minimal orbits. Indeed, let $\lam:=\alp_1+\alp_2+\alp_3$ and $\mu:=\alp_2+2\sum_{i=3}^{n-2}\alp_i+\alp_{n-1}+\alp_n$. Then $\fg^{\times}_{-\alp}+\fg^{\times}_{-\lam}$ belongs to the orbit given by the partition $2^41^{n-4}$, and $\fg^{\times}_{-\alp}+\fg^{\times}_{-\mu}$ belongs to the orbit given by the partition $31^{n-3}$. To see this note that in the $\eps$ notation we have $\alp=\eps_2-\eps_3$, $\lam=\eps_1-\eps_4$ and $\mu=\eps_2+\eps_3$. 
\end{remark}

\section{Detailed examples} \label{sec:examples}
In this section we will illustrate how to use the framework introduced above to compute certain Fourier coefficients in detail, many of which are of particular interest in string theory. In particular, we will in \S\ref{sec:D5} show examples for $D_5$ with detailed steps and deformations that reproduce the results of Theorems~\ref{thm:min-rep}, \ref{thm:G0min} and \ref{thm:ntm-rep}, while in the following sections we will illustrate how to apply these theorems in different examples.

As in previous sections we will here often identify $\varphi \in \lie g^*$ with its Killing form dual $f_\varphi \in \lie g$. Since we have also seen that it is convenient to specify a Cartan element $S \in \lie h$ by how the simple roots $\alpha_i$ act on $S$ we will make use of the fundamental coweights $\omega_j^\vee \in \lie h$ satisfying $\alpha_i(\omega_j^\vee) = \delta_{ij}$. 

\subsection{\texorpdfstring{Examples for $D_5$}{Examples for D5}}\label{sec:D5}
In the following examples we will consider $G = \Spin_{5,5}(\A)$ with $\Gamma = \Spin_{5,5}(\K)$. 
We use the conventional Bourbaki labelling of the roots shown in Figure~\ref{fig:D5-labels}.
The complex nilpotent orbits for $D_5$ are labeled by certain integer partitions of $10$ with a partial ordering illustrated in the Hasse diagram of Figure~\ref{fig:D5-orbits} where $\Oh_{1^{10}}$ is the trivial orbit and $\Oh_{2^21^6}$ the minimal orbit.  Note that this ordering is based on the closure on complex orbits and not on the partial ordering that we introduced in~\cite{Part1}. There is no unique next-to-minimal orbit and both $\Oh_{2^41^2}$ and $\Oh_{31^7}$ {can} occur as {Whittaker supports of automorphic forms} arising in string theory. These two complex orbits are usually denoted $(2A_1)'$ and $(2A_1)''$ in Bala--Carter notation~\cite{CM} with $2A_1$ indicating two orthogonal simple roots and the primes distinguish the two possible pairs (up to Weyl conjugation, see Lemma~\ref{cor:2RootConjSimple}).

We will focus on examples of importance in string theory. In particular we consider expansions in the string perturbation limit associated to the maximal parabolic subgroup $P_{\alpha_1}$ and the decompactification limit associated to $P_{\alpha_5}$ discussed in section~\ref{sec:string}. The Fourier coefficients computed in \eqref{eq:D5-min-rank1} and \eqref{eq:D5-ntm-rank2} below have previously been computed for particular Eisenstein series in \cite{GMV} equations (4.84) and (4.88) respectively; although 
with very different methods using theta lifts. While the Fourier coefficient \eqref{eq:D5-min-rank1} for a minimal automorphic form is readily checked to be of the same form as \cite[(4.84)]{GMV}, the comparison between Fourier coefficient \eqref{eq:D5-ntm-rank2} for a next-to-minimal automorphic form and \cite[(4.88)]{GMV} is a bit more intricate and will be discussed further in Remark \ref{rmk:D5-ntm-rank-2} below.

\begin{figure}[ht]
    \begin{minipage}[t]{.5\textwidth}
\begin{center}
    \begin{tikzpicture}[scale=0.75]
        \tikzstyle{dot}=[circle, draw, thick, fill=white, inner sep=2pt]
        \tikzstyle{line}=[thick, shorten >=-2pt, shorten <=-2pt]
        \node (n1) at (0, 0) [dot, label=below:1] {};
        \node (n2) at (1, 0) [dot, label=below:2] {};
        \node (n3) at (2, 0) [dot, label=below:3] {};
        \node (n4) at (2, 1) [dot, label=above:5] {};
        \node (n5) at (3, 0) [dot, label=below:4] {};

        \begin{scope}[on background layer] 
            \draw [line] (n1) -- (n2) -- (n3) -- (n5);
                        \draw [line] (n3) -- (n4);
        \end{scope}

    \end{tikzpicture}
\end{center}
\caption{\label{fig:D5-labels}Root labels used for $D_5$.}
\end{minipage}%
\begin{minipage}[t]{.5\textwidth} 
\begin{center}
\begin{tikzpicture}[thick, scale=0.9]
    \node(A1) at (0,0) {$1^{10}$};
    \node(A2) at (0,1) {$2^21^6$};
    \node(A3) at (-1,2) {$2^41^2$};
    \node(A4) at (+1,2) {$31^7$};
    \node(A5) at (0,3) {$32^21^3$};
    \node(A6) at (0,4) {$3^21^4$};
    \node(A7) at (-1,5) {$3^22^2$};
    \node(A8) at (+1,5) {$51^5$};
    \node(A9) at (0,6) {$3^31$};
    \node(A10) at (-1,7) {$4^21^2$};
    \node(A11) at (+1,7) {$52^21$};
    \node(A12) at (0,8) {$531^2$};
    \node(A13) at (+1,9) {$5^2$};
    \node(A14) at (-1,9) {$71^3$};
    \node(A15) at (0,10) {$73$};
    \node(A16) at (0,11) {$91$};
    \foreach \x/\y in {A1/A2, A2/A3, A2/A4, A3/A5, A4/A5, A5/A6, A6/A7, A6/A8, A7/A9, A9/A10, A9/A11, A8/A11, A10/A12, A11/A12, A12/A13, A12/A14, A13/A15, A14/A15, A15/A16}
        \draw[->,>=stealth] (\x) to (\y);
\end{tikzpicture}
\end{center}
\caption{\label{fig:D5-orbits}Hasse diagram of nilpotent orbits for $D_5$ with respect to the closure ordering on complex orbits. There are two non-special orbits given by $32^21^3$ and $52^21$. }
\end{minipage}
\end{figure}

\newcommand{\DFive}[5]{{
    \begin{tikzpicture}[scale=0.1]
        \tikzstyle{w}=[circle, draw, fill=white, inner sep=1pt]
        \tikzstyle{b}=[circle, draw, fill=black, inner sep=1pt]
        \tikzstyle{g}=[circle, draw, fill=red, inner sep=1pt]
        \tikzstyle{line}=[thick, shorten >=-1pt, shorten <=-1pt]
        \node (n1) at (0, 0) [#1] {};
        \node (n2) at (1, 0) [#2] {};
        \node (n3) at (2, 0) [#3] {};
        \node (n5) at (3, 0) [#4] {};
        \node (n4) at (2, 1) [#5] {};

    \end{tikzpicture}
}}

\subsubsection{Minimal representation}
    We will start with considering a minimal automorphic function $\eta_\textnormal{min}$ on $G = \Spin_{5,5}(\A)$. Such a minimal automorphic form can for instance be obtained as a residue of a maximal parabolic Eisenstein series~\cite{GRS2,GMV,FGKP}. We will compute the Fourier coefficients of $\eta_\textnormal{min}$ with respect to the unipotent radical of the maximal parabolic subgroup $P_{\alpha_1}$ associated to the root $\alpha_1$, which is the string perturbation limit discussed in \S\ref{sec:string} and the corresponding Levi subgroup $L_{\alpha_1}$ has semisimple part of type $D_4$. 

    We may describe such Fourier coefficients by Whittaker pairs $(S_{\alpha_1}, \varphi)$ where $S_{\alpha_1} ={2} \omega_1^\vee$ and $\varphi \in \lie (\lie g^*)^{S_{\alpha_1}}_{-2}$. Indeed, the associated Fourier coefficient $\mathcal{F}_{S_{\alpha_1}, \varphi}$ is then the expected period integral over $N_{S_{\alpha_1}, \varphi} = U_{\alpha_1}$, the unipotent radical of $P_{\alpha_1}$, where we recall that $N_{S_{\alpha_1}, \varphi}$ is given by \eqref{eq:N_Sphi}. 
    \begin{equation}
        \mathcal{F}_{S_{\alpha_1}, \varphi}[\eta_\textnormal{min}](g) := \intl_{(U_{\alpha_1} \cap \Gamma) \bs U_{\alpha_1}} \eta_\textnormal{min}(ug) \varphi(u)^{-1} \, du \,.
    \end{equation}
    As in previous sections we will use the shorthand notation $[U] = (U\cap \Gamma) \bs U$ for the compact quotient of a unipotent subgroup $U$.
    
    Since $\eta_\textnormal{min}$ is minimal, Theorem~\ref{thm:min-rep}(\ref{it:min0}) gives that $\mathcal{F}_{S_{\alpha_1}, \varphi}[\eta_\textnormal{min}]$ is non-vanishing only if $\varphi \in \Oh_\textnormal{min} = \Oh_{2^21^6}$ or $\varphi = 0$. We will now consider the former. The latter can be computed using Theorem~\ref{thm:G0min} with $G$ of type $D_4$ or the results from \cite{MWCov} for Eisenstein series.
    
    By Corollary~\ref{cor:EasyMin}(\ref{it:Min1}), $\varphi \in \Oh_\textnormal{min}$ can be conjugated to $\varphi' = \Ad^*(\gamma_0) \varphi \in \lie g^\times_{-\alpha_1}$ by an element $\gamma_0 \in L_{\alpha_1} \cap \Gamma$. This conjugation leaves the integration domain invariant, or, equivalently, we may use Lemma~\ref{lem:conjugation-translation} to obtain 
    \begin{equation}
        \mathcal{F}_{S_{\alpha_1}, \varphi}[\eta_\textnormal{min}](g) = \mathcal{F}_{S_{\alpha_1}, \varphi'}[\eta_\textnormal{min}](\gamma_0 g) \, .
    \end{equation}

    The unipotent radical $U_{\alpha_1}$ is a subgroup of the unipotent radical $N$ of our fixed Borel subgroup, and we may make further Fourier expansions along the complement of $U_{\alpha_1}$ in $N$. Of these Fourier coefficients, only the constant term survives since such non-trivial characters, combined with $\varphi'$ are in a larger orbit than $\Oh_\textnormal{min}$ and therefore do not contribute according to 
    Corollary~\ref{cor:domin}. By repeating these arguments, or equivalently use Lemma~\ref{lem:1_1} based on a special case of 
    Theorem~\ref{thm:IntTrans} (where $V$ is trivial), we obtain that
    \begin{equation}
        \label{eq:D5-min-rank1}
        \mathcal{F}_{S_{\alpha_1}, \varphi}[\eta_\textnormal{min}](g) = \mathcal{W}_{\varphi'}[\eta_\textnormal{min}](\gamma_0 g) :=  \intl_{(N \cap \Gamma) \bs N} \eta_\textnormal{min}(n \gamma_0 g) \varphi'(n)^{-1} \, dn \, ,
    \end{equation}
{confirming Theorem \ref{thm:min-rep}\eqref{it:min-min} for this case.}

\subsubsection{Next-to-minimal representations}

    Let $\eta_\textnormal{ntm}$ be a next-to-minimal automorphic form on $G = \Spin_{5,5}(\A)$. {Since there are two next-to-minimal orbits for $D_5$ there are two cases to consider. We begin with automorphic forms associated with the next-to-minimal orbit $\WS(\eta_\textnormal{ntm}) = \{ \Oh_{31^7} \}$ that has dimension $16$, also known as $(2A_1)'$ in Bala-Carter notation.}  Let also $P_{\alpha_1} = L_{\alpha_1} U_{\alpha_1}$ be the maximal parabolic subgroup of $G$ with respect to the simple root $\alpha_1$ such that the Levi subgroup $L_{\alpha_1}$ has semisimple part of type $D_4$. Automorphic forms with the above Whittaker support can for example be obtained as generic elements of the degenerate principal series of maximal parabolic Eisenstein series associated with $P_{\alpha_1}$. 
    
    We will now compute the Fourier coefficients of $\eta_\textnormal{ntm}$ with respect to $U_{\alpha_1}$ using Theorem~\ref{thm:ntm-rep}. These are described by Whittaker pairs $(S_{\alpha_1},\varphi)$ where $S_{\alpha_1} ={2} \omega_1^\vee$ and $\varphi \in (\lie g^*)^{S_{\alpha_1}}_{-2}$. The case $\varphi = 0$ can be treated using Theorem~\ref{thm:ntm-rep2} with $G$ of type $D_4$. According to Theorem~\ref{thm:ntm-rep} or Corollary~\ref{cor:domin}, we are thus left with $\varphi$ being minimal or next-to-minimal where the latter in this case only gives non-vanishing Fourier coefficients for $\varphi \in \Oh_{31^7}$ and not $\Oh_{2^41^2}$.

    A minimal element $\varphi = \varphi_\textnormal{min} \in \Oh_\textnormal{min} = \Oh_{2^21^6}$ can be conjugated to some standard form $\psi = \Ad^*(\gamma_\textnormal{min})\varphi_\textnormal{min} \in \lie g^\times_{-\alpha_1}$ where $\gamma_\textnormal{min} \in L_{\alpha_1} \cap \Gamma$ using Corollary~\ref{cor:EasyMin}(\ref{it:Min1}). From Lemma~\ref{lem:conjugation-translation} we then have that
    \begin{equation}
        \cF_{S_{\alpha_1}, \varphi_\textnormal{min}}[\eta_\textnormal{ntm}](g) = \cF_{S_{\alpha_1}, \psi}[\eta_\textnormal{ntm}](\gamma_\textnormal{min}g)\,.
    \end{equation}

        Let $I^{(\perp \alpha_1)} = (\beta_1, \beta_2, \beta_3) := (\alpha_5,\alpha_4,\alpha_3)$ and $L_{i}$ be the Levi subgroup of $G$ obtained from a subsequence of simple roots $(\beta_1, \ldots, \beta_{i})$ of $I^{(\perp \alpha_1)}$. Each semisimple part of $L_i$ has simple components of type $A$ for which all simple roots are abelian according to Table \ref{tab:QA}, and thus $I^{(\perp \alpha_1)}$ is an abelian enumeration.
    Using Therorem~\ref{thm:ntm-rep}(\ref{it:Pt2}) we obtain
    \begin{equation}
        \label{eq:D5-expansion}
        \cF_{S_{\alpha_1}, \varphi_\textnormal{min}}[\eta_\textnormal{ntm}](g) = \cW_{\psi}[\eta_\textnormal{ntm}](\gamma_\textnormal{min}g) + \sum_{i=1}^3 C^\psi_i[\eta_\textnormal{ntm}](\gamma_\textnormal{min}g) 
    \end{equation}
    where
    \begin{equation}
        C^\psi_i[\eta_\textnormal{ntm}](\gamma_\textnormal{min}g) = A^\psi_i[\eta_\textnormal{ntm}](\gamma_\textnormal{min}g) = \sum_{\gamma \in \Gamma_{i-1}} \sum_{\varphi' \in \lie g^\times_{-\beta_i}} \cW_{\psi + \varphi'}[\eta_\textnormal{ntm}](\gamma \gamma_\textnormal{min} g) \, .
    \end{equation}
    As explained in Section~
    \ref{subsec:ThB}, $\Gamma_{i-1}$ is defined as follows. {Let $Q_{i-1}$ denote the parabolic subgroup of $L_{i-1}$ given by the restriction of $\beta_i^{\vee}$ to  $L_{i-1}$. Then $Q_{i-1}$ is the stabilizer in $L_{i-1}$ of the root space $\lie g^*_{-\beta_i}$ of $L_i$. Then $\Gamma_{i-1} = (L_{i-1} \cap \Gamma) / (Q_{i-1} \cap \Gamma)$ with $\Gamma_0 = \{1\}$.}
    Concretely, we may take the representatives 
    \begin{equation}
        \Gamma_0 = \Gamma_1 = \{1\} \qquad \Gamma_2 = \{1\} \cup w_4 \Exp(\lie g_{-\alpha_4}) \cup w_5 \Exp(\lie g_{-\alpha_5}) \cup w_4 w_5 \Exp(\lie g_{-\alpha_4} \oplus \lie g_{-\alpha_5}) 
    \label{eq:GD5}
    \end{equation}
    where $w_i$ is a representative in $\Gamma$ of the simple reflection corresponding to the simple root $\alpha_i$. The last equality in~\eqref{eq:GD5} is the Bruhat decomposition of $\Gamma_2$ and is isomorphic to $\mathbb{P}^1(\K) \times \mathbb{P}^1(\K)$.

    Let us now consider next-to-minimal characters $\varphi = \varphi_\textnormal{ntm} \in (\lie g^*)^{S_{\alpha_1}}_{-2}$ instead.
    By Proposition~\ref{prop:nice},  { $\varphi_\textnormal{ntm}$ can be conjugated using $L_\alpha \cap \Gamma$ into $\lie g^\times_{-\alpha_1} + \lie g^\times_{-\alpha_\text{max}}$}. In fact, $\varphi_\textnormal{ntm} \in \Oh_{31^7}$ since $\lie g^\times_{-\alpha_1} + \lie g^\times_{-\alpha_\text{max}}$ can be Weyl reflected to $\lie g^\times_{-\alpha_4} + \lie g^\times_{-\alpha_5}$ which are known to be in $\Oh_{31^7}$.
    Indeed, {by Corollary~\ref{cor:RootPairClasses}  there} is a Weyl word $w$ that moves the roots $\alpha_1$ and $\alp_\text{max}$ to two orthogonal simple roots, and from the proof of the Corollary we have that these roots have to be $\alpha_4$ and $\alpha_5$.
    
    Lemma~\ref{lem:conjugation-translation} together with Theorem~\ref{thm:ntm-rep}\eqref{itm:ntm} for any of these choices give
    \begin{equation}
        \label{eq:D5-ntm-rank2}
        \begin{split}
            \cF_{S_{\alpha_1}, \varphi_\textnormal{ntm}}[\eta_\textnormal{ntm}](g) &= \cF_{S_{\alpha_1}, \Ad^*(\gamma_\textnormal{ntm})\varphi_\textnormal{ntm}}[\eta_\textnormal{ntm}](\gamma_\textnormal{ntm} g) \\
            &= \intl_{V} \cW_{\Ad^*(w\gamma_\textnormal{ntm})\varphi_\textnormal{ntm}}[\eta_\textnormal{ntm}](v w \gamma_\textnormal{ntm} g) \, dv
        \end{split}
    \end{equation}
    with $V = \Exp(\lie v)(\A)$ where $\lie v = \lie g_{-\alpha_{3}} \oplus \lie g_{-\alpha_{2} - \alpha_{3}} \oplus \lie g_{-\alpha_{1} - \alpha_{2} - \alpha_{3}}$.
    
    \begin{remark}
        \label{rmk:D5-ntm-rank-2}
        We may now revisit the comparison between \eqref{eq:D5-ntm-rank2} and the Fourier coefficient \cite[(4.88)]{GMV} for a particular Eisenstein series. The latter is expressed in of double divisor sums and a single Bessel function. Specifying to the same Eisenstein series in \eqref{eq:D5-ntm-rank2}, the Whittaker coefficient on the right-hand side resolves to a product of two (single) divisor sums and two Bessel functions (see for example \cite{FGKP}). We expect that the non-compact adelic integral in \eqref{eq:D5-ntm-rank2} will allow us to relate the two expressions, something that will require further investigation.
    \end{remark}
    
    Lastly, we will consider the other next-to-minimal orbit $\Oh_{2^41^2}$ of dimension 20 and Bala-Carter label $(2A_1)''$. That is, consider $\eta_\text{ntm}$ such that $\operatorname{WS}(\eta_\text{ntm}) = \{ \Oh_{2^41^2}\}$. Such an automorphic form can, for example, be obtained as generic elements of the degenerate principal series of maximal parabolic Eisenstein series associated with $P_{\alpha_4}$ or $P_{\alpha_5}$. We showed above that all the next-to-minimal elements in $(\lie g^*)^{S_{\alpha_1}}_{-2}$ are in $\Oh_{31^7}$ and thus the corresponding next-to-minimal Fourier coefficients $\cF_{S_{\alpha_1},\varphi}[\eta_\text{ntm}]$ would vanish.

    Therefore, we will here consider another parabolic subgroup $P_{\alpha_5} = L_{\alpha_5} U_{\alpha_5}$ associated with the root $\alpha_5$ such that $L_{\alpha_5}$ has semi-simple part of type $A_4$. Let $S_{\alpha_5} = {2}\omega^\vee_5$ and $\varphi_\text{ntm}$ a next-to-minimal element in $(\lie g^*)^{S_{\alpha_5}}_{-2}$. {By Proposition~\ref{prop:nice},  there exists} $\gamma_\text{ntm} \in L_{\alpha_5} \cap \Gamma$ such that {$\Ad^*(\gamma_\text{ntm}) \varphi_\text{ntm} \in \lie g^\times_{-\alpha_5} + \lie g^\times_{-\alpha_\text{max}}$}.

    Furthermore, by Corollary~\ref{cor:2RootConjSimple} there exists a Weyl word $w_{ij}$, and simple roots $\alpha_i$ and $\alpha_j$ such that $\Ad^*(w_{ij}\gamma_\text{ntm}) \varphi_\text{ntm} \in \lie g^\times_{-\alpha_i} + \lie g^\times_{-\alpha_j}$ with the possible choices listed in \eqref{eq:Vij} below, up to interchanging the two roots. For any (and therefore all) such choices of simple roots $\alpha_i$ and $\alpha_j$ it is known that $\lie g^\times_{-\alpha_i} + \lie g^\times_{-\alpha_j} \subset \Oh_{2^41^2}$ and thus $\varphi_\text{ntm} \in \Oh_{2^41^2}$.

    For any of the choices, Lemma~\ref{lem:conjugation-translation} together with Theorem~\ref{thm:ntm-rep}\eqref{itm:ntm} gives
    \begin{equation}
        \begin{split}
            \cF_{S_{\alpha_5}, \varphi_\textnormal{ntm}}[\eta_\textnormal{ntm}](g) &= \cF_{S_{\alpha_5}, \Ad^*(\gamma_\textnormal{ntm})\varphi_\textnormal{ntm}}[\eta_\textnormal{ntm}](\gamma_\textnormal{ntm} g)\\ 
            &= \intl_{V_{ij}} \cW_{\Ad^*(w_{ij}\gamma_\textnormal{ntm})\varphi_\textnormal{ntm}}[\eta_\textnormal{ntm}](v w_{ij} \gamma_\textnormal{ntm} g) \, dv
        \end{split}
    \end{equation}
    where $V_{ij} = \Exp(\lie v_{ij})(\A)$ and $\lie v_{ij} = \lie v_{ji}$ can be read from the following table using the notation $\alpha_{m_1m_2m_3m_4m_5} = \sum_{i=1}^5 m_i \alpha_i$.
    \begin{equation}
        \label{eq:Vij}
        \begin{tabular}{lll} \toprule
            $\alpha_i$ & $\alpha_j$ \,\,\, & $\lie v_{ij} = \lie v_{ji}$ \\ \midrule
            $\alpha_{1}$ & $\alpha_{3}$ & $\lie g_{-\alpha_{00010}} \oplus \lie g_{-\alpha_{01000}} \oplus \lie g_{-\alpha_{01110}} \oplus \lie g_{-\alpha_{01111}}$ \\
            $\alpha_{1}$ & $\alpha_{4}$ & $\lie g_{-\alpha_{01000}} \oplus \lie g_{-\alpha_{01100}} \oplus \lie g_{-\alpha_{01101}}$ \\
            $\alpha_{1}$ & $\alpha_{5}$ & $\lie g_{-\alpha_{00100}} \oplus \lie g_{-\alpha_{00110}} \oplus \lie g_{-\alpha_{01100}} \oplus \lie g_{-\alpha_{01110}} \oplus \lie g_{-\alpha_{01211}}$ \\
            $\alpha_{2}$ & $\alpha_{4}$ & $\lie g_{-\alpha_{00100}} \oplus \lie g_{-\alpha_{00101}}$ \\
            $\alpha_{2}$ & $\alpha_{5}$ & $\lie g_{-\alpha_{00100}} \oplus \lie g_{-\alpha_{10000}} \oplus \lie g_{-\alpha_{00110}} \oplus \lie g_{-\alpha_{11100}} \oplus \lie g_{-\alpha_{11110}} \oplus \lie g_{-\alpha_{11211}}$ \\ \bottomrule
        \end{tabular} 
    \end{equation}

    As one can see from the above table, the size of $V$ depends strongly on the choice of representative roots. The smallest choice is obtained in the fourth row.

\subsection{An \texorpdfstring{$E_8$-example}{E8-example}}
\label{sec:E8ex}
In this section we will illustrate our general results in the context of automorphic forms on $E_8$. We will give the complete Fourier expansion in the minimal and next-to-minimal representations along a Heisenberg parabolic subgroup{, see {Proposition \ref{prop:Heis}} 
for a general discussion of such expansions}. We also discuss relations with related results in the literature. 

\subsubsection{The explicit Fourier expansions of \texorpdfstring{$\eta_{\textnormal{min}}$}{eta_min} and \texorpdfstring{$\eta_{\textnormal{ntm}}$}{eta_ntm}}

We will now illustrate Theorems~\ref{thm:G0min},~\ref{thm:ntmFull} and \ref{thm:ntmFullSimple} in the case of $E_8$. According to theorems~\ref{thm:G0min} and ~\ref{thm:ntmFull} the general structure of the expansions of automorphic forms $\eta_{\textnormal{min}}$ and $\eta_{\textnormal{ntm}}$ attached to the minimal and next-to-minimal representation of $E_8$ are given by 
\begin{eqnarray}
\eta_{\textnormal{min}}&=&\mathcal{F}_{S_\alpha , 0}[\eta_{\textnormal{min}}] + A_n +B_n,
\\
\eta_{\textnormal{ntm}}&=&\mathcal{F}_{S_\alpha , 0}[\eta_{\textnormal{ntm}}] + A_n +A_{nn}+ \sum_{\substack{j<n \\ j\bot n}}A_{nj}+B_n+B_{nn} +\sum_{\substack{j<n \\ j\bot n}}B_{nj},
\end{eqnarray}
where the notation and the definitions of the individual terms are given in sections \ref{subsec:ThB} and \ref{subsec:EFG}. 

To illustrate this more explicitly we now pick the Bourbaki enumeration as in {Theorem \ref{thm:ntmFullSimple}} that is quasi-abelian for $E_8$. Let $P=LU$ be the Heisenberg parabolic of $E_8$, with Levi $L=E_7\times GL_1$ and unipotent $U$ a 57-dimensional Heisenberg group with one-dimensional center $C=[U,U]$. This corresponds to expanding with respect to the Heisenberg root $\alpha=\alpha_8$. In its full glory the  expansion now amounts to the following expression in the minimal case
\begin{equation}
\label{eq:e8minExplicit}
\eta_{\textnormal{min}}(g)=\mathcal{F}_{S_{\alpha_8} , 0}[\eta_{\textnormal{min}}] (g)+\sum_{\gamma\in\Gamma_7}\sum_{\varphi\in\mathfrak{g}_{-\alpha_8}^\times}\mathcal{W}_{\varphi}[\eta_{\textnormal{min}}](\gamma g)+\sum_{\omega\in \Omega_8} \sum_{\varphi\in \mathfrak{g}_{-\alpha_8}^{\times}}\mathcal{W}_\varphi[\eta_{\textnormal{min}}](\omega \gamma_8g)\,,
\end{equation}
and for the next-to-minimal representation we have a slightly more complicated expression
\begin{align}
\label{eq:e8ntmExplicit}
\eta_{\textnormal{ntm}}(g) &=\mathcal{F}_{S_{\alpha_8}, 0}(g)+\underbrace{\sum_{\gamma\in\Gamma_7}\sum_{\varphi\in\mathfrak{g}_{-\alpha_8}^\times}\mathcal{W}_{\varphi}(\gamma g)}_{A_8}
+\sum_{j=1}^{6} \underbrace{\sum_{\gamma'\in \Gamma_7}\sum_{\varphi\in \mathfrak{g}_{-\alpha_8}^{\times}} \sum_{\gamma \in \Gamma_{j-1}}\sum_{\psi\in \mathfrak{g}_{-\alpha_j}^{\times}}\mathcal{W}_{\varphi+\psi}(\gamma \gamma' g)}_{A_{8j}} \nonumber
\\  \nonumber
&\quad+ \underbrace{\frac{1}{2}\sum_{\tilde{\gamma}\in \Lambda_{\alpha_8}}\sum_{\varphi\in \mathfrak{g}_{-\alpha_8}^{\times}}\sum_{\psi\in \mathfrak{g}_{-\delta_8}^{\times}}\int_{V_{g_8}} \mathcal{W}_{{\Ad^*(g_8)(\varphi+\psi)}}(vg_8\tilde{\gamma}g)dv}_{A_{88}}
+\underbrace{\sum_{\omega\in \Omega_8} \sum_{\varphi\in \mathfrak{g}_{-\alpha_8}^{\times}}\mathcal{W}_\varphi(\omega \gamma_8g)}_{B_8}
\\ \nonumber
&\quad+ \underbrace{\sum_{\omega\in \Omega_8} \sum_{\tilde{\gamma}\in \mathcal{M}_{\alpha_8}}\sum_{\varphi\in \mathfrak{g}_{-\alpha_8}^{\times}}\sum_{\psi\in \mathfrak{g}_{-\delta_8}^{\times}}\int_{V_{g_8}} \mathcal{W}_{{\Ad^*(g_8)(\varphi+\psi)}}(v{g_8}\tilde{\gamma}\omega \gamma_{8} g)dv}_{B_{88}}
\\ 
& \quad +\sum_{j=1}^6\underbrace{\sum_{\omega\in \Omega_8}\sum_{\varphi\in \mathfrak{g}_{-\alpha_8}^{\times}}\sum_{\gamma \in \Gamma'_{j-1}}\sum_{\psi\in \mathfrak{g}_{-\alpha_j}^\times}\mathcal{W}_{\varphi+\psi}(\gamma \omega \gamma_{8} g)}_{B_{8j}}\,,
\end{align}
where all coefficients are evaluated for the automorphic form $\eta=\eta_{\textnormal{ntm}}$. The elements $g_8$ and $\gamma_8$ are defined in \S\ref{subsec:EFG} and \S\ref{subsec:ThB}, respectively,

As discussed in section~\ref{intro} the expansion can be separated into an abelian contribution and a non-abelian contribution. The form of the expansion given above reflects this structure, as we now explain in more detail. We focus on the next-to-minimal case as this is the more complicated case. 
  
 Let $\psi_U$ be a unitary character on $U(\mathbb{A})$, trivial on $U(\K)$. It is supported only on the abelianization $U^{\text{ab}}=C\backslash U$. The \emph{abelian contribution} to the Fourier expansion is then given by the constant term with respect to the center of the Heisenberg group
\begin{equation}
\int_{C(\K)\backslash C(\mathbb{A})}\eta_\textnormal{ntm}(zg)dz
\end{equation} 
which can be expanded into a Fourier sum of the form $\sum_{\psi_U}$ where we sum over all characters $\psi_U$.  The first term in the expansion $\mathcal{F}_{S_{\alp_8},0}[\eta_\textnormal{ntm}](g)$ is the constant term of $\eta_\textnormal{ntm}$ with respect to $U$, i.e. corresponding to the contribution with trivial character $\psi_U$. The abelian part, corresponding to terms labelled $A$, of the non-trivial Fourier coefficients is made up of the second, third and fourth terms on the right-hand side of equation~\eqref{eq:e8ntmExplicit}. The first of these is attached to the minimal orbit $\mathcal{O}_{\textnormal{min}}$ while the last two are attached to $\mathcal{O}_{\textnormal{ntm}}$. These coefficients are not sufficient to recreate the entire automorphic form $\eta_\textnormal{ntm}$; we also need to consider the contributions from non-trivial characters on the center $C$. Let $\psi_C$ be a \emph{non-trivial} character on $C(\mathbb{A})$, trivial on $C(\K)$. The non-abelian contribution to the Fourier expansion is then given schematically by 
\begin{equation}
\sum_{\psi_C} \int_{C(\K)\backslash C(\mathbb{A})} \eta_\textnormal{ntm}(zg)\psi_C(z)^{-1} dz.
\end{equation}
This makes up the remaining three terms in equation~(\ref{eq:e8ntmExplicit}), corresponding to terms labelled $B$. {We note that the non-abelian terms contain the transformation $\gamma_8$ mapping $\alpha_{\textnormal{max}}$ to $\alpha_8$, signalling the fact they come originally from a non-trivial character on the center of the Heisenberg group.} The first one represents the contribution from $\mathcal{O}_{\textnormal{min}}$, while the last two (bottom line) capture the contribution from $\mathcal{O}_{\textnormal{ntm}}$.
\subsubsection{Comparison with related results in the literature}\label{subsec:compar}
Various works have determined similar Fourier coefficients of small representations in special cases and we now briefly compare our results to them, with a particular emphasis on the $E_8$ expansions. 

We begin with the example of a minimal automorphic form $\eta$ on $E_8$ with the expansion determined in~\eqref{eq:e8minExplicit}, that was also studied by Ginzburg--Rallis--Soudry \cite{GRS} and by Kazhdan--Polishchuk \cite{KazhdanPolishchuk}. 

In \cite{GRS}, Ginzburg--Rallis--Soudry showed that the constant term of $\eta_{\textnormal{min}}$ with respect to the center $C$ of the Heisenberg unipotent $U$ of $E_8$ was given by a single Levi (i.e. $E_7$) orbit of a Fourier coefficient $\mathcal{F}_{\psi_{\alpha_8}}$ on $U$, where $\psi_{\alpha_8}$ is a character on $U$ supported only on the single simple root $\alpha_8$. This corresponds precisely to the second term in~\eqref{eq:e8minExplicit}. Our results generalize this by also determining  $\mathcal{F}_{\psi_{\alpha_8}}$ explicitly in terms of  Whittaker coefficients $\mathcal{W}_\varphi[\eta_{\textnormal{min}}]$. 

In \cite{KazhdanPolishchuk}, the authors give an explicit form of the full non-abelian Fourier expansion of $\eta$ with respect to $U$ and our result (\ref{eq:e8minExplicit}) is perfectly consistent with theirs. Kazhdan and Polishchuk have, however, a different approach, where they first determine the local contributions (spherical vectors) to the Fourier coefficients and then assemble them together into a global automorphic functional. To connect the two results one must therefore evaluate the Whittaker coefficients in~(\ref{eq:e8minExplicit}) and extract their contributions at each local place. For the abelian terms, this has in fact already been done in \cite{GKP} and by combining those results with ours one achieves perfect agreement with \cite{KazhdanPolishchuk}. It remains to evaluate explicitly the Whittaker coefficient in the last term of equation~(\ref{eq:e8minExplicit}), corresponding to $B_8$, and extract its Euler product. It would be of particular interest to see if one can reproduce the cubic phase in the spherical vectors of \cite{KazhdanPolishchuk} in this way.

Next we turn to the Fourier expansion of an $E_8$ automorphic form in the next-to-minimal representation given in~\eqref{eq:e8ntmExplicit} that has been studied previously by Bossard--Pioline~\cite{Bossard:2016hgy}. According to the discussion in \S\ref{sec:string} the decomposition in~\eqref{eq:e8ntmExplicit} corresponds to the decompactification limit and an expression for the abelian part of the Fourier expansion for the next-to-minimal spherical Eisenstein series on $E_8$ was given in~\cite[Eq.~(3.15)]{Bossard:2016hgy}  that we reproduce here for convenience
\begin{align}
\label{eq:E8dec}
\eta&= \mathcal{F}_{S_{\alp},0}[\eta] + 16\pi \xi(4) R^4 \sum_{\substack{\Gamma \in \mathcal{L}_\alpha \\ \Gamma\times\Gamma=0}} \sigma_{8}(\Gamma)  \frac{K_4(2\pi  R |Z(\Gamma)|)}{|Z(\Gamma)|^{4}} e^{2\pi i \langle \Gamma, a\rangle}\nonumber\\
    &\quad + 16\pi \xi(3) R  \sum_{\substack{\Gamma \in \mathcal{L}_\alpha \\ \Gamma\times\Gamma=0}} \sigma_2(\Gamma) (\gcd \Gamma)^2\eta^{E_6}_{\textnormal{min}} \frac{K_1(2\pi R |Z(\Gamma)|)}{|Z(\Gamma)|^3} e^{2\pi i \langle \Gamma, a\rangle}\\
    &\quad + 16 \pi R^{-5}  \!\!\!\!\!\! \sum_{\substack{\Gamma \in \mathcal{L}_\alpha \\ \Gamma\times\Gamma\neq 0, \, I_4'(\Gamma)=0}} \sum_{n|\Gamma}  n^{d+1}  \sigma_{3}\big({\tfrac{\Gamma\times\Gamma}{n^2}}\big)  \frac{B_{5/2,3/2}(R^2 |Z(\Gamma)|^2, R^2 \sqrt{\Delta(\Gamma)})}{\Delta(\Gamma)^{3/4}} e^{2\pi i \langle \Gamma, a\rangle} + \ldots\nonumber\,.
\end{align}
Here, explicit coordinates on $E_8/(\operatorname{Spin}_{16}/\ints_2)$ adapted to the $E_7$ parabolic are used. Specifically, $R$ is a coordinate for the $\GL_1$ factor in the Levi and $a$ denotes (axionic) coordinates on the $56$-dimensional abelian part of the unipotent. $\mathcal{L}_\alpha$ is a lattice in this $56$-dimensional representation of $E_7$ and the coordinates on the $E_7$ factor of the Levi enter implicitly through the functions $Z(\Gamma)$ and $\Delta(\Gamma)$. We do not require their precise form for the present comparison. $K_s$ denotes the modified Bessel function and $\eta^{E_6}_{\textnormal{min}}$ a spherical vector in the minimal representation of $E_6$.

We now establish that~\eqref{eq:E8dec} and~\eqref{eq:e8ntmExplicit} are compatible. The Fourier expansion in~\eqref{eq:E8dec} is written in terms of sums over charges $\Gamma$ in the integral lattice $\mathcal{L}_\alpha$ in the $56$-dimensional unipotent and thus resembles structurally~\eqref{eq:e8ntmExplicit} above as the space $(\mathfrak{g^*})^{S_{\alp}}_{-2}$ represents the space of characters on this unipotent. The Fourier mode for a `charge' $\Gamma$ is given by $e^{2\pi i \langle \Gamma,a\rangle}$ and is the character on $(\mathfrak{g})^{S_{\alp}}_{2}$. Besides the constant term $\mathcal{F}_{S_{\alp},0}[\eta]$ there is a sum over characters in the minimal and next-to-minimal orbits within $(\mathfrak{g^*})^{S_{\alp}}_{-2}$; the last term in our~\eqref{eq:e8ntmExplicit} is a non-abelian term that was not determined in~\cite{Bossard:2016hgy}. 

Minimal characters correspond to charges $\Gamma$ such that they satisfy the (rank-one) condition $\Gamma\times\Gamma=0$ in the notation of~\cite{Bossard:2016hgy} and looking at~\eqref{eq:E8dec} we see that there are two contributions from such charges. These correspond exactly to the two terms $A_{8}$ and $A_{8j}$ in the first line of our~\eqref{eq:e8ntmExplicit}: The first term $A_8$ represents the purely minimal charges while the second term $A_{8j}$ in our equation is the second line of~\eqref{eq:E8dec} where a minimal charge is combined with a minimal automorphic form on $E_6$. Expanding this minimal automorphic form on $E_6$ leads to Whittaker coefficients of the form $\mathcal{W}_{\varphi+\psi}$ as they are given in the third term of the of the first line in~\eqref{eq:e8ntmExplicit}, i.e. corresponding to $A_{8j}$. The sums over $j$, $\Gamma_{j-1}$ and $\mathfrak{g}_{-\beta_j}^{\times}$ in our expression correspond to the $E_7$ orbits of such charges $\Gamma$. The term $A_{88}$ in our formula~\eqref{eq:e8ntmExplicit} contains a non-compact integral over Whittaker coefficient $\mathcal{W}_{\varphi+\psi}$ and corresponds to the last line in~\eqref{eq:E8dec} where a similar integrated Whittaker coefficient $B_{5/2,3/2}$ appears. The non-abelian terms with $B$-labels in the last line of~\eqref{eq:e8ntmExplicit} have not been determined in~\cite{Bossard:2016hgy} and are given by the ellipses in~\eqref{eq:E8dec}.

\newcommand{\etalchar}[1]{$^{#1}$}
\def\cprime{$'$}
\providecommand{\href}[2]{#2}\begingroup\raggedright

\endgroup

\end{document}